\documentclass[a4paper,10pt]{amsart}

\usepackage[utf8]{inputenc} 
\usepackage[T1]{fontenc}
\usepackage{mathtools,amssymb,amsthm} 
\usepackage{dsfont}
\usepackage[all]{xy}
\usepackage{verbatim} 
\usepackage{todonotes}
\usepackage[english,french]{babel} 
\usepackage{hyperref}

\swapnumbers 
\theoremstyle{definition} 
\newtheorem{Def}{Définition}[subsection] 
\theoremstyle{plain} 
\newtheorem{Pro}[Def]{Proposition} 
\newtheorem{Lem}[Def]{Lemme} 
\newtheorem{The}[Def]{Théorème} 
\newtheorem{Cor}[Def]{Corollaire} 

\newtheorem*{The*}{Théorème}

\theoremstyle{remark} 

\makeatletter
\newenvironment{abstracts}{%
  \ifx\maketitle\relax
    \ClassWarning{\@classname}{Abstract should precede
      \protect\maketitle\space in AMS document classes; reported}%
  \fi
  \global\setbox\abstractbox=\vtop \bgroup
    \normalfont\Small
    \list{}{\labelwidth\z@
      \leftmargin3pc \rightmargin\leftmargin
      \listparindent\normalparindent \itemindent\z@
      \parsep\z@ \@plus\p@
      
      \itemsep\medskipamount
    }%
}{%
  \endlist\egroup
  \ifx\@setabstract\relax \@setabstracta \fi
}

\newcommand{\abstractin}[1]{%
  \otherlanguage{#1}%
  \item[\hskip\labelsep\scshape\abstractname.]%
}
\makeatother

\title[Sur les $\ell$-blocs de niveau zéro des groupes $p$-adiques]{Sur les $\ell$-blocs de niveau zéro des groupes $p$-adiques}
\author{Thomas Lanard}
\email{thomas.lanard@imj-prg.fr}

\begin{document}

\begin{abstracts}
\abstractin{french}
Soit $G$ un groupe $p$-adique se déployant sur une extension non-ramifiée. Nous décomposons $Rep_{\Lambda}^{0}(G)$, la catégorie abélienne des re\-pré\-sen\-ta\-tions lisses de $G$ de niveau $0$ à coefficients dans $\Lambda=\overline{\mathbb{Q}}_{\ell}$ ou $\overline{\mathbb{Z}}_{\ell}$, en un produit de sous-catégories indexées par des paramètres inertiels à valeurs dans le dual de Langlands. Ces dernières sont construites à partir de systèmes d'idempotents sur l'immeuble de Bruhat-Tits et de la théorie de Deligne-Lusztig. Nous montrons ensuite des compatibilités aux foncteurs d'induction et de restriction paraboliques ainsi qu'à la correspondance de Langlands locale.

\abstractin{english}
Let $G$ be a $p$-adic group that splits over an unramified extension. We decompose $Rep_{\Lambda}^{0}(G)$, the abelian category of smooth level $0$ representations of $G$ with coefficients in $\Lambda=\overline{\mathbb{Q}}_{\ell}$ or $\overline{\mathbb{Z}}_{\ell}$, into a product of subcategories indexed by inertial Langlands parameters. We construct these categories via systems of idempotents on the Bruhat-Tits building and Deligne-Lusztig theory. Then, we prove compatibilities with parabolic induction and restriction functors and the local Langlands correspondence.
\end{abstracts}

\maketitle

\section*{Introduction}
Soient $k$ un corps local non archimédien et $\mathbf{G}$ un groupe réductif connexe défini sur $k$. Notons $G:=\mathbf{G}(k)$. Soit $\ell$ un nombre premier, $\ell\neq p$, et posons $\Lambda=\overline{\mathbb{Q}}_{\ell}$ ou $\overline{\mathbb{Z}}_{\ell}$. On appelle $Rep_{\Lambda}(G)$ la catégorie abélienne des représentations lisses de $G$ à coefficients dans $\Lambda$ et $Rep_{\Lambda}^{0}(G)$ la sous-catégorie pleine des représentations de niveau 0.
\bigskip

Pour $\Lambda=\overline{\mathbb{Q}}_{\ell}$ le théorème de décomposition de Bernstein nous fournit une décomposition de $Rep_{\overline{\mathbb{Q}}_{\ell}}(G)$ en un produit de blocs (c'est à dire de sous-catégories indécomposables). Le cas $\Lambda=\overline{\mathbb{Z}}_{\ell}$ est quant à lui assez peu connu. Pour $\mathbf{G}=GL_{n}$, Vignéras a obtenu dans \cite{vigneras} une décomposition de la catégorie $Rep_{\overline{\mathbb{F}}_{\ell}}(GL_{n}(k))$ en blocs (voir aussi les travaux de Sécherre et Stevens \cite{SecherreStevens}). Celle-ci a permis par la suite à Helm \cite{helm} d'obtenir une décomposition de $Rep_{\overline{\mathbb{Z}}_{\ell}}(GL_{n}(k))$. Prenons une classe d'équivalence inertielle de paires $(L,\pi)$ où $L$ est un Levi de $GL_{n}(k)$ et $\pi$ est une représentation supercuspidale irréductible de $L$ sur $\overline{\mathbb{F}}_{\ell}$. Définissons alors $Rep_{\overline{\mathbb{Z}}_{\ell}}(GL_{n}(k))_{[L,\pi]}$ la sous-catégorie pleine de $Rep_{\overline{\mathbb{Z}}_{\ell}}(GL_{n}(k))$ dont les objets sont les $\Pi$ tels que tous les sous-quotients simples de $\Pi$ ont un "support supercuspidal inertiel mod $\ell$" donné par $(L,\pi)$. Les blocs de $Rep_{\overline{\mathbb{Z}}_{\ell}}(GL_{n}(k))$ sont exactement les sous-catégories $Rep_{\overline{\mathbb{Z}}_{\ell}}(GL_{n}(k))_{[L,\pi]}$.

\bigskip

Ces méthodes sont basées sur "l'unicité du support supercuspidal", qui est vraie pour $GL_{n}$ mais ne l'est pas en général. Un contre exemple dans $Sp_{8}$ sur un corps fini a été trouvé récemment par Dudas puis a été relevé sur un corps $p$-adique par Dat dans \cite{datParabolic}.

\bigskip

Dans le cas du niveau 0, Dat propose (voir \cite{dat_equivalences_2014}) une nouvelle construction des blocs de $GL_{n}(k)$ en utilisant la théorie de Deligne-Lusztig et des systèmes d'idempotents sur l'immeuble de Bruhat-Tits semi-simple (comme dans l'article de Meyer et Solleveld \cite{meyer_resolutions_2010}). Il réinterprète également dans \cite{datFunctoriality} les paramétrisations des décompositions précédentes de $Rep_{\Lambda}(GL_{n}(k))$ en termes "duaux". Introduisons quelques notations pour un énoncé plus précis.

\bigskip

On suppose que $\mathbf{G}$ est déployé sur l'extension non-ramifiée maximale de $k$, c'est à dire que $\mathbf{G}$ est une forme intérieure d'un groupe non-ramifié. Soient $W_{k}$ le groupe de Weil absolu de $k$ et $I_{k}$ le sous-groupe d'inertie. Le groupe $\Gamma:=Gal(\bar{k}/k)/I_{k}$ est topologiquement engendré par le Frobenius inverse $\text{Frob}$. Via le choix d'un épinglage de $\widehat{\mathbf{G}}$, le dual de $\mathbf{G}$ sur $\overline{\mathbb{Q}}_{\ell}$, $\text{Frob}$ induit un automorphisme que l'on nomme $\widehat{\vartheta}$ sur $\widehat{\mathbf{G}}$. On note $\Phi(\textbf{G})$ l'ensemble des morphismes admissibles $\phi : W'_{k} \rightarrow {}^{L}\textbf{G}(\overline{\mathbb{Q}}_{\ell})$ modulo les automorphismes intérieurs par des éléments de $\widehat{\textbf{G}}(\overline{\mathbb{Q}}_{\ell})$, où  $W_{k}'=W_{k}\ltimes \overline{\mathbb{Q}}_{\ell}$ désigne le groupe de Weil-Deligne et ${}^{L}\textbf{G}(\overline{\mathbb{Q}}_{\ell}):=\langle\widehat{\vartheta}\rangle \ltimes\widehat{\textbf{G}}(\overline{\mathbb{Q}}_{\ell})$ le groupe de Langlands dual. Posons $I_{k}^{\Lambda}:=I_{k}$ si $\Lambda = \overline{\mathbb{Q}}_{\ell}$ et $I_{k}^{\Lambda}:=I_{k}^{(\ell)}$, le sous-groupe fermé maximal de $I_{k}$ de pro-ordre premier à $\ell$, si $\Lambda = \overline{\mathbb{Z}}_{\ell}$. Pour $I$ un sous-groupe de $W_{k}$, définissons $\Phi(I,\textbf{G})$ l'ensemble des classes de $\widehat{\textbf{G}}$-conjugaison de morphismes continus $I \rightarrow {}^{L}\textbf{G}(\overline{\mathbb{Q}}_{\ell})$ qui admettent une extension à un $L$-morphisme de $\Phi(\textbf{G})$. On s'intéressera principalement à $\Phi(I_{k}^{\Lambda},\textbf{G})$. Enfin, si $I$ contient l'inertie sauvage, posons $\Phi_{m}(I,\textbf{G})$ les éléments de $\Phi(I,\textbf{G})$ qui sont triviaux sur l'inertie sauvage.

 Il y a alors une bijection entre les blocs de Bernstein pour $GL_{n}(k)$ et les éléments de $\Phi(I_k,GL_{n})$ ainsi qu'entre les blocs de Vignéras-Helm et $\Phi(I_{k}^{(\ell)},GL_{n})$, nous donnant ainsi une décomposition en blocs
\[Rep_{\Lambda}(GL_{n}(k))=\prod_{\phi \in \Phi(I_{k}^{\Lambda},GL_{n})} Rep^{\phi}_{\Lambda}(GL_{n}(k)).\]
De plus les facteurs apparaissant dans $Rep_{\Lambda}^{0}(GL_{n}(k))$, la sous-catégorie de niveau 0, correspondent aux $\phi \in \Phi_{m}(I_{k}^{\Lambda},GL_{n})$.

\bigskip

La méthode de Vignéras ne pouvant s'appliquer dans le cas général, nous nous proposons ici de généraliser la construction de systèmes d'idempotents via la théorie de Deligne-Lusztig pour obtenir des décompositions de $Rep_{\Lambda}^{0}(G)$. Nous obtenons ainsi une décomposition analogue à celle de Dat (voir sections \ref{sectionParamInertie} et \ref{secProprieteRepPhi})

\begin{The*}
Soit $\mathbf{G}$ un groupe réductif connexe défini sur $k$ et déployé sur une extension non-ramifiée de $k$. Alors la catégorie de niveau $0$ se décompose en
\[ Rep_{\Lambda}^{0}(G) = \prod_{\phi \in \Phi_{m}(I_{k}^{\Lambda},\mathbf{G})} Rep_{\Lambda}^{\phi}(G).\]

De plus, les catégories $Rep_{\Lambda}^{\phi}(G)$ vérifient les propriétés suivantes :

\begin{enumerate}
\item Lien entre $\overline{\mathbb{Z}}_{\ell}$ et $\overline{\mathbb{Q}}_{\ell}$ : Soit $\phi \in \Phi_{m}(I_{k}^{(\ell)},\mathbf{G})$, alors $Rep_{\overline{\mathbb{Z}}_{\ell}}^{\phi}(G) \cap Rep_{\overline{\mathbb{Q}}_{\ell}}(G) = \prod_{\phi'} Rep_{\overline{\mathbb{Q}}_{\ell}}^{\phi'}(G)$ où le produit est pris sur les $\phi' \in \Phi_{m}(I_k,\mathbf{G})$ tels que $\phi'_{\mid I_{k}^{(\ell)}} \sim \phi$.

\item Représentations irréductibles de $Rep_{\overline{\mathbb{Q}}_{\ell}}^{\phi}(G)$ : Soit $\pi \in \text{Irr}_{\overline{\mathbb{Q}}_{\ell}}(G)$. Alors $\pi \in Rep_{\overline{\mathbb{Q}}_{\ell}}^{\phi}(G)$ si et seulement s'il existe $\mathbf{T}$ un tore maximal non-ramifié de $\mathbf{G}$, $\phi_{\mathbf{T}} \in \Phi_{m}(I_k,\mathbf{T})$ et $x$ un sommet de l'immeuble de $G$ (sur $k$) qui est dans l'appartement de $\mathbf{T}$ (sur $K$) tels que
\begin{enumerate}
\item $\langle \pi^{G_{x}^{+}}, \mathcal{R}_{\textbf{\textsf{T}}_{x}}^{\overline{\textsf{\textbf{G}}}_{x}}(\theta_{\mathbf{T}}) \rangle \neq 0$
\item $\iota \circ \phi_{\mathbf{T}} \sim \phi$
\end{enumerate}
où $\iota$ est un plongement $\iota : {}^{L}\mathbf{T} \hookrightarrow {}^{L}\mathbf{G}$, $G_{x}^{\circ}$ est le sous-groupe parahorique en $x$, $G_{x}^{+}$ son pro-p-radical, $\overline{\textsf{G}}_{x} \simeq G_{x}^{\circ}/G_{x}^{+}$ le quotient réductif, $\textbf{\textsf{T}}_{x}$ est le tore induit par $\mathbf{T}$ sur $\overline{\textsf{G}}_{x}$, $\mathcal{R}_{\textbf{\textsf{T}}_{x}}^{\overline{\textsf{\textbf{G}}}_{x}}$ est l'induction de Deligne-Lusztig et $\theta_{\mathbf{T}}$ est le caractère de niveau 0 de $\mathbf{T}^{F}$ correspondant à $\phi_{\mathbf{T}}$ via la correspondance de Langlands pour les tores restreinte à l'inertie.\\
(Notons que l'on obtient également une description de $Rep_{\overline{\mathbb{Z}}_{\ell}}^{\phi}(G) \cap \text{Irr}_{\overline{\mathbb{Q}}_{\ell}}(G)$ grâce au $(1)$.)

\item Compatibilité à l'induction et à la restriction parabolique : Soient $\mathbf{P}$ un sous-groupe parabolique de $\mathbf{G}$ ayant pour facteur de Levi $\mathbf{M}$ et $\iota: {}^{L}\mathbf{M} \hookrightarrow  {}^{L}\mathbf{G}$ un plongement.

\begin{enumerate}
\item Soit $\phi_{M} \in \Phi_{m}(I_{k}^{\Lambda},\mathbf{M})$, alors $i_{P}^{G}( Rep_{\Lambda}^{\phi_{M}}(M) ) \subseteq Rep_{\Lambda}^{\iota \circ \phi_{M}}(G)$, où $i_{P}^{G}$ dé\-signe l'induction parabolique.
\item Soit $\phi \in  \Phi_{m}(I_{k}^{\Lambda},\mathbf{G})$, alors $r_{P}^{G}( Rep_{\Lambda}^{\phi}(G) ) \subseteq \prod_{\phi_{M}} Rep_{\Lambda}^{\phi_{M}}(M)$, où $r_{P}^{G}$ dé\-signe la restriction parabolique et le produit est pris sur les $\phi_{M} \in \Phi_{m}(I_{k}^{\Lambda},\mathbf{M})$ tels que $\iota \circ \phi_{M} \sim \phi$.

\item Soit $\phi_{M} \in \Phi_{m}(I_{k}^{\Lambda},\mathbf{M})$. Posons $\phi = \iota \circ \phi_{M} \in  \Phi_{m}(I_{k}^{\Lambda},\mathbf{G})$ et notons $C_{\widehat{\mathbf{G}}}(\phi)$ le centralisateur dans $\widehat{\mathbf{G}}$ de l'image de $\phi$. Alors si $C_{\widehat{\mathbf{G}}}(\phi) \subseteq \iota(\widehat{\mathbf{M}})$ le foncteur $i_{P}^{G}$ réalise une équivalence de catégories entre $Rep_{\Lambda}^{\phi_{M}}(M)$ et $Rep_{\Lambda}^{\phi}(G)$.
\end{enumerate}
\end{enumerate}
\end{The*}

On s'attend à ce que cette décomposition soit compatible à la correspondance de Langlands. Nous le vérifions dans le cas des groupes classiques.

\begin{The*} 

Supposons que $\mathbf{G}$ est un groupe classique non-ramifié, $\Lambda=\overline{\mathbb{Q}}_{\ell}$, $k$ est de caractéristique nulle et de caractéristique résiduelle impaire. On obtient alors également les propriétés suivantes
\begin{enumerate}
\item Compatibilité à la correspondance de Langlands : Soient $\pi \in \text{Irr}_{\overline{\mathbb{Q}}_{\ell}}(G)$ une représentation irré\-ductible de niveau 0 et $\phi \in \Phi_{m}(I_k,\mathbf{G})$. Notons $\varphi_{\pi}$ le paramètre de Langlands associé à $\pi$ via la correspondance de Langlands locale pour les groupes classiques (\cite{HarrisTaylor} \cite{henniart} \cite{arthur} \cite{mok} \cite{KMSW}). Alors

\[ \pi \in Rep_{\overline{\mathbb{Q}}_{\ell}}^{\phi}(G) \Leftrightarrow \varphi_{\pi | I_{k}} \sim \phi \]

\item Blocs stables :  Soit $\phi \in  \Phi_{m}(I_{k},\mathbf{G})$ tel que $C_{\widehat{\mathbf{G}}}(\phi)$ soit connexe. Alors $Rep_{\overline{\mathbb{Q}}_{\ell}}^{\phi}(G)$ est un "bloc stable" (c'est-à-dire, correspond à un idempotent primitif du centre de Bernstein stable au sens de Haines \cite{haines}).
\end{enumerate}
\end{The*}

Contrairement au cas de $GL_{n}$ les catégories $Rep_{\Lambda}^{\phi}(G)$ ne sont pas des facteurs indécomposables en général. Dans un prochain article, en cours d'écriture, nous expliquerons comment pousser la méthode utilisée ici, pour obtenir une décomposition minimale construite à partir de systèmes d'idempotents et de la théorie de Deligne-Lusztig. Nous interpréterons également cette nouvelle décomposition à l'aide d'invariants cohomologiques.

\bigskip

Dans cet article nous nous limitons à l'étude de $Rep_{\Lambda}^{0}(G)$ la catégorie de niveau 0. On peut cependant espérer que les travaux menés sur la théorie des types permettront d'en déduire des résultats sur $Rep_{\Lambda}(G)$. En effet Chinello montre dans \cite{chinello} que l'on a une équivalence de catégories entre chaque bloc de $Rep_{\Lambda}(GL_{n}(k))$ et un bloc de niveau 0 de  $Rep_{\Lambda}^{0}(G')$ où $G'$ est un groupe de la forme $G'=\prod_{i=1}^{r} GL_{n_i}(D_{i})$ avec pour $i \in \{1, \cdots ,r\}$, $n_{i}$ un entier et $D_i$ une algèbre à division centrale de dimension finie sur un corps $p$-adique. Des travaux en cours de Stevens et Kurinczuk tentent d'étendre ces résultats à un $G$ plus général.

\bigskip

Cet article se compose de 4 parties. La partie 1 rappelle les résultats sur les systèmes cohérents d'idempotents. Dans la seconde partie nous expliquons comment construire des idempotents à partir de la théorie de Deligne-Lusztig. Nous associons aux paramètres de l'inertie modérés des systèmes cohérents dans la troisième partie pour obtenir la décomposition du théorème précédent. La dernière partie a pour but de montrer les propriétés des facteurs directs ainsi obtenus. Les quatre premières propriétés découlent de la construction de ces catégories et la dernière repose sur les travaux de Moeglin \cite{moeglin}, Moussaoui \cite{moussaoui}, Haines \cite{haines}, Lust et Stevens \cite{StevensLust}.

\bigskip

\paragraph{\textbf{Remerciements}}
Je tiens à remercier Jean-François Dat pour son aide précieuse concernant la rédaction de cet article.

\tableofcontents

\section*{Notations}

Soit $k$ un corps local non archimédien et $\mathfrak{f}$ son corps résiduel. Notons $q=|\mathfrak{f}|$. On fixe une clôture algébrique $\overline{k}$ de $k$ et on note $K$ l'extension non-ramifiée maximale de $k$ dans $\overline{k}$. On appellera $\mathfrak{o}_{k}$ (resp. $\mathfrak{o}_{K}$) l'anneau des entiers de $k$ (resp. $K$). Notons également $\mathfrak{F}$ le corps résiduel de $K$ qui est alors une clôture algébrique de $\mathfrak{f}$.

On adopte les conventions d'écriture suivantes. $\textbf{G}$ désignera un groupe réductif connexe défini sur $k$ que l'on identifiera avec $\textbf{G}(\overline{k})$ et l'on notera $G:=\textbf{G}(k)$ et $G^{nr}:=\textbf{G}(K)$. On appellera $\widehat{\mathbf{G}}$ son groupe dual sur $\overline{\mathbb{Q}}_{\ell}$. Pour les groupes réductifs connexes sur $\mathfrak{f}$, nous utiliserons la police d'écriture $\textsf{\textbf{G}}$ et l'on identifiera $\textsf{\textbf{G}}$ avec $\textsf{\textbf{G}}(\mathfrak{F})$ et de même on note $\textsf{G}:=\textsf{\textbf{G}}(\mathfrak{f})$. Le groupe dual de $\textbf{G}$ sur $\mathfrak{F}$ sera noté $\textsf{\textbf{G}}^{*}$.

Si $H$ désigne l'ensemble des points d'un groupe algébrique à valeur dans un corps alors $H_{ss}$ désigne l'ensemble des classes de conjugaison semi-simples dans $H$.

On fixe dans tout ce papier un système compatible de racines de l'unité $(\mathbb{Q}/\mathbb{Z})_{p'} \overset{\backsim}{\rightarrow} \mathfrak{F}^{\times}$ et $(\mathbb{Q}/\mathbb{Z})_{p'} \hookrightarrow \overline{\mathbb{Z}}_{\ell}^{\times}$.

Dans la suite $\textbf{G}$ désignera un groupe réductif connexe défini sur $k$.

\section{Système cohérent d'idempotents}

Les diverses décompositions obtenues dans cet article sont construites à partir de systèmes cohérents d'idempotents. Cette partie à pour but de rappeler leur définition ainsi que les premières propriétés.

\bigskip

Soit $K_{1}$ une extension non-ramifiée de $k$. On note $BT(K_{1})$ l'immeuble de Bruhat-Tits semi-simple associé à $\mathbf{G}(K_{1})$. L'immeuble est un complexe polysimplicial et l'on note $BT_{0}(K_{1})$ pour l'ensemble des polysimplexes de dimension 0, c'est à dire les sommets. Dans la suite on utilisera des lettres latines $x$,$y$,$\cdots$ pour parler des sommets et des lettres grecques pour parler des polysimplexes généraux $\sigma$,$\tau$, $\cdots$. L'immeuble $BT(K_{1})$ est partiellement ordonné par la relation d'ordre $\sigma \leq \tau$ si $\sigma$ est une facette de $\tau$. Un ensemble de polysimplexes $\sigma_{1},\cdots,\sigma_{k}$ est dit adjacent s'il existe un polysimplexe $\sigma$ tel que $\forall i \in {1,\cdots,k}$, $\sigma_{i} \leq \sigma$. Si $x$ et $y$ sont deux sommets adjacents on notera $[x,y]$ le plus petit polysimplexe contenant $x\cup y$. Notons également, pour $\sigma$,$\tau$ deux polysimplexes, $H(\sigma,\tau)$ l'enveloppe polysimpliciale de $\sigma$ et $\tau$, c'est à dire l'intersection de tous les appartements contenant $\sigma \cup \tau$.

Pour simplifier les notations nous noterons $BT:=BT(k)$ et $BT_{0}:=BT_{0}(k)$.

\bigskip

Soit $R$ un anneau commutatif dans lequel $p$ est inversible. On munit $G$ d'une mesure de Haar et on note $\mathcal{H}_{R}(G)$ l'algèbre de Hecke à coefficients dans $R$, c'est à dire l'algèbre des fonctions de $G$ dans $R$ localement constantes et à support compact.

\begin{Def}
\label{defcoherent}
On dit qu'un système d'idempotents $e=(e_{x})_{x\in BT_{0}}$ de $\mathcal{H}_{R}(G)$ est cohérent si les propriétés suivantes sont satisfaites :
\begin{enumerate}
\item $e_{x}e_{y}=e_{y}e_{x}$ lorsque $x$ et $y$ sont adjacents.
\item $e_{x}e_{z}e_{y}=e_{x}e_{y}$ lorsque $z \in H(x,y)$ et $z$ est adjacent à $x$.
\item $e_{gx}=ge_{x}g^{-1}$ quel que soit $x\in BT_{0}$ et $g\in G$.
\end{enumerate}
\end{Def}

Soit $Rep_{R}(G)$ la catégorie abélienne des représentations lisses de $G$ à coefficients dans $R$. Grâce à un résultat de Meyer et Solleveld on a :
\begin{The}[\cite{meyer_resolutions_2010}, Thm 3.1]
\label{themoyersolleveld}
Soit $e=(e_{x})_{x\in BT_{0}}$ un système cohérent d'idempotents, alors la sous-catégorie pleine $Rep_{R}^{e}(G)$ des objets $V$ de $Rep_{R}(G)$ tels que $V=\sum_{x\in BT_{0}}e_{x}V$ est une sous-catégorie de Serre.
\end{The}

Soit $\sigma \in BT$. Notons $G_{\sigma}$ le fixateur de $\sigma$. Celui-ci contient un sous-groupe appelé sous-groupe parahorique, que l'on note $G_{\sigma}^{\circ}$, qui est le "fixateur connexe" de $\sigma$. Enfin on note $G_{\sigma}^{+}$ le pro-$p$-radical (pro-$p$-sous-groupe distingué maximal) de $G_{\sigma}^{\circ}$.

Si $x$ est un sommet de l'immeuble $BT$ alors $G_{x}^{+}$ détermine un idempotent $e_{x}^{+} \in \mathcal{H}_{\mathbb{Z}[1/p]}(G_{x})$.

\begin{Pro}[\cite{meyer_resolutions_2010}, Section 2.2]
\label{proexcoherent}
Le système d'idempotents $(e_{x}^{+})_{x\in BT_{0}}$ est cohérent.
\end{Pro}

On a de plus que, pour tout polysimplexe $\sigma$, l'idempotent $e_{\sigma}^{+}:=\prod_{x\in \sigma} e_{x}^{+}$ est l'idempotent associé à $G_{\sigma}^{+}$.

\begin{Lem}
\label{lemchemin}
Soient $x,y \in BT_{0}$ deux sommets. Alors il existe une suite de sommets $x_{0}=x,x_{1},\cdots,x_{\ell}=y$ joignant $x$ à $y$, telle que pour tout $i \in\{0,\cdots,l-1\}$, $x_{i+1}$ est adjacent à $x_{i}$ et $x_{i+1} \in H(x_{i},y)$.
\end{Lem}

\begin{proof}
Pour $x_{1}$, il suffit de prendre un sommet dans $H(x,y)$ qui est adjacent à $x$ et tel que la distance de $x_1$ à $y$ est strictement inférieure à celle de $x$ à $y$. En ré-appliquant ce résultat à $x_{1}$ et $y$ on construit $x_{2}$ et ainsi de suite pour obtenir le résultat voulu par récurrence.
\end{proof}

\begin{Def}
On dit qu'un système $(e_{\sigma})_{\sigma\in BT}$ est 0-cohérent si
\begin{enumerate}
\item $e_{gx}=ge_{x}g^{-1}$ quel que soit $x\in BT_{0}$ et $g\in G$.
\item $e_{\sigma}=e_{\sigma}^{+}e_{x}=e_{x}e_{\sigma}^{+}$ pour $x \in BT_{0}$ et $\sigma \in BT$ tels que $x \leq \sigma$.
\end{enumerate}
\end{Def}

En s'inspirant de \cite{dat_equivalences_2014} section 3.2.1 on montre :

\begin{Pro}
\label{prosystemecoherent}
Soit $(e_{\sigma})_{\sigma\in BT}$ un système d'idempotents 0-cohérent, alors le système d'idempotents $(e_{x})_{x \in BT_{0}}$ est cohérent.

On a de plus, pour $x,y \in BT_{0}$, $e_{x}^{+}e_{y}=e_{x}e_{y}$.
\end{Pro}

\begin{proof}
Il ne reste à vérifier que les conditions 1. et 2. de la définition \ref{defcoherent}.

Commençons par vérifier la propriété 1. de \ref{defcoherent}:
Soient $x$ et $y$ deux sommets adjacents et $\sigma$ le polysimplexe $[x,y]$. On sait déjà que $e_{x}^{+}e_{y}^{+}=e_{\sigma}^{+}=e_{\sigma}^{+}e_{\sigma}^{+}$. Ainsi
\[e_{x}e_{y}=e_{x}e_{x}^{+}e_{y}^{+}e_{y}=e_{x}e_{\sigma}^{+}e_{\sigma}^{+}e_{y}=e_{\sigma}e_{\sigma}=e_{\sigma}\]
Ce qui montre la propriété 1. de \ref{defcoherent}.

Examinons maintenant la propriété 2. de \ref{defcoherent}:
Soient $x$,$y$ et $z$ des sommets de $BT$ tels que $z$ soit dans l'enveloppe polysimpliciale de $\{x,y\}$ et $z$ adjacent à $x$. Par la proposition \ref{proexcoherent} on sait que $e_{x}^{+}e_{y}^{+}=e_{x}^{+}e_{z}^{+}e_{y}^{+}$. Par ce qui précède $e_{[x,z]}=e_{x}e_{z}$ et on a
\[ e_{x}e_{y}=e_{x}e_{x}^{+}e_{y}^{+}e_{y}=e_{x}e_{x}^{+}e_{z}^{+}e_{y}^{+}e_{y}=e_{x}e_{[x,z]}^{+}e_{y}=e_{[x,z]}e_{y}=e_{x}e_{z}e_{y}\]
Le système d'idempotents $(e_{x})_{x \in BT_{0}}$ est cohérent.

\bigskip

Montrons maintenant que $e_{x}^{+}e_{y}=e_{x}e_{y}$.\\
Si $x$ et $y$ sont adjacents on obtient que

\[ e_{x}^{+}e_{y}=e_{x}^{+}e_{y}^{+}e_{y}=e_{[x,y]}^{+}e_{y}=e_{[x,y]}=e_{x}e_{y}\]
Dans le cas général choisissons grâce au lemme \ref{lemchemin} une suite de sommets $x_{0}=x,x_{1},\cdots,x_{\ell}=y$ joignant $x$ à $y$, telle que pour tout $i \in\{0,\cdots,l-1\}$, $x_{i+1}$ est adjacent à $x_{i}$ et $x_{i+1} \in H(x_{i},y)$. Alors
\begin{align*}
e_{x}^{+}e_{y}&=e_{x}^{+}e_{y}^{+}e_{y}=e_{x}^{+}e_{x_{1}}^{+}e_{y}^{+}e_{y}=\cdots=e_{x}^{+}e_{x_{1}}^{+}\cdots e_{x_{l-1}}^{+}e_{y}^{+}e_{y}\\
&=e_{x}^{+}e_{x_{1}}^{+}\cdots e_{x_{l-1}}^{+}e_{y}=e_{x}^{+}e_{x_{1}}^{+}\cdots e_{x_{l-1}}e_{y}=\cdots=e_{x}e_{x_{1}}\cdots e_{x_{l-1}}e_{y}\\
&=e_{x}e_{x_{1}}\cdots e_{x_{l-2}}e_{y}=\cdots=e_{x}e_{x_{1}}e_{y}=e_{x}e_{y}
\end{align*}
(La première ligne découle de la propriété 2. de \ref{defcoherent} appliquée aux $e_{x}^{+}$. Pour la seconde, on utilise le fait que $e_{x_{i}}^{+}e_{x_{i+1}}=e_{x_{i}}e_{x_{i+1}}$ car $x_{i}$ et $x_{i+1}$ sont adjacents. Enfin pour la dernière on applique que $e_{x_{i}}e_{x_{i+1}}e_{y}=e_{x_{i}}e_{y}$ car $x_{i+1}$ est adjacent à $x_{i}$ et $x_{i+1} \in H(x_{i},y)$.)

Ainsi
\[ \forall x,y \in BT_{0}, e_{x}^{+}e_{y}=e_{x}e_{y}\]
\end{proof}

\section{Construction d'idempotents}

Nous expliquons ici comment construire des idempotents sur l'immeuble à partir de la théorie de Deligne-Lusztig ainsi que les conditions qu'ils doivent vérifier pour obtenir un système cohérent.

\subsection{Théorie de Deligne-Lusztig}

\label{secDeligneLusztig}

Rappelons brièvement le fonctionnement de la théorie de Deligne-Lusztig. Dans cette section les groupes algébriques considérés seront sur $\mathfrak{F}$ et $\ell$ est un nombre premier différent de $p$.

\bigskip

Soit $\textsf{\textbf{G}}$ un groupe réductif défini sur $\mathfrak{f}$. Soit $\textsf{\textbf{P}}$ un sous-groupe parabolique de radical unipotent $\textsf{\textbf{U}}$ et supposons que $\textsf{\textbf{P}}$ contienne un Levi $F$-stable $\textsf{\textbf{L}}$. On associe alors à ces données la variété de Deligne-Lusztig $Y_{\textsf{\textbf{P}}}$ définie par
\[ Y_{\textsf{\textbf{P}}}=\{ g\textsf{\textbf{U}} \in \textsf{\textbf{G}}/\textsf{\textbf{U}}, g^{-1}F(g) \in \textsf{\textbf{U}}F(\textsf{\textbf{U}})\}\]
C'est une variété définie sur $\mathfrak{F}$ avec une action à gauche de $\mathsf{G}:=\textsf{\textbf{G}}^{F}$ donnée par $(\gamma,g\textsf{\textbf{U}})\mapsto\gamma g \textsf{\textbf{U}}$ et une action à droite de $\mathsf{L}:=\textsf{\textbf{L}}^{F}$ donnée par $(g \textsf{\textbf{U}},\delta)\mapsto g\delta \textsf{\textbf{U}}$. Le complexe cohomologique $R\Gamma_{c}(Y_{\textsf{\textbf{P}}},\overline{\mathbb{Z}}_{\ell})$ est alors un complexe de $(\overline{\mathbb{Z}}_{\ell}[G],\overline{\mathbb{Z}}_{\ell}[L])$-bimodules et induit deux foncteurs adjoints
\[ \mathcal{R}_{\textsf{\textbf{L}} \subset \textsf{\textbf{P}}}^{\textsf{\textbf{G}}}=R\Gamma_{c}(Y_{\textsf{\textbf{P}}},\overline{\mathbb{Z}}_{\ell}) \otimes_{\overline{\mathbb{Z}}_{\ell}[L]} - : D^{b}(\overline{\mathbb{Z}}_{\ell}[L]) \longrightarrow D^{b}(\overline{\mathbb{Z}}_{\ell}[G])\]
\[ {}^{*}\mathcal{R}_{\textsf{\textbf{L}} \subset \textsf{\textbf{P}}}^{\textsf{\textbf{G}}}=R\text{Hom}_{\overline{\mathbb{Z}}_{\ell}[G]}(R\Gamma_{c}(Y_{\textsf{\textbf{P}}},\overline{\mathbb{Z}}_{\ell}) , -) : D^{b}(\overline{\mathbb{Z}}_{\ell}[G]) \longrightarrow D^{b}(\overline{\mathbb{Z}}_{\ell}[L])\]
où $D^{b}$ signifie "catégorie dérivée bornée".

\bigskip

Nous avons fixé des systèmes compatibles de racines de l'unité $(\mathbb{Q}/\mathbb{Z})_{p'} \overset{\backsim}{\longrightarrow} \mathfrak{F}^{\times}$ et $(\mathbb{Q}/\mathbb{Z})_{p'} \hookrightarrow \overline{\mathbb{Z}}_{\ell}^{\times}$. Si $\textsf{\textbf{T}}$ est un tore défini sur $\mathfrak{f}$ et $\textsf{\textbf{T}}^{*}$ est son tore dual, aussi défini sur $\mathfrak{f}$, alors les deux applications précédentes permettent de définir une bijection $\textsf{\textbf{T}}^{*F}\rightarrow \text{Hom}(T,\overline{\mathbb{Z}}_{\ell}^{\times})$. Ainsi un élément $t\in \textsf{\textbf{T}}^{*F}$ détermine un caractère $\widehat{t}:T \rightarrow \overline{\mathbb{Z}}_{\ell}^{\times}$.

Soit $s$ une classe de conjugaison semi-simple dans $\mathsf{G}^{*}:=\textsf{\textbf{G}}^{*F}$. Une représentation irréductible $\pi \in \text{Irr}_{\overline{\mathbb{Q}}_{\ell}}(\mathsf{G})$ appartient à la série rationnelle attachée à $s$ s'il existe un tore $F$-stable $\textsf{\textbf{T}}$ dans $\textsf{\textbf{G}}$, un élément $t\in\textsf{\textbf{T}}^{*F}$ qui appartienne à $s$ et un Borel $\textsf{\textbf{B}}$ contenant $\textsf{\textbf{T}}$ tel que $\pi$ apparaisse avec une multiplicité non nulle dans $[\mathcal{R}_{\textsf{\textbf{T}} \subset \textsf{\textbf{B}}}^{\textsf{\textbf{G}}}(\widehat{t})]$ (où la notation $[\cdot]$ signifie que l'on prend l'image du complexe dans le groupe de Grothendieck). On note alors $\mathcal{E}(\mathsf{G},s)$ l'ensemble des telles séries rationnelles et par $e_{s,\overline{\mathbb{Q}}_{\ell}}^{\mathsf{G}} \in \overline{\mathbb{Q}}_{\ell}[\textsf{G}]$ l'idempotent central les sélectionnant.

\begin{Pro}
\label{prodelignelusztig}

\begin{enumerate}
\item (\cite{cabanes_enguehard} théorème 8.23) Nous avons $1=\sum_{s} e_{s,\overline{\mathbb{Q}}_{\ell}}^{\mathsf{G}}$ dans $\overline{\mathbb{Q}}_{\ell}[\mathsf{G}]$.
\item (\cite{bonnafe_rouquier} théorème A' et remarque 11.3) Si $s$ se compose d'éléments $\ell$-réguliers, alors nous avons un idempotent $e_{s,\overline{\mathbb{Z}}_{\ell}}^{\mathsf{G}}=\sum_{s' \backsim_{\ell} s} e_{s',\overline{\mathbb{Q}}_{\ell}}^{\mathsf{G}} \in \overline{\mathbb{Z}}_{\ell}[\mathsf{G}]$, où $s' \backsim_{\ell} s$ signifie que $s$ est la partie $\ell$-régulière de $s'$.
\end{enumerate}
\end{Pro}

\label{sectiondelignelusztig}

Soit $\textsf{\textbf{L}}$ un Levi $F$-stable de $\textsf{\textbf{G}}$ contenu dans un parabolique $\textsf{\textbf{P}}$. Une classe de conjugaison $t$ dans $\textsf{L}^{*}$ donne une classe de conjugaison $s$ dans $\textsf{G}^{*}$. Ainsi nous avons une application $\varphi_{\textsf{L}^{*},\textsf{G}^{*}}$ à fibres finies définie par
\[\begin{array}{ccccc}
\varphi_{\textsf{L}^{*},\textsf{G}^{*}} & : & \textsf{L}^{*}_{ss} & \to & \textsf{G}^{*}_{ss} \\
 & & t & \mapsto & s
\end{array}\]

Soit $\Lambda= \overline{\mathbb{Q}}_{\ell}$ ou $\overline{\mathbb{Z}}_{\ell}$. Construisons alors un idempotent $e_{s,\Lambda}^{\mathsf{L}}:=\sum_{t \in \varphi_{\textsf{L}^{*},\textsf{G}^{*}}^{-1}(s)} e_{t,\Lambda}^{\mathsf{L}}$.

Dans le cas où $\textsf{\textbf{P}}$ est lui-même $F$-stable, on note $\textsf{\textbf{U}}$ le radical unipotent de $\textsf{\textbf{P}}$ et $\textsf{U}=\textsf{\textbf{U}}^{F}$. Notons $e_{\mathsf{U}}:=\frac{1}{|\mathsf{U}|} \sum_{x\in \mathsf{U}} \mathds{1}_{x}$ l'idempotent réalisant la moyenne sur $\mathsf{U}$. Alors on a

\begin{Pro}[\cite{dat_equivalences_2014} section 2.1.4]
\label{proidemparabolique}
\[ e_{s,\Lambda}^{\mathsf{G}}e_{\mathsf{U}}=e_{\mathsf{U}}e_{s,\Lambda}^{\mathsf{L}}=:e_{s,\Lambda}^{\mathsf{P}}\]
où $e_{s,\Lambda}^{\mathsf{P}}$ est un idempotent central dans  $\Lambda[\mathsf{P}]$.
\end{Pro}

Soient $\textsf{\textbf{G}}'$ un autre groupe réductif défini sur $\mathfrak{f}$ et $\varphi : \textsf{\textbf{G}} \rightarrow \textsf{\textbf{G}}'$ un isomorphisme compatible avec les $F$-structures. Alors $\varphi$ induit une bijection (voir annexe \ref{sectionRappelsGD})

\[ \varphi^{*} : \mathsf{G}^{*}_{ss} \longrightarrow \mathsf{G}'^{*}_{ss}\]

\begin{Lem}
\label{lemphiidempotent}
On a
\[ \varphi(e_{s,\Lambda}^{\mathsf{G}})=e_{\varphi^{*}(s),\Lambda}^{\mathsf{G}'}\]
\end{Lem}

\begin{proof}
Par construction un élément $t\in\textsf{\textbf{T}}^{*F}$ appartenant à $s$ est envoyé sur $\varphi^{*}(t)$ appartenant à $\varphi^{*}(s)$. Ainsi $\varphi$ envoie $\mathcal{E}(\mathsf{G},s)$ sur $\mathcal{E}(\mathsf{G}',\varphi^{*}(s))$ et on a le résultat.
\end{proof}

\subsection{Construction d'idempotents sur l'immeuble}

Maintenant que l'on sait fabriquer des idempotents sur les groupes finis, il nous faut les relever en des idempotents sur le groupe $p$-adique. On utilise pour cela le fait que les sous-groupes parahoriques $G_{\sigma}^{\circ}$ admettent un modèle entier et que le quotient $G_{\sigma}^{\circ}/G_{\sigma}^{+}$ est alors l'ensemble des points d'un groupe réductif connexe à valeur dans un corps fini.

\bigskip

Soit $\sigma$ un polysimplexe dans $BT$. D'après \cite{tits_reductive} section 3.4 il existe un schéma en groupes affine lisse $\mathcal{G}_{\sigma}$ défini sur $\mathfrak{o}_{k}$, unique à isomorphisme près, tel que
\begin{enumerate}
\item La fibre générique $\mathcal{G}_{\sigma,k}$ de $\mathcal{G}_{\sigma}$ est $G$
\item Pour toute extension galoisienne non-ramifiée $K_{1}$ de $k$, $\mathcal{G}_{\sigma}(\mathfrak{o}_{K_{1}})$ est le sous-groupe compact maximal de $\mathbf{G}(K_{1})_{\sigma}$, où $\sigma$ est identifié avec son image canonique dans $BT(K_{1})$.
\end{enumerate}

L'application de réduction modulo $p$ fournit un morphisme surjectif $G_{\sigma}=\mathcal{G}_{\sigma}(\mathfrak{o}_{k}) \rightarrow \tilde{\mathsf{G}}_{\sigma}:=\tilde{\textbf{\textsf{G}}}_{\sigma}^{F}$, où $\tilde{\textsf{\textbf{G}}}_{\sigma}$, la fibre spéciale de $\mathcal{G}_{\sigma}$, est un groupe algébrique défini sur $\mathfrak{f}$. On note $\mathcal{G}_{\sigma}^{\circ}$ la composante neutre de $\mathcal{G}_{\sigma}$ et $\tilde{\textsf{\textbf{G}}}_{\sigma}^{\circ}$ celle de $\tilde{\textsf{\textbf{G}}}_{\sigma}$. D'après \cite{BT} section 5.2.6, on a $G_{\sigma}^{\circ}=\mathcal{G}_{\sigma}^{\circ}(\mathfrak{o}_{k})$. D'où un morphisme surjectif $G_{\sigma}^{\circ} \rightarrow \tilde{\textsf{G}}_{\sigma}^{\circ}$.

Notons $\overline{\textbf{\textsf{G}}}_{\sigma}$ le quotient réductif de $\tilde{\textbf{\textsf{G}}}_{\sigma}^{\circ}$. On a donc un morphisme surjectif $G_{\sigma}^{\circ} \rightarrow \overline{\mathsf{G}}_{\sigma}$ de noyau $G_{\sigma}^{+}$, d'où  un isomorphisme :
\[ G_{\sigma}^{\circ} / G_{\sigma}^{+} \overset{\sim}{\longrightarrow} \overline{\mathsf{G}}_{\sigma}\]
Soit $s \in (\overline{\mathsf{G}}^{*}_{\sigma})_{ss}$ (rappelons que $\overline{\mathsf{G}}^{*}_{\sigma}=\overline{\textsf{\textbf{G}}}{}^{*F}_{\sigma}$) d'ordre inversible dans $\Lambda$, on peut alors tirer en arrière par cet isomorphisme l'idempotent $e_{s,\Lambda}^{\overline{\textsf{G}}_{\sigma}}$ en un idempotent $e_{\sigma}^{s,\Lambda} \in \Lambda[ G_{\sigma}^{\circ}/G_{\sigma}^{+} ] \subset \mathcal{H}_{\Lambda}(G_{\sigma})$.

\bigskip

Soit $x\in BT_{0}$. Si l'on considère la sous-partie de l'immeuble constituée des polysimplexes $\tau$ tels que $x \leq \tau$ alors d'après \cite{tits_reductive} section 3.5.4, on obtient l'immeuble sphérique ("immeuble des sous-groupes $\mathfrak{f}$-paraboliques") de $\overline{\mathsf{G}}_{x}$.

Soit $\sigma \in BT$ tel que $x \leq \sigma$. Alors $G_{\sigma}^{\circ} \subset G_{x}^{\circ}$. On a ainsi un morphisme $G_{\sigma}^{\circ} \rightarrow G_{x}^{\circ} \rightarrow \overline{\mathsf{G}}_{x}$. Notons $\mathsf{P}_{\sigma}$ l'image de $G_{\sigma}^{\circ}$ dans $\overline{\mathsf{G}}_{x}$ qui est un sous-groupe parabolique et $\mathsf{U}_{\sigma}$ son radical unipotent. L'image réciproque de $\mathsf{U}_{\sigma}$ dans $G_{\sigma}^{\circ}$ est $G_{\sigma}^{+}$. Ceci fournit donc un isomorphisme $G_{\sigma}/G_{\sigma}^{+} \simeq \overline{\mathsf{G}}_{\sigma} \simeq \mathsf{P}_{\sigma} / \mathsf{U}_{\sigma}$.

Considérons $\textsf{M}_{\sigma}^{*}$ un sous-groupe de Levi de $\overline{\mathsf{G}}_{x}^{*}$ relevant $\overline{\textsf{G}}_{\sigma}^{*}$. Dans la section \ref{sectiondelignelusztig}, nous avons défini une application $\varphi_{\textsf{M}_{\sigma}^{*},\overline{\mathsf{G}}^{*}_{x}} : (\textsf{M}_{\sigma}^{*})_{ss} \rightarrow (\overline{\mathsf{G}}^{*}_{x})_{ss} $. Cette application est indépendante du choix du relèvement de $\overline{\textsf{G}}_{\sigma}^{*}$ et nous définit donc une application $\varphi^{*}_{\sigma,x}: (\overline{\mathsf{G}}^{*}_{\sigma})_{ss} \rightarrow (\overline{\mathsf{G}}^{*}_{x})_{ss} $. On définit alors pour $s \in (\overline{\mathsf{G}}^{*}_{x})_{ss}$ d'ordre inversible dans $\Lambda$, l'idempotent $e_{\sigma}^{s,\Lambda}:=\sum_{t\in \varphi_{\sigma,x}^{*-1}(s)}e_{\sigma}^{t,\Lambda}$.

\begin{Pro}
\label{proesigmaplus}
Soient $x\in BT_{0}$, $\sigma \in BT$ tel que $x \leq \sigma$ et $s \in (\overline{\mathsf{G}}^{*}_{x})_{ss}$ d'ordre inversible dans $\Lambda$. Alors $e_{\sigma}^{+}e_{x}^{s,\Lambda}=e_{x}^{s,\Lambda}e_{\sigma}^{+}=e_{\sigma}^{s,\Lambda}$.
\end{Pro}

\begin{proof}
D'après la proposition \ref{proidemparabolique} on a $e_{s,\Lambda}^{\overline{\mathsf{G}}_{x}}e_{\mathsf{U}_{\sigma}}=e_{\mathsf{U}_{\sigma}}e_{s,\Lambda}^{\overline{\mathsf{G}}_{\sigma}}$ dans $\Lambda[\overline{\mathsf{G}}_{x}]$. Lorsque l'on tire en arrière ces idempotents par l'isomorphisme $G_{x}^{\circ} / G_{x}^{+} \overset{\sim}{\longrightarrow} \overline{\mathsf{G}}_{x}$, on obtient dans $\Lambda[ G_{x}^{\circ}/G_{x}^{+} ]$ :
\[ e_{x}^{s,\Lambda}e_{\sigma}^{+}=e_{\sigma}^{+}e_{\sigma}^{s,\Lambda}.\]
Maintenant comme $e_{\sigma}^{+}e_{\sigma}^{s,\Lambda}=e_{\sigma}^{s,\Lambda}$ et $e_{\sigma}^{+}e_{x}^{s,\Lambda}=e_{x}^{s,\Lambda}e_{\sigma}^{+}$ on a le résultat.
\end{proof}

\subsection{Systèmes 0-cohérents de classes de conjugaison}

\label{sectSystemescoherents}

À partir d'un polysimplexe $\sigma$ et d'une classe de conjugaison semi-simple $s$ nous savons maintenant construire un idempotent $e_{\sigma}^{s,\Lambda}$. On décrit alors dans cette partie les conditions que l'on doit imposer pour que le système d'idempotents formé à partir des $e_{\sigma}^{s,\Lambda}$ soit un système 0-cohérent.

\bigskip

Soient $g \in G$ et $\sigma \in BT$, on a $g G_{\sigma}^{\circ} g^{-1}=G_{g\sigma}^{\circ}$ et $g G_{\sigma}^{+} g^{-1}= G_{g \sigma}^{+}$, d'où $g (G_{\sigma}^{\circ} / G_{\sigma}^{+} )g^{-1} \simeq G_{g\sigma}^{\circ} / G_{g \sigma}^{+}$. D'après \cite{BT} 4.6.30, la conjugaison par $g$ se prolonge en un isomorphisme du $\mathfrak{o}_{k}$-schéma en groupes $\mathcal{G}_{\sigma}^{\circ}$ sur le $\mathfrak{o}_{k}$-schéma en groupes $\mathcal{G}_{g\sigma}^{\circ}$ et donc induit un isomorphisme $\varphi_{g,\sigma}: \overline{\textsf{G}}_{\sigma} \overset{\sim}{\rightarrow} \overline{\textsf{G}}_{g\sigma}$. On obtient alors, comme dans l'annexe \ref{sectionRappelsGD}, un isomorphisme sur les classes de conjugaison semi-simples des groupes duaux
\[ \varphi^{*}_{g,\sigma}:(\overline{\textsf{G}}^{*}_{\sigma})_{ss} \overset{\sim}{\longrightarrow} (\overline{\textsf{G}}^{*}_{g\sigma})_{ss}\]

Pour $\sigma \in BT$, on note
\[(\overline{\mathsf{G}}^{*}_{\sigma})_{ss,\Lambda}=\{ s \in (\overline{\mathsf{G}}^{*}_{\sigma})_{ss} \text{ tel que } s \text{ soit d'ordre inversible dans } \Lambda \}\]

\begin{Def}
\label{defsystemsurcoherentconju}
Soit $S=(S_{\sigma})_{\sigma \in BT}$ un système d'ensembles de classes de conjugaison avec $S_{\sigma} \subseteq (\overline{\mathsf{G}}^{*}_{\sigma})_{ss,\Lambda}$. On dit que $S$ est 0-cohérent si
\begin{enumerate}
\item $\varphi^{*}_{g,x}(S_{x})=S_{gx}$ quel que soit $x\in BT_{0}$ et $g\in G$.
\item $\varphi_{\sigma,x}^{*-1}(S_{x})=S_{\sigma}$ pour $x\in BT_{0}$ et $\sigma \in BT$ tels que $x \leq \sigma$.
\end{enumerate}
\end{Def}

Soit $S=(S_{\sigma})_{\sigma \in BT}$ un système 0-cohérent. Soit $\sigma \in BT$, on définit alors $e_{\sigma}^{S,\Lambda}=\sum_{s \in S_{\sigma}} e_{\sigma}^{s,\Lambda}$.

\begin{Pro}
\label{prosystemeidempotent}
Le système $(e_{\sigma}^{S,\Lambda})_{\sigma \in BT}$ est 0-cohérent.
\end{Pro}

\begin{proof}
Commençons par vérifier la condition 1. :\\
Soient $x \in BT_{0}$ et $g\in G$. On a la commutativité du diagramme suivant
\[ \xymatrix{
G_{x}^{\circ} / G_{x}^{+} \ar@{->}[r]^-{\sim} \ar@{->}[d]^{\tiny{\mbox{conj par }}g} & \overline{\mathsf{G}}_{x} \ar@{->}[d]^{\varphi_{g,x}}\\
G_{gx}^{\circ} / G_{gx}^{+} \ar@{->}[r]^-{\sim} & \overline{\mathsf{G}}_{gx}}
\]
Le lemme \ref{lemphiidempotent} nous dit que $\varphi_{g,x}(e_{s,\Lambda}^{\overline{\mathsf{G}}_{x}})=e_{\varphi^{*}_{g,x}(s),\Lambda}^{\overline{\mathsf{G}}_{gx}}$. Ainsi
\[ge_{x}^{S,\Lambda}g^{-1}=\sum_{s \in S_{x}} ge_{x}^{s,\Lambda} g^{-1}=\sum_{s \in S_{x}} e_{gx}^{\varphi^{*}_{g,x}(s),\Lambda}=\sum_{s \in \varphi^{*}_{g,x}(S_{x})} e_{gx}^{s,\Lambda}=\sum_{s \in S_{gx}} e_{gx}^{s,\Lambda} =e_{gx}^{S,\Lambda}\]

Passons maintenant à la condition 2. : \\
Soient $x \in BT_{0}$ et $\sigma \in BT$ tels que $x \leq \sigma$.
\[e_{\sigma}^{+}e_{x}^{S,\Lambda}=\sum_{s \in S_{x}} e_{\sigma}^{+}e_{x}^{s,\Lambda}\]
Par la proposition \ref{proesigmaplus} on a $e_{\sigma}^{+}e_{x}^{s,\Lambda}=\sum_{t\in \varphi^{*-1}_{\sigma,x}(s)} e_{\sigma}^{t,\Lambda}$. Donc
\[e_{\sigma}^{+}e_{x}^{S,\Lambda}=\sum_{s \in S_{x}} \sum_{t\in \varphi^{*-1}_{\sigma,x}(s)} e_{\sigma}^{t,\Lambda}=\sum_{t \in \varphi_{\sigma,x}^{*-1}(S_{x})}e_{\sigma}^{t,\Lambda}=\sum_{t \in S_{\sigma}} e_{\sigma}^{t,\Lambda}=e_{\sigma}^{S,\Lambda}\]

\end{proof}

On note $Rep_{\Lambda}(G)$ la catégorie abélienne des représentations lisses de $G$ à coefficients dans $\Lambda$. Notons $Rep_{\Lambda}^{0}(G)$ la sous-catégorie des représentations de niveau 0, c'est à dire la sous-catégorie découpée par le système d'idempotents $(e_{x}^{+})_{x \in BT_{0}}$.

Soit $S=(S_{\sigma})_{\sigma \in BT}$ un système 0-cohérent, il définit alors un système $(e_{\sigma}^{S,\Lambda})_{\sigma \in BT}$ 0-cohérent et forme donc, d'après le théorème \ref{themoyersolleveld}, une catégorie $Rep_{\Lambda}^{S}(G)$.

\begin{Def}
Soient $S_{1}=(S_{1,\sigma})_{\sigma \in BT}$ et $S_{2}=(S_{2,\sigma})_{\sigma \in BT}$ deux systèmes de classes de conjugaison. On définit alors $S_{1} \cup S_{2}:=(S_{1,\sigma} \cup S_{2,\sigma})_{\sigma \in BT}$ et $S_{1} \cap S_{2}:=(S_{1,\sigma} \cap S_{2,\sigma})_{\sigma \in BT}$. On dit que $S_{2} \subseteq S_{1}$ si pour tout $\sigma \in BT$ $S_{2,\sigma} \subseteq S_{1,\sigma}$. Enfin, si $S_{2} \subseteq S_{1}$, on note $S_{1} \backslash S_{2}:=(S_{1,\sigma} \backslash S_{2,\sigma})_{\sigma \in BT}$.
\end{Def}

\begin{Lem}
\label{lemsystemeorthogonal}
Soient $S_{1}$ et $S_{2}$ deux systèmes 0-cohérents tels que $S_{1} \cap S_{2} = \emptyset$. Alors les catégories $Rep_{\Lambda}^{S_{1}}(G)$ et $Rep_{\Lambda}^{S_{2}}(G)$ sont orthogonales.
\end{Lem}

\begin{proof}
Soit $V$ un objet de $Rep_{\Lambda}^{S_{1}}(G)$. Nous devons montrer que pour tout sommet de l'immeuble $x$, on a $e_{x}^{S_{2}}V=0$. Fixons un tel $x$. Par définition $V=\sum_{y \in BT_{0}} e_{y}^{S_{1}} V$, donc $e_{x}^{S_{2}}V=\sum_{y \in BT_{0}} e_{x}^{S_{2}} e_{y}^{S_{1}} V$. Soit $y \in BT_{0}$. Comme $(e_{\sigma}^{S_{1}})_{\sigma \in BT}$ est 0-cohérent, on sait par \ref{prosystemecoherent} que $e_{x}^{+}e_{y}^{S_{1}}=e_{x}^{S_{1}}e_{y}^{S_{1}}$ et on a que
\[ e_{x}^{S_{2}}e_{y}^{S_{1}}=e_{x}^{S_{2}}e_{x}^{+}e_{y}^{S_{1}}=e_{x}^{S_{2}}e_{x}^{S_{1}}e_{y}^{S_{1}}\]

Or si $s$ et $s'$ sont deux éléments distincts de $(\widehat{\overline{\mathsf{G}}}_{x})_{ss}$ d'ordre inversible dans $\Lambda$, $e_{x}^{s,\Lambda}e_{x}^{s',\Lambda}=0$ donc $e_{x}^{S_{2}}e_{x}^{S_{1}}=0$ et on a le résultat.
\end{proof}

\begin{Pro}
\label{prodecompocategorie}
Soient $S_{1},\cdots,S_{n}$ des systèmes 0-cohérents tels que $S_{i} \cap S_{j} = \emptyset$ si $i \neq j$ et $\bigcup_{i=1}^{n} S_{i} = ((\overline{\mathsf{G}}^{*}_{\sigma})_{ss,\Lambda})_{\sigma \in BT}$. Alors la catégorie de niveau $0$ se décompose en 

\[ Rep_{\Lambda}^{0}(G) = \prod_{i=1}^{n} Rep_{\Lambda}^{S_{i}}(G) \]
\end{Pro}

\begin{proof}
 
 D'après le lemme \ref{lemsystemeorthogonal} nous savons déjà que les catégories $Rep_{\Lambda}^{S_{i}}(G)$ sont deux à deux orthogonales. Prenons maintenant $V$ un objet de $Rep_{\Lambda}^{0}(G)$. Par définition, $V=\sum_{x \in BT_{0}} e_{x}^{+}V$. Fixons un sommet $x \in BT_{0}$. D'après \ref{prodelignelusztig}, on a $e_{x}^{+}=\sum_{s \in (\widehat{\overline{\mathsf{G}}}_{x})_{ss}} e_{x}^{s,\overline{\mathbb{Q}}_{\ell}}$. Ainsi $e_{x}^{+}=\sum_{i=1}^{n} e_{x}^{S_{i}}$. On en déduit que
 \[ V=\sum_{x \in BT_{0}} e_{x}^{+}V=\sum_{x \in BT_{0}} \sum_{i=1}^{n} e_{x}^{S_{i}}V =\sum_{i=1}^{n} \left(\sum_{x \in BT_{0}} e_{x}^{S_{i}}V\right).\]
 Or $\sum_{x \in BT_{0}} e_{x}^{S_{i}}V$ est un objet de $Rep_{\Lambda}^{S_{i}}(G)$ d'après \cite{meyer_resolutions_2010} proposition 3.2, d'où le résultat.
\end{proof}

\section{Paramètres de l'inertie modérés}

\label{sectionParamInertie}

Dans toute cette section on suppose de plus que $\textbf{G}$ est $K$-déployé. Cela signifie que $\mathbf{G}$ est une forme intérieure d'un groupe non-ramifié. On souhaite obtenir une décomposition de $Rep_{\Lambda}^{0}(G)$ indexée par les paramètres de l'inertie modérés $\phi$. Pour cela on construit un procédé permettant d'associer à chaque $\phi$ un système de classes de conjugaison 0-cohérent.

\subsection{Classes de conjugaison dans $\textsf{\textbf{G}}^{*}$}

\label{secParamInertClasseConj}

Commençons par définir les paramètres de l'inertie modérés et montrons que l'on peut décrire ceux-ci en terme de classes de conjugaison semi-simples dans $\textsf{\textbf{G}}^{*}$.

\bigskip

On notera $\mathcal{G}_{k}=Gal(\overline{k}/k)$ le groupe de Galois absolu de $k$, $W_{k}$ le groupe de Weil absolu de $k$ et $I_{k}$ le sous-groupe d'inertie. Le groupe $\Gamma:=\mathcal{G}_{k}/I_{k}$ est topologiquement engendré par un élément $\text{Frob}$ dont l'inverse induit l'automorphisme $x \mapsto x^{q}$ sur $\mathfrak{F}$. Ainsi $K=\overline{k}^{I_{k}}$ et $k=K^{\text{Frob}}$.
L'action de $\mathcal{G}_{k}$ sur $\mathbf{G}$ donne une action de $\Gamma$ sur $\mathbf{G}(K)$, complètement déterminée par un automorphisme $F \in Aut(\mathbf{G}(K))$ donné par l'action de $\text{Frob}$. On a alors $G=\mathbf{G}(K)^{F}$.

Soit $\mathbf{T}$ un $k$-tore maximal $K$-déployé maximalement déployé, alors $I_{k}$ agit trivialement sur $X_{*}(\mathbf{T})$, le groupe des co-caractères de $\mathbf{T}$, et l'action de $\mathcal{G}_{k}$ sur $X_{*}(\mathbf{T})$ se factorise à travers $\Gamma$. Notons alors $\vartheta$ l'automorphisme de $X_{*}(\mathbf{T})$ induit par $F$. La dualité entre $X_{*}(\mathbf{T})$ et $X^{*}(\mathbf{T})$ permet d'associer de façon naturelle à $\vartheta$ un automorphisme $\widehat{\vartheta} \in Aut(X^{*}(\mathbf{T}))$. Cet automorphisme s'étend alors alors en un automorphisme $\widehat{\vartheta} \otimes 1$ de $\widehat{\mathbf{T}}(\overline{\mathbb{Q}}_{\ell}):=X^{*}(\mathbf{T})\otimes \overline{\mathbb{Q}}_{\ell}^{\times}$ que nous noterons encore $\widehat{\vartheta}$. Fixons un épinglage $(\widehat{\mathbf{G}},\widehat{\mathbf{B}},\widehat{\mathbf{T}},\{x_{\alpha}\}_{\alpha \in \Delta})$ de $\widehat{\mathbf{G}}$ où $\widehat{\mathbf{B}}$ est un Borel contenant $\widehat{\mathbf{T}}$. Celui-ci permet de prolonger $\widehat{\vartheta}$ en un automorphisme $\widehat{\vartheta} \in Aut(\widehat{\mathbf{G}})$.

On note $P_{k}$ le groupe d'inertie sauvage, c'est à dire le pro-$p$ sous-groupe maximal de $I_{k}$. Le groupe d'inertie modérée est le quotient $I_{k}/P_{k}$ et le groupe de Weil modéré est le quotient $W_{k}/P_{k}$. On note $W_{k}'=W_{k}\ltimes \overline{\mathbb{Q}}_{\ell}$ le groupe de Weil-Deligne.

\begin{Def}
\label{defLmorphisme}
Un morphisme $\varphi : W'_{k} \rightarrow {}^{L}\textbf{G}(\overline{\mathbb{Q}}_{\ell}):=\langle\widehat{\vartheta}\rangle \ltimes\widehat{\textbf{G}}(\overline{\mathbb{Q}}_{\ell})$ est dit admissible si
\begin{enumerate}
\item Le diagramme suivant commute :
\[ \xymatrix{
W'_{k} \ar@{->}[r]^{\varphi} \ar@{->}[d] & {}^{L}\textbf{G}(\overline{\mathbb{Q}}_{\ell}) \ar@{->}[d]\\
\langle \text{Frob} \rangle \ar@{->}[r] & \langle \widehat{\vartheta} \rangle }\]

\item $\varphi$ est continue, $\varphi(\overline{\mathbb{Q}}_{\ell})$ est unipotent dans $\widehat{\textbf{G}}(\overline{\mathbb{Q}}_{\ell})$, et $\varphi$ envoie $W_{k}$ sur des éléments semi-simples de ${}^{L}\textbf{G}(\overline{\mathbb{Q}}_{\ell})$ (un élément de ${}^{L}\textbf{G}(\overline{\mathbb{Q}}_{\ell})$ est semi-simple si sa projection dans $\langle\widehat{\vartheta}\rangle / n\langle\widehat{\vartheta}\rangle\ltimes\widehat{\textbf{G}}(\overline{\mathbb{Q}}_{\ell})$ est semi-simple, où $n$ est l'ordre de $\widehat{\vartheta}$).
\end{enumerate}
\end{Def}

On note alors $\Phi(\textbf{G})$ l'ensemble des morphismes admissibles $\varphi : W'_{k} \rightarrow {}^{L}\textbf{G}(\overline{\mathbb{Q}}_{\ell})$ modulo les automorphismes intérieurs par des éléments de $\widehat{\textbf{G}}(\overline{\mathbb{Q}}_{\ell})$.

\bigskip

Soit $I$ un sous-groupe de $W_{k}$. On note alors $\Phi(I,\textbf{G})$ l'ensemble des classes de $\widehat{\textbf{G}}$-conjugaison des morphismes continus $I \rightarrow {}^{L}\textbf{G}(\overline{\mathbb{Q}}_{\ell})$ (où $\overline{\mathbb{Q}}_{\ell}$ est muni de la topologie discrète) qui admettent une extension à un $L$-morphisme de $\Phi(\textbf{G})$. Dans ce qui suit nous allons nous intéresser principalement aux paramètres de Langlands inertiels $\Phi(I_{k},\textbf{G})$.

\begin{Def}
Si $I$ contient $P_{k}$, l'inertie sauvage, on dit qu'un paramètre $\phi \in \Phi(I,\textbf{G})$ est modéré s'il est trivial sur $P_{k}$, et on note $\Phi_{m}(I,\textbf{G})$ pour l'ensemble des paramètres de $I$ modérés.
\end{Def}

Intéressons nous à $\Phi_{m}(I_k,\textbf{G})$. Comme $I_{k}/P_{k}$ est procyclique de pro-ordre premier à $p$ un morphisme continu $I_{k}/P_{k} \rightarrow \widehat{\textbf{G}}(\overline{\mathbb{Q}}_{\ell})$ est donné par le choix d'un élément $s$ d'ordre fini premier à $p$. Nous avons la décomposition $W_{k}/P_{k}=\langle \text{Frob} \rangle \ltimes (I_{k}/P_{k})$, où pour $x \in (I_{k}/P_{k})$, $\text{Frob }^{-1}x \text{ Frob} = x^q$. Un paramètre de Langlands doit envoyer $\text{Frob}$ sur $\widehat{\vartheta}f$ où $f$ est un élément semi-simple de $\widehat{\textbf{G}}(\overline{\mathbb{Q}}_{\ell})$. Un tel morphisme s'étend donc à $W_{k}/P_{k}$ si $Ad(f) \circ \widehat{\vartheta} \circ s^{q} =s$, où $Ad(f)$ désigne la conjugaison par $f$. Ainsi à un paramètre inertiel modéré 
$\phi \in \Phi_{m}(I_k,\textbf{G})$ on peut associer une classe de conjugaison semi-simple dans $\widehat{\textbf{G}}(\overline{\mathbb{Q}}_{\ell})$ stable sous $\widehat{\vartheta}\circ \psi$ où $\psi$ est l'élévation à la puissance $q$-ième. On a donc une application $\Phi_{m}(I_k,\textbf{G}) \rightarrow (\widehat{\textbf{G}}(\overline{\mathbb{Q}}_{\ell})_{ss})^{\widehat{\vartheta}\circ \psi}$. Or nous savons que $( \widehat{\textbf{T}}(\overline{\mathbb{Q}}_{\ell})/W_{0} ) \overset{\sim}{\longrightarrow}  \widehat{\textbf{G}}(\overline{\mathbb{Q}}_{\ell})_{ss}$, où $W_{0}:=N(\mathbf{T})/ \mathbf{T}$ désigne le groupe de Weyl de $\mathbf{T}$, ce qui nous permet de définir l'application $\Phi_{m}(I_k,\textbf{G}) \rightarrow ( \widehat{\textbf{T}}(\overline{\mathbb{Q}}_{\ell})/W_{0} )^{\widehat{\vartheta}\circ \psi}$. Réciproquement, prenons un élément de $( \widehat{\textbf{T}}(\overline{\mathbb{Q}}_{\ell})/W_{0} )^{\widehat{\vartheta}\circ \psi}$. Ceci nous fournit un élément semi-simple $s$ tel que $\widehat{\vartheta}\circ \psi(s)=w \cdot s$, où $w\in W_{0}$. Soit $f$ un relèvement de $w$, qui est alors un élément semi-simple, on a alors $\widehat{\vartheta}\circ \psi(s)=Ad(f)(s)$, donc on obtient un $\phi \in \Phi_{m}(I_k,\textbf{G})$. Ceci nous montre que l'on a une correspondance 
\[\Phi_{m}(I_k,\textbf{G}) \longleftrightarrow ( \widehat{\textbf{T}}(\overline{\mathbb{Q}}_{\ell})/W_{0} )^{\widehat{\vartheta}\circ \psi}.\]

Soit $s \in( \widehat{\textbf{T}}(\overline{\mathbb{Q}}_{\ell})/W_{0} )^{\widehat{\vartheta}\circ \psi}$, on peut représenter $s$ par $s=(a_{1}, \cdots, a_{n})$, avec $a_{i} \in \overline{\mathbb{Q}}_{\ell}^{\times}$ ($\widehat{\textbf{T}} \simeq \mathbb{G}_{m}^{n}$). Soit $k\in \mathbb{N}$, par définition on a $\widehat{\vartheta}^{k}(s^{p^k})=s$ dans $( \widehat{\textbf{T}}(\overline{\mathbb{Q}}_{\ell})/W_{0})^{\widehat{\vartheta}\circ \psi}$. Comme $\widehat{\vartheta}$ est d'ordre fini, disons $N$, $s^{p^{N}}=s$ dans $( \widehat{\textbf{T}}(\overline{\mathbb{Q}}_{\ell})/W_{0})^{\widehat{\vartheta}\circ \psi}$. Donc il existe $w \in W_{0}$ tel que $(a_{1}^{p^{N}}, \cdots, a_{n}^{p^{N}})=w\cdot (a_{1},\cdots, a_{n})$. Or $W_{0}$ est de cardinal fini, donc il existe $k \in \mathbb{N}^{*}$, tel que $(a_{1}^{p^{kN}}, \cdots, a_{n}^{p^{kN}})=(a_{1}, \cdots, a_{n})$. Ainsi $\forall i \in \{1,\cdots,n\}$, $a_{i}^{p^{kN}-1}=1$. Les $a_{i}$ sont donc des racines $p'$-ièmes de l'unité (racines de l'unité d'ordre premier à $p$).

Notre groupe $\mathbf{G}$ étant $K$-déployé, il possède une forme intérieure non-ramifiée. Cette dernière permet de définir sur $\textbf{\textsf{G}}^{*}$, le groupe dual de $\mathbf{G}$ sur $\mathfrak{F}$, une $\mathfrak{f}$-structure (et donc un Frobenius $F$) en choissant un sommet hyperspécial dans l'immeuble. Le choix d'un système compatible de racines de l'unité (que l'on a fixé au début dans les notations) permet d'identifier 
\[( \widehat{\mathbf{T}}(\overline{\mathbb{Q}}_{\ell})/W_{0} )^{\widehat{\vartheta}\circ \psi} \longleftrightarrow (\textsf{\textbf{T}}^{*}(\mathfrak{F})/W_{0} )^{\widehat{\vartheta}\circ \psi}.\]
(Nous rappelons que $\widehat{\mathbf{T}}$ désigne le dual de $\mathbf{T}$ sur $\overline{\mathbb{Q}}_{\ell}$ et $\textsf{\textbf{T}}^{*}$ celui sur $\mathfrak{F}$. Ainsi $\widehat{\mathbf{T}}(\overline{\mathbb{Q}}_{\ell})=X^{*}(\mathbf{T})\otimes \overline{\mathbb{Q}}_{\ell}^{\times}$ et $\textsf{\textbf{T}}^{*}(\mathfrak{F})=X^{*}(\mathbf{T})\otimes \mathfrak{F}^{\times}$.)

Or l'action de $\widehat{\vartheta}\circ \psi$ sur $\textsf{\textbf{T}}^{*}(\mathfrak{F})$ correspond à l'action du Frobenius $F$ (voir annexe \ref{secFrobenius}, ici $\widehat{\vartheta}=\tau_{X}^{-1}$). Ainsi 
\[( \textsf{\textbf{T}}^{*}(\mathfrak{F})/W_{0} )^{\widehat{\vartheta}\circ \psi} = ( \textsf{\textbf{T}}^{*}(\mathfrak{F})/W_{0} )^{F} \longleftrightarrow (\textsf{\textbf{G}}^{*}(\mathfrak{F})_{ss})^{F}.\]
En résumé nous avons montré

\begin{Pro}
\label{proparaminertie}
La discussion précédente nous fournit une identification :
\[ \Phi_{m}(I_k,\mathbf{G}) \longleftrightarrow (\textnormal{\textsf{\textbf{G}}}^{*}(\mathfrak{F})_{ss})^{F}.\]
\end{Pro}

Dans le but d'étudier les représentations à coefficients dans $\overline{\mathbb{Z}}_{\ell}$ nous avons besoin de restreindre $\Phi_{m}(I_k,\mathbf{G})$. Introduisons $I_{k}^{(\ell)}:=ker\{I_{k}\rightarrow\mathbb{Z}_{\ell}(1)\}$ qui est le sous-groupe fermé maximal de $I_{k}$ de pro-ordre premier à $\ell$. Sous l'identification de la proposition \ref{proparaminertie}, $\Phi_{m}(I_k^{(\ell)},\mathbf{G})$ correspond aux $s \in (\widehat{\textnormal{\textsf{\textbf{G}}}}(\mathfrak{F})_{ss})^{F}$ d'ordre premier à $\ell$.

Pour unifier les notations, notons $I_{k}^{\Lambda}$ qui vaut $I_k$ si $\Lambda=\overline{\mathbb{Q}}_{\ell}$ et $I_{k}^{(\ell)}$ si $\Lambda = \overline{\mathbb{Z}}_{\ell}$. On obtient alors

\begin{Pro}
\label{proIdentificationParamInert}
L'identification de la proposition \ref{proparaminertie} se restreint en :
\[ \Phi_{m}(I_{k}^{\Lambda},\mathbf{G}) \longleftrightarrow \{ s \in (\textnormal{\textsf{\textbf{G}}}^{*}(\mathfrak{F})_{ss})^{F}, s \text{ d'ordre inversible dans } \Lambda\}.\]
\end{Pro}

\subsection{Classes de conjugaison dans les quotients réductifs des groupes parahoriques}

\label{secConjParahoriques}

Nous venons de voir que l'on pouvait identifier les paramètres de l'inertie modérés avec des classes de conjugaison semi-simples dans $\textsf{\textbf{G}}^{*}$. Pour obtenir des systèmes 0-cohérents nous avons besoin de classes de conjugaison dans les quotients réductifs des groupes parahoriques. Nous construisons alors dans cette section un système d'applications compatibles $((\overline{\textsf{\textbf{G}}}^{*}_{\sigma})_{ss})^{F} \rightarrow (\textsf{\textbf{G}}^{*}(\mathfrak{F})_{ss})^{F}$, pour $\sigma \in BT$.

\bigskip

Soit $\mathbf{S}$ un tore déployé maximal, tel que $\sigma \in \mathcal{A}(\mathbf{S},k)$, où $\mathcal{A}(\mathbf{S},k)$ est l'appartement associé à $\textbf{S}$ dans $BT$. Notons $\mathbf{T}$ un $k$-tore maximal $K$-déployé contenant $\mathbf{S}$ (qui existe par \cite{BT} 5.1.12). De plus par \cite{tits_reductive} 2.6.1 $\mathcal{A}(\mathbf{S},k)=BT \cap \mathcal{A}(\mathbf{T},K)$. Notons $\sigma_{1}$ l'image canonique de $\sigma$ dans $BT(K)$.

Notons $W$ le groupe de Weyl affine de $\mathbf{G}(K)$, $W_{0}=N(\mathbf{T})/\mathbf{T}$, le groupe de Weyl de $\mathbf{G}(K)$ et $W_{\sigma_{1}}$ le groupe engendré par les réflexions des hyperplans contenant $\sigma_{1}$ dans $BT(K)$. Nous avons $W=W_{0}\ltimes T/{}^{\circ}T$, où ${}^{\circ}T$ désigne le sous-groupe borné maximal de $T$. De plus $W_{\sigma_{1}}$ est un sous-groupe de $W$, on a donc une application $W_{\sigma_{1}} \rightarrow W \rightarrow W_{0}$. Le noyau du morphisme $W=W_{0}\ltimes T/{}^{\circ}T \rightarrow W_{0}$ est un groupe sans torsion. Or $W_{\sigma_{1}}$ est un groupe fini donc l'application $W_{\sigma_{1}} \rightarrow W_{0}$ est injective et nous permet de voir $W_{\sigma_{1}}$ comme un sous-groupe de $W_{0}$.

\bigskip

Par \cite{tits_reductive} 3.4.3, $\mathcal{G}_{\sigma_{1}}$ est obtenu à partir de $\mathcal{G}_{\sigma}$ par changement de base. En particulier $\overline{\textbf{\textsf{G}}}_{\sigma_{1}}=\overline{\textbf{\textsf{G}}}_{\sigma}\times_{\mathfrak{f}}\mathfrak{F}$.

Le tore $\mathbf{S}$ (resp. $\mathbf{T}$) se prolonge en un tore de $\mathcal{G}_{\sigma}$, $\mathcal{S}_{\sigma}$ (resp. $\mathcal{T}_{\sigma}$), défini sur $\mathfrak{o}_{k}$ de fibre générique $\mathcal{S}_{\sigma,k}=S$ (resp. $\mathcal{T}_{\sigma,k}=T$). Notons $\mathsf{S}_{\sigma}$ (resp $\mathsf{T}_{\sigma}$) la fibre spéciale de $\mathcal{S}_{\sigma}$ (resp $\mathcal{T}_{\sigma}$). Alors $\mathsf{T}_{\sigma}$ est un tore maximal de $\overline{\mathsf{G}}_{\sigma}$ défini sur $\mathfrak{f}$. De plus on a que $\textbf{\textsf{T}}_{\sigma_{1}}=\textbf{\textsf{T}}_{\sigma}\times_{\mathfrak{f}}\mathfrak{F}$.

Le groupe des caractères de $\mathsf{T}_{\sigma_{1}}$, $X^{*}(\textbf{\textsf{T}}_{\sigma_{1}})$, est canoniquement isomorphe à $X=X^{*}(\mathbf{T})$, on les identifiera désormais. De plus, par \cite{tits_reductive} 3.5.1, le groupe de Weyl de $\overline{\mathsf{G}}_{\sigma_{1}}$ associé à $\mathsf{T}_{\sigma_{1}}$ est $W_{\sigma_{1}}$. L'action de $W_{\sigma_{1}}$ sur $X^{*}(\mathsf{T}_{\sigma_{1}})$ coïncide avec l'action de l'image de $W_{\sigma_{1}} \rightarrow W_{0}$ sur $X^{*}(\mathbf{T})$.

\bigskip

On obtient alors :
\[((\overline{\textsf{\textbf{G}}}^{*}_{\sigma})_{ss})^{F} \simeq \left( \textbf{\textsf{T}}^{*}_{\sigma}/W_{\sigma_{1}}\right)^{F} \simeq \left((X\otimes_{\mathbb{Z}}\mathfrak{F}^{\times})/W_{\sigma_{1}}\right)^{F}\]
Le morphisme $W_{\sigma_{1}} \rightarrow W_{0}$ induit
\[ \left((X\otimes_{\mathbb{Z}}\mathfrak{F}^{\times})/W_{\sigma_{1}}\right)^{F} \rightarrow \left((X\otimes_{\mathbb{Z}}\mathfrak{F}^{\times})/W_{0}\right)^{F}.\]
Et de même que précédemment, on a un isomorphisme
\[\left((X\otimes_{\mathbb{Z}}\mathfrak{F}^{\times})/W_{0}\right)^{F} \simeq (\textsf{\textbf{G}}^{*}(\mathfrak{F})_{ss})^{F}.\]
On vient donc de construire une application
\[ \tilde{\psi}_{\sigma} : ((\overline{\textsf{\textbf{G}}}^{*}_{\sigma})_{ss})^{F} \rightarrow (\textsf{\textbf{G}}^{*}(\mathfrak{F})_{ss})^{F}\]

\begin{Lem}
L'application $\tilde{\psi}_{\sigma}$ est indépendante du choix du tore $\textbf{S}$.
\end{Lem}

\begin{proof}
Soit $\textbf{S}'$ un autre tore déployé maximal tel que $\sigma \in \mathcal{A}(\textbf{S}')$. Nous utiliserons la notation ' pour les éléments se rapportant à  $\textbf{S}'$.

D'après \cite{BT} 4.6.28, $G_{\sigma_1}^{\circ}$ permute transitivement les appartements de $BT(K)$ contenant  $\sigma_1$. Ainsi, $\textbf{T}$ et $\textbf{T}'$ sont conjugués par un élément $g \in G_{\sigma_1}^{\circ}$, c'est à dire $\textbf{T}'=g \textbf{T}g^{-1}$. Comme $\mathbf{T}$ et $\mathbf{T}'$ sont deux $k$-tores, $g$ vérifie que $g^{-1}F(g) \in N(\mathbf{G},\mathbf{T})$, le normalisateur de $\mathbf{T}$ dans $\mathbf{G}$. La conjugaison par $g$, $Ad(g)$, induit alors un isomorphisme $X \longrightarrow X':=X^{*}(\mathbf{T}')$. De plus comme elle envoie $\mathcal{A}(\textbf{T},K)$ sur $\mathcal{A}(\textbf{T}',K)$ et que les morphismes $W_{\sigma_{1}} \longrightarrow W_{0}$ et $W_{\sigma_{1}}' \longrightarrow W_{0}'$ sont définis à partir des racines, on a le diagramme commutatif
\[ \xymatrix{
W_{\sigma_{1}} \ar@{->}[r] \ar@{->}[d]^{Ad(g)} & W_{0} \ar@{->}[d]^{Ad(g)}\\
W_{\sigma_{1}}' \ar@{->}[r] & W_{0}' }\]
et donc le diagramme commutatif
\[ \xymatrix{
\left((X\otimes_{\mathbb{Z}}\mathfrak{F}^{\times})/W_{\sigma_{1}}\right)^{F} \ar@{->}[r] \ar@{->}[d]^{Ad(g)} & \left((X\otimes_{\mathbb{Z}}\mathfrak{F}^{\times})/W_{0}\right)^{F} \ar@{->}[d]^{Ad(g)}\\
\left((X'\otimes_{\mathbb{Z}}\mathfrak{F}^{\times})/W_{\sigma_{1}}'\right)^{F} \ar@{->}[r] & \left((X'\otimes_{\mathbb{Z}}\mathfrak{F}^{\times})/W_{0}'\right)^{F} }\]

L'application $G_{\sigma_1}^{\circ} \rightarrow \overline{\textsf{\textbf{G}}}_{\sigma}$  envoie $g$ sur un élément que l'on note $\overline{g}$. De plus nous savons que l'action par conjugaison par $\overline{g}$ qui envoie $X_{*}(\textbf{\textsf{T}}_{\sigma})$ sur $X_{*}(\textbf{\textsf{T}}'_{\sigma})$ coïncide avec l'action par conjugaison par $g$ qui envoie $X$ sur $X'$.

La conjugaison par $\overline{g} \in \overline{\textsf{\textbf{G}}}_{\sigma}$ d'un coté et par $g \in \mathbf{G}$ de l'autre, induit les deux diagrammes commutatifs suivants (lemme \ref{lemconjugaisondual}) :
\[ \xymatrix{
((\overline{\textsf{\textbf{G}}}^{*}_{\sigma})_{ss})^{F} \ar@{->}[r]^-{\sim} \ar@{=}[d] & \left((X\otimes_{\mathbb{Z}}\mathfrak{F}^{\times})/W_{\sigma_{1}}\right)^{F} \ar@{->}[d]^{Ad(g)}\\
((\overline{\textsf{\textbf{G}}}^{*}_{\sigma})_{ss})^{F} \ar@{->}[r]^-{\sim} & \left((X'\otimes_{\mathbb{Z}}\mathfrak{F}^{\times})/W_{\sigma_{1}}'\right)^{F}}
\xymatrix{
\left((X\otimes_{\mathbb{Z}}\mathfrak{F}^{\times})/W_{0}\right)^{F} \ar@{->}[d]^{Ad(g)} \ar@{->}[r]^-{\sim} &(\textsf{\textbf{G}}^{*}(\mathfrak{F})_{ss})^{F} \ar@{=}[d]\\
\left((X'\otimes_{\mathbb{Z}}\mathfrak{F}^{\times})/W_{0}'\right)^{F} \ar@{->}[r]^-{\sim} &(\textsf{\textbf{G}}^{*}(\mathfrak{F})_{ss})^{F}}\]
On obtient alors que le diagramme suivant commute :
\[ \xymatrix{
((\overline{\textsf{\textbf{G}}}^{*}_{\sigma})_{ss})^{F} \ar@{->}[r]^-{\sim} \ar@{=}[d] & \left((X\otimes_{\mathbb{Z}}\mathfrak{F}^{\times})/W_{\sigma_{1}}\right)^{F} \ar@{->}[r] \ar@{->}[d]^{Ad(g)} & \left((X\otimes_{\mathbb{Z}}\mathfrak{F}^{\times})/W_{0}\right)^{F} \ar@{->}[d]^{Ad(g)} \ar@{->}[r]^-{\sim} &(\textsf{\textbf{G}}^{*}(\mathfrak{F})_{ss})^{F} \ar@{=}[d]\\
((\overline{\textsf{\textbf{G}}}^{*}_{\sigma})_{ss})^{F} \ar@{->}[r]^-{\sim} & \left((X'\otimes_{\mathbb{Z}}\mathfrak{F}^{\times})/W_{\sigma_{1}}'\right)^{F} \ar@{->}[r] & \left((X'\otimes_{\mathbb{Z}}\mathfrak{F}^{\times})/W_{0}'\right)^{F} \ar@{->}[r]^-{\sim} &(\textsf{\textbf{G}}^{*}(\mathfrak{F})_{ss})^{F}}\]
Ce qui nous montre le résultat.
\end{proof}

\subsection{Classes de conjugaison dans un groupe fini}

La partie précédente nous fournit des classes de conjugaison semi-simples géométriques d'un groupe réductif connexe fini. Nous sommes plus intéressé par des classes de conjugaison rationnelles. On rappelle alors ici le lien entre les deux.

\bigskip

Dans cette sous-section $\textsf{\textbf{G}}$ désigne un groupe réductif connexe défini sur $\mathfrak{f}$. Pour un élément semi-simple $x \in \textsf{\textbf{G}}$, on note $[x]$ sa classe de conjugaison, $[x] \in \textsf{\textbf{G}}_{ss}$.

\begin{Lem}
\label{lemclassestable}
Soit $s \in (\textnormal{\textsf{\textbf{G}}}_{ss})^{F}$. Alors il existe $x \in \textnormal{\textsf{G}}:=\textnormal{\textsf{\textbf{G}}}^{F}$ tel que $s=[x]$.
\end{Lem}

\begin{proof}
$s=[y]$ avec $y \in \textsf{\textbf{G}}$. La classe de conjugaison $s$ étant $F$-stable, il existe $g \in \textsf{\textbf{G}}$ tel que $F(y)=g^{-1}yg$. L'application de Lang, $Lan:\textsf{\textbf{G}} \rightarrow \textsf{\textbf{G}}$ définie par $Lan(g)=g^{-1}F(g)$ est surjective d'après \cite{cabanes_enguehard} Théorème 7.1. Ainsi, il existe $h \in \textsf{\textbf{G}}$ tel que $g=h^{-1}F(h)$. Alors
\[ F(hyh^{-1})=F(h)F(y)F(h)^{-1}=F(h)g^{-1}ygF(h)^{-1}=hyh^{-1}\]
Ainsi $x=hyh^{-1}$ convient.
\end{proof}

\begin{Cor}
\label{proclassestable}
L'application $\mathsf{G}_{ss} \twoheadrightarrow (\textnormal{\textsf{\textbf{G}}}_{ss})^{F}$ est surjective.
\end{Cor}

\subsection{Systèmes 0-cohérents de classes de conjugaison associés aux paramètres de l'inertie modérés}

\label{secSystemeInertie}

On met bout à bout les résultats des sous-sections précédentes pour obtenir une application qui à un paramètre inertiel modéré associe un système de classes de conjugaison 0-cohérent.

\bigskip

En composant la proposition \ref{proclassestable}, l'application $\tilde{\psi}_{\sigma}$ et la proposition \ref{proparaminertie}, on obtient une application
\[ \psi_{\sigma} : (\overline{\mathsf{G}}^{*}_{\sigma})_{ss} \longrightarrow ((\overline{\textsf{\textbf{G}}}^{*}_{\sigma})_{ss})^{F} \overset{\tilde{\psi}_{\sigma}}{\longrightarrow} (\textsf{\textbf{G}}^{*}(\mathfrak{F})_{ss})^{F} \overset{\sim}{\longrightarrow} \Phi_{m}(I_k,\textbf{G})\]

\bigskip

Soient $\sigma, \omega \in BT$ tels que $\sigma \leq \omega$. Nous avons vu que $\overline{\mathsf{G}}_{\omega}$ est un Levi de $\overline{\mathsf{G}}_{\sigma}$. Ceci nous donne donc, comme dans la section \ref{sectiondelignelusztig}, une application $\varphi^{*}_{\omega,\sigma}:(\overline{\mathsf{G}}^{*}_{\omega})_{ss} \rightarrow (\overline{\mathsf{G}}^{*}_{\sigma})_{ss}$.

\begin{Lem}
\label{lemcompositionpsi}
Soient $\sigma, \omega \in BT$ tels que $\sigma \leq \omega$. Alors
\[ \psi_{\omega}=\psi_{\sigma} \circ \varphi^{*}_{\omega,\sigma}\]
\end{Lem}

\begin{proof}
$W_{\omega_{1}}$ est le groupe engendré par les réflexions des hyperplans contenant $\omega_{1}$ où $\omega_{1}$ est l'image canonique de $\omega$ dans $BT(K)$. Or $\sigma_{1} \leq \omega_{1}$, donc un hyperplan contenant $\omega_{1}$ contient aussi $\sigma_{1}$ et $W_{\omega_{1}}$ est un sous-groupe de $W_{\sigma_{1}}$. Ainsi le diagramme commutatif 
\[ \xymatrix{
W_{\omega_{1}} \ar@{->}[r] \ar@{->}[d]& W_{0}\\
W_{\sigma_{1}}\ar@{->}[ur]}\]
induit le diagramme commutatif
\[ \xymatrix{
\left((X\otimes_{\mathbb{Z}}\mathfrak{F}^{\times})/W_{\omega_{1}}\right)^{F} \ar@{->}[r] \ar@{->}[d]& \left((X\otimes_{\mathbb{Z}}\mathfrak{F}^{\times})/W_{0}\right)^{F}\\
\left((X\otimes_{\mathbb{Z}}\mathfrak{F}^{\times})/W_{\sigma_{1}}\right)^{F} \ar@{->}[ur]} \]
D'où la commutativité de
\[ \xymatrix{
(\overline{\mathsf{G}}^{*}_{\omega})_{ss} \ar@{->}[r] \ar@{->}[d]^{\varphi^{*}_{\omega,\sigma}}& ((\overline{\textsf{\textbf{G}}}^{*}_{\omega})_{ss})^{F} \ar@{->}[r]^-{\sim} \ar@{->}[d]^{\varphi_{\overline{\textsf{\textbf{G}}}^{*}_{\omega},\overline{\textsf{\textbf{G}}}^{*}_{\sigma}}} & \left((X\otimes_{\mathbb{Z}}\mathfrak{F}^{\times})/W_{\omega_{1}}\right)^{F} \ar@{->}[r] \ar@{->}[d]& \left((X\otimes_{\mathbb{Z}}\mathfrak{F}^{\times})/W_{0}\right)^{F}\\
(\overline{\mathsf{G}}^{*}_{\sigma})_{ss} \ar@{->}[r] & ((\overline{\textsf{\textbf{G}}}^{*}_{\sigma})_{ss})^{F} \ar@{->}[r]^-{\sim} & \left((X\otimes_{\mathbb{Z}}\mathfrak{F}^{\times})/W_{\sigma_{1}}\right)^{F} \ar@{->}[ur]}\]
et on a le résultat voulu.
\end{proof}

Soient $g \in G$ et $\sigma \in BT$, nous avons déjà vu (au début de la section \ref{sectSystemescoherents}) que la conjugaison par $g$ induisait deux applications
\[ \varphi^{*}_{g,\sigma}:(\overline{\textsf{G}}^{*}_{\sigma})_{ss} \longrightarrow (\overline{\textsf{G}}^{*}_{g\sigma})_{ss}\]
\[\boldsymbol{\varphi}^{*}_{g,\sigma}:((\overline{\textsf{\textbf{G}}}^{*}_{\sigma})_{ss})^{F} \longrightarrow ((\overline{\textsf{\textbf{G}}}^{*}_{g\sigma})_{ss})^{F}\]
\begin{Lem}
\label{lempsiconj}
Soient $g \in G$ et $\sigma \in BT$ alors
\[ \psi_{\sigma} = \psi_{g \sigma} \circ \varphi^{*}_{g,\sigma}\]
\end{Lem}

\begin{proof}
Soit $\mathbf{S}$ un tore déployé maximal tel que $\sigma \in \mathcal{A}(\mathbf{S})$. Alors si l'on pose $\mathbf{S}'=Ad(g)(\mathbf{S})$, $\mathbf{S}'$ est un tore déployé maximal tel que $g\sigma \in \mathcal{A}(\mathbf{S}')$. La conjugaison par $g$ induit un isomorphisme de $X$ vers $X'$. Le lemme \ref{lemactiongroupedual} nous donne le diagramme commutatif suivant :

\[ \xymatrix{
((\overline{\textsf{\textbf{G}}}^{*}_{\sigma})_{ss})^{F} \ar@{->}[r]^-{\sim} \ar@{->}[d]^{\boldsymbol{\varphi}^{*}_{g,\sigma}} & \left((X\otimes_{\mathbb{Z}}\mathfrak{F}^{\times})/W_{\sigma_{1}}\right)^{F} \ar@{->}[d]^{Ad(g)}\\
((\overline{\textsf{\textbf{G}}}^{*}_{g\sigma})_{ss})^{F} \ar@{->}[r]^-{\sim} & \left((X'\otimes_{\mathbb{Z}}\mathfrak{F}^{\times})/W_{g\sigma_{1}}\right)^{F}
}\]

La conjugaison par $g$ envoie les racines affines pour $\mathbf{S}$ s'annulant sur $\sigma_{1}$ sur les racines affines pour $\mathbf{S}'$ s'annulant sur $g\sigma_{1}$. On a donc le diagramme commutatif suivant

\[ \xymatrix{
W_{\sigma_{1}} \ar@{->}[r] \ar@{->}[d]^{Ad(g)} & W_{0} \ar@{->}[d]^{Ad(g)}\\
W_{g\sigma_{1}} \ar@{->}[r] & W_{0}'\\
}\]
et
\[ \xymatrix{
\left((X\otimes_{\mathbb{Z}}\mathfrak{F}^{\times})/W_{\sigma_{1}}\right)^{F} \ar@{->}[r] \ar@{->}[d]^{Ad(g)} & \left((X\otimes_{\mathbb{Z}}\mathfrak{F}^{\times})/W_{0}\right)^{F} \ar@{->}[d]^{Ad(g)}\\
\left((X'\otimes_{\mathbb{Z}}\mathfrak{F}^{\times})/W_{g\sigma_{1}}\right)^{F} \ar@{->}[r] & \left((X'\otimes_{\mathbb{Z}}\mathfrak{F}^{\times})/W_{0}'\right)^{F}\\
}\]

Enfin la conjugaison par $g$ étant un isomorphisme intérieur sur $G$, le diagramme ci-dessous commute (lemme \ref{lemconjugaisondual})

\[ \xymatrix{
\left((X\otimes_{\mathbb{Z}}\mathfrak{F}^{\times})/W_{0}\right)^{F} \ar@{->}[d]^{Ad(g)} \ar@{->}[r]^-{\sim} &(\textsf{\textbf{G}}^{*}(\mathfrak{F})_{ss})^{F} \ar@{=}[d]\\
\left((X'\otimes_{\mathbb{Z}}\mathfrak{F}^{\times})/W_{0}'\right)^{F} \ar@{->}[r]^-{\sim} &(\textsf{\textbf{G}}^{*}(\mathfrak{F})_{ss})^{F}}\]

Mis bout à bout ces diagrammes donnent la commutativité de
\[
\xymatrix{
(\overline{\textsf{G}}^{*}_{\sigma})_{ss} \ar@{->}[d] \ar@{->}[r]^{\varphi^{*}_{g,\sigma}} & (\overline{\textsf{G}}^{*}_{g\sigma})_{ss} \ar@{->}[d] \\
((\overline{\textsf{\textbf{G}}}^{*}_{\sigma})_{ss})^{F} \ar@{->}[d]^-{\sim} \ar@{->}[r]^{\boldsymbol{\varphi}^{*}_{g,\sigma}} &((\overline{\textsf{\textbf{G}}}^{*}_{g\sigma})_{ss})^{F} \ar@{->}[d]^-{\sim} \\
\left((X\otimes_{\mathbb{Z}}\mathfrak{F}^{\times})/W_{\sigma_{1}}\right)^{F} \ar@{->}[d] \ar@{->}[r]^{Ad(g)} &\left((X'\otimes_{\mathbb{Z}}\mathfrak{F}^{\times})/W_{g\sigma_{1}}\right)^{F} \ar@{->}[d]\\
\left((X\otimes_{\mathbb{Z}}\mathfrak{F}^{\times})/W_{0}\right)^{F} \ar@{->}[r]^{Ad(g)} \ar@{->}[d]^-{\sim} &\left((X'\otimes_{\mathbb{Z}}\mathfrak{F}^{\times})/W_{0}'\right)^{F} \ar@{->}[d]^-{\sim}\\
(\textsf{\textbf{G}}^{*}(\mathfrak{F})_{ss})^{F} \ar@{=}[r] & (\textsf{\textbf{G}}^{*}(\mathfrak{F})_{ss})^{F}
}\]
ce qui finit la preuve.
\end{proof}

Construisons maintenant un système 0-cohérent de classes de conjugaison.

\begin{Def}
Soient $\phi \in \Phi_{m}(I_{k}^{\Lambda},\mathbf{G})$ et $\sigma \in BT$. On définit le système de classes de conjugaison $S_{\phi}=(S_{\phi,\sigma})_{\sigma \in BT}$ par
\[ S_{\phi,\sigma}= \psi_{\sigma}^{-1}(\phi)\]
\end{Def}

\begin{Pro}
\label{prosystphi}
Soit $\phi \in \Phi_{m}(I_{k}^{\Lambda},\mathbf{G})$. Le système $S_{\phi}$ est 0-cohérent.
\end{Pro}

\begin{proof}

La condition 1. de \ref{defsystemsurcoherentconju} est vérifiée par \ref{lempsiconj} et la condition 2. par \ref{lemcompositionpsi}.
\end{proof}

Ainsi par la proposition \ref{prodecompocategorie}, si l'on note $Rep_{\Lambda}^{\phi}(G):=Rep_{\Lambda}^{S_{\phi}}(G)$, alors

\begin{The}
\label{thmDecompoInertiel}
Soit $\mathbf{G}$ un groupe réductif connexe défini sur $k$ et $K$-déployé. Alors la catégorie de niveau $0$ se décompose en 

\[ Rep_{\Lambda}^{0}(G) = \prod_{\phi \in \Phi_{m}(I_{k}^{\Lambda},\mathbf{G})} Rep_{\Lambda}^{\phi}(G) \]
\end{The}

Notons que si $\mathbf{G}$ est quasi-déployé alors il est non-ramifié et possède donc un sommet hyperspécial $o$. Dans ce cas, l'application $\tilde{\psi}_{o}$ est bijective, donc $\psi_{o}$ est surjective et $Rep_{\Lambda}^{\phi}(G)$ est non vide pour tout $\phi \in \Phi_{m}(I_{k}^{\Lambda},\mathbf{G})$. Cependant, lorsque $\mathbf{G}$ n'est pas quasi-déployé, les catégories $Rep_{\Lambda}^{\phi}(G)$ peuvent être vides. Nous devons rajouter une condition de "relevance" pour avoir $Rep_{\Lambda}^{\phi}(G)$ non vide, ce que nous détaillerons dans la partie \ref{secRelevance}.

\section{Propriétés de $Rep_{\Lambda}^{\phi}(G)$}

\label{secProprieteRepPhi}

Fixons dans toute cette section un paramètre inertiel modéré $\phi \in \Phi_{m}(I_{k}^{\Lambda},\mathbf{G})$. Le but de cette section est d'étudier quelques propriétés vérifiées par $Rep_{\Lambda}^{\phi}(G)$. Rappelons qu'à $\phi$ nous avons associé dans la partie \ref{secSystemeInertie} un système 0-cohérent de classes de conjugaison $S_{\phi}$, qui permet de définir $e_{\phi}=(e_{\phi,x})_{x \in BT_{0}}$ un système  0-cohérent d'idempotents défini par $e_{\phi,x}=\sum_{s \in S_{\phi,x}} e_{x}^{s,\Lambda}$.

\subsection{Lien entre les décompositions sur $\overline{\mathbb{Z}}_{\ell}$ et $\overline{\mathbb{Q}}_{\ell}$}

\label{secLienLambda}

Au vu de la construction de $Rep_{\Lambda}^{\phi}(G)$ il est assez simple de comprendre le lien entre $\Lambda = \overline{\mathbb{Z}}_{\ell}$ et $\Lambda=\overline{\mathbb{Q}}_{\ell}$ ce que nous faisons ici.

\bigskip

Considérons ici que $\phi \in \Phi_{m}(I_{k}^{(\ell)},\mathbf{G})$. Soit $x \in BT_{0}$ et notons $S'_{\phi,x}$ l'ensemble des $s' \in (\overline{\mathsf{G}}^{*}_{x})_{ss}$ dont $s$ la partie $\ell$-régulière de $s'$ est dans $S_{\phi,x}$. Alors par construction, $e_{\phi,x}=\sum_{s \in S_{\phi,x}} e_{x}^{s,\overline{\mathbb{Z}}_{\ell}}=\sum_{s'\in S'_{\phi,x}} e_{x}^{s,\overline{\mathbb{Q}}_{\ell}}$. Prenons $s' \in (\overline{\mathsf{G}}^{*}_{x})_{ss}$ et nommons $\phi' \in \Phi_{m}(I_k,\mathbf{G})$ le paramètre inertiel qui lui est associé, c'est à dire $\phi':=\psi_{x}(s')$. Soit $s \in S_{\phi,x}$ (donc $\psi_{x}(s)=\phi$), $s$ est la partie $\ell$-régulière de $s'$ si et seulement si $\phi'_{\mid I_{k}^{(\ell)}} \sim \phi$. Le lien entre les décompositions sur $\overline{\mathbb{Z}}_{\ell}$ et $\overline{\mathbb{Q}}_{\ell}$ est alors clair

\begin{Pro}
\label{proLienLambda}
Soit $\phi \in \Phi_{m}(I_{k}^{(\ell)},\mathbf{G})$, alors
\[ Rep_{\overline{\mathbb{Z}}_{\ell}}^{\phi}(G) \cap Rep_{\overline{\mathbb{Q}}_{\ell}}(G) = \prod_{\phi'} Rep_{\overline{\mathbb{Q}}_{\ell}}^{\phi'}(G)\]
où le produit est pris sur les $\phi' \in \Phi_{m}(I_k,\mathbf{G})$ tels que $\phi'_{\mid I_{k}^{(\ell)}} \sim \phi$.
\end{Pro}

\subsection{Représentations irréductibles de $Rep_{\Lambda}^{\phi}(G)$ }

\label{secRepIrr}
Nous souhaitons dans cette partie décrire les représentations irréductibles qui sont dans $Rep_{\Lambda}^{\phi}(G)$.

\bigskip

Soit $\mathbf{T}$ un tore maximal non-ramifié de $\mathbf{G}$ ($\mathbf{T}$ est un $k$-tore $K$-déployé maximal de $\mathbf{G}$). Nommons $\mathbf{T}_{0}$ le tore de référence utilisé pour définir $\widehat{\vartheta}$ et $\widehat{\mathbf{G}}$. Le tore $\mathbf{T}$ étant non-ramifié il existe $g \in G^{nr}$ tel que $T^{nr}={}^{g}T_{0}^{nr}$. Dans ce cas $g^{-1}F(g) \in N(T_{0}^{nr},G^{nr})$ et définit un élément $w \in W_{0}$. Ainsi ${}^{L}\mathbf{T} \simeq \langle w\widehat{\vartheta} \rangle \ltimes \widehat{\mathbf{T}}_{0}(\overline{\mathbb{Q}}_{\ell})$. Le choix d'un relèvement $\dot{w} \in N(\widehat{\mathbf{T}}_{0},\widehat{\mathbf{G}})$ de $w$ permet alors de définir un plongement ${}^{L}\mathbf{T} \hookrightarrow {}^{L}\mathbf{G}$ par $\widehat{\mathbf{T}}_{0}(\overline{\mathbb{Q}}_{\ell}) \subseteq \widehat{\mathbf{G}}(\overline{\mathbb{Q}}_{\ell})$ et $ w\widehat{\vartheta} \mapsto (\dot{w},\widehat{\vartheta})$. Ce plongement dépend (même à $\widehat{\mathbf{G}}(\overline{\mathbb{Q}}_{\ell})$-conjugaison près) du choix du relèvement de $w$. Il induit cependant une application
\[\iota : \Phi_{m}(I_k,\mathbf{T}) \rightarrow\Phi_{m}(I_k,\mathbf{G})\]
qui elle est indépendante des choix effectués car les paramètres inertiels sont à valeurs dans $\widehat{\mathbf{T}}_{0}(\overline{\mathbb{Q}}_{\ell})$ (ou $\widehat{\mathbf{G}}(\overline{\mathbb{Q}}_{\ell})$).

Soit $\phi_{\mathbf{T}} \in \Phi_{m}(I_k,\mathbf{T})$. Notons $X:=X^{*}(\mathbf{T})$. Nous avons vu dans les sections \ref{secConjParahoriques} et \ref{secParamInertClasseConj} que l'on a une bijection $\Phi_{m}(I_k,\mathbf{T}) \simeq (X \otimes_{\mathbb{Z}} \mathfrak{F}^{\times})^{F}$. Soit $x \in \mathcal{A}(\mathbf{T},K) \cap BT_{0}$. Nous savons que l'on a également un isomorphisme $(X \otimes_{\mathbb{Z}} \mathfrak{F}^{\times})^{F} \simeq (\textsf{\textbf{T}}_{x}^{*})^{F} \simeq Hom(\textbf{\textsf{T}}_{x}^{F},\overline{\mathbb{Q}}_{\ell}^{\times})$. On associe donc à $\phi_{\mathbf{T}}$ de manière bijective un caractère $\theta_{\mathbf{T}}:\textbf{\textsf{T}}_{x}^{F} \rightarrow \overline{\mathbb{Q}}_{\ell}^{\times}$, qui se relève en un caractère de niveau 0 : $\theta_{\mathbf{T}}:{}^{0}\textbf{T}^{F} \rightarrow \overline{\mathbb{Q}}_{\ell}^{\times}$. L'association qui à $\phi_{\mathbf{T}}$ donne $\theta_{\mathbf{T}}$ est alors la correspondance de Langlands locale pour les tores restreinte à l'inertie.

\begin{The}
\label{theRepIrr}
Soit $\pi \in \text{Irr}_{\overline{\mathbb{Q}}_{\ell}}(G)$. Alors $\pi \in Rep_{\overline{\mathbb{Q}}_{\ell}}^{\phi}(G)$ si et seulement s'il existe $\mathbf{T}$ un tore maximal non-ramifié de $\mathbf{G}$, $\phi_{\mathbf{T}} \in \Phi_{m}(I_k,\mathbf{T})$ et $x \in \mathcal{A}(\mathbf{T},K) \cap BT_{0}$ tels que $\iota(\phi_{\mathbf{T}}) \sim \phi$ et $\langle \pi^{G_{x}^{+}}, \mathcal{R}_{\textbf{\textsf{T}}_{x}}^{\overline{\textsf{\textbf{G}}}_{x}}(\theta_{\mathbf{T}}) \rangle \neq 0$ (où $\pi^{G_{x}^{+}}$ est vue comme une représentation de $\overline{\textsf{G}}_{x} \simeq G_{x}^{\circ}/G_{x}^{+}$ et $\mathcal{R}_{\textbf{\textsf{T}}_{x}}^{\overline{\textsf{\textbf{G}}}_{x}}$ désigne l'induction de Deligne-Lusztig).
\end{The}

\begin{proof}
Par définition de $Rep_{\overline{\mathbb{Q}}_{\ell}}^{\phi}(G)$, comme $\pi$ est une représentation irré\-ductible, $\pi \in Rep_{\overline{\mathbb{Q}}_{\ell}}^{\phi}(G)$ si et seulement s'il existe $x \in BT_{0}$ tel que $e_{\phi,x} \pi^{G_{x}^{+}} \neq 0$. Soit $x \in BT_{0}$, alors par construction $e_{\phi,x} \pi^{G_{x}^{+}} \neq 0$ est équivalent à l'existence d'une classe de conjugaison rationnelle semi-simple $s\in S_{\phi,x}$ telle que $e_{s,\overline{\mathbb{Q}}_{\ell}}^{\overline{\textsf{G}}_{x}} \pi^{G_{x}^{+}} \neq 0$. Soit $\textsf{\textbf{T}}_{x}$ un $\mathfrak{f}$-tore maximal de $\overline{\textsf{\textbf{G}}}_{x}$ tel que $s \in (\textsf{\textbf{T}}_{x}^{*})^{F}$. Relevons $\textsf{\textbf{T}}_{x}$ en $\mathbf{T}$ un tore maximal non-ramifié de $\mathbf{G}$. Nous avons que $\Phi_{m}(I_k,\mathbf{T}) \simeq (X \otimes_{\mathbb{Z}} \mathfrak{F}^{\times})^{F} \simeq (\textsf{\textbf{T}}_{x}^{*})^{F} $ et donc $s$ correspond à $\phi_{\mathbf{T}}$ un paramètre inertiel modéré de $\mathbf{T}$. La discussion qui précède le théorème montre que $s$ est également associé au caractère $\theta_{\mathbf{T}}:\textbf{\textsf{T}}_{x}^{F} \rightarrow \overline{\mathbb{Q}}_{\ell}^{\times}$. La section \ref{secDeligneLusztig} nous dit alors que $e_{s,\overline{\mathbb{Q}}_{\ell}}^{\overline{\textsf{G}}_{x}} \pi^{G_{x}^{+}} \neq 0$ si et seulement si $\langle \pi^{G_{x}^{+}}, \mathcal{R}_{\textbf{\textsf{T}}_{x}}^{\overline{\textsf{\textbf{G}}}_{x}}(\theta_{\mathbf{T}}) \rangle \neq 0$.

On vient donc de montrer que $\pi \in Rep_{\overline{\mathbb{Q}}_{\ell}}^{\phi}(G)$ si et seulement s'il existe $\mathbf{T}$ un tore maximal non-ramifié de $\mathbf{G}$,  $x \in \mathcal{A}(\mathbf{T},K) \cap BT_{0}$, $\phi_{\mathbf{T}} \in \Phi_{m}(I_k,\mathbf{T})$ correspondant à $s \in S_{\phi,x}$ tels que $\langle \pi^{G_{x}^{+}}, \mathcal{R}_{\textbf{\textsf{T}}_{x}}^{\overline{\textsf{\textbf{G}}}_{x}}(\theta_{\mathbf{T}})\rangle \neq 0$. Pour achever la preuve du théorème il ne nous reste donc qu'à montrer que $\phi = \iota(\phi_{\mathbf{T}})$. Or cela découle du diagramme commutatif suivant

\[ \xymatrix{
(\textsf{\textbf{T}}^{*}_{x})^{F} \ar@{->}[d] \ar@{->}[r]  &(\overline{\textsf{G}}^{*}_{x})_{ss} \ar@{->}[d]\\
(X\otimes_{\mathbb{Z}}\mathfrak{F}^{\times})^{F} \ar@{->}[d] \ar@{->}[r] &\left((X\otimes_{\mathbb{Z}}\mathfrak{F}^{\times})/W_{x}\right)^{F} \ar@{->}[d]\\
(X\otimes_{\mathbb{Z}}\mathfrak{F}^{\times})^{F} \ar@{->}[r] \ar@{->}[d]^-{\sim} &\left((X\otimes_{\mathbb{Z}}\mathfrak{F}^{\times})/W_{0}\right)^{F} \ar@{->}[d]^-{\sim}\\
(\textsf{\textbf{T}}^{*}(\mathfrak{F}))^{F} \ar@{->}[r] \ar@{->}[d]^-{\sim} &(\textsf{\textbf{G}}^{*}(\mathfrak{F})_{ss})^{F} \ar@{->}[d]^-{\sim}\\
\Phi_{m}(I_k,\mathbf{T}) \ar@{->}[r]^{\iota} &\Phi_{m}(I_k,\mathbf{G})}\]
\end{proof}

Notons que le théorème précédent n'est énoncé que pour $\Lambda=\overline{\mathbb{Q}}_{\ell}$ puisque l'on peut en déduire une description de $\text{Irr}_{\overline{\mathbb{Q}}_{\ell}}(G) \cap Rep_{\overline{\mathbb{Z}}_{\ell}}^{\phi}(G)$ grâce à la proposition \ref{proLienLambda}. Notons également que pour $\Lambda=\overline{\mathbb{Z}}_{\ell}$, les objets simples de $Rep_{\overline{\mathbb{Z}}_{\ell}}^{\phi}(G)$ sont
\begin{enumerate}
\item Les objets simples de caractéristique 0 qui sont les $\pi \in \text{Irr}_{\overline{\mathbb{Q}}_{\ell}}(G) \cap Rep_{\overline{\mathbb{Z}}_{\ell}}^{\phi}(G)$ qui ne sont pas entières.
\item Les objets simples de caractéristique $\ell$ qui sont les sous-quotients simples des réductions modulo $\ell$ des $\pi \in \text{Irr}_{\overline{\mathbb{Q}}_{\ell}}(G) \cap Rep_{\overline{\mathbb{Z}}_{\ell}}^{\phi}(G)$ qui sont entières (voir le lemme 6.8 de \cite{dattempered}, les hypothèses peuvent être supprimées ici car on est en niveau 0).
\end{enumerate}

\subsection{Condition de relevance}

\label{secRelevance}

Nous avons noté précédemment que si $\mathbf{G}$ n'est pas quasi-déployé alors les catégories $Rep_{\Lambda}^{\phi}(G)$ peuvent être vides. Nous allons montrer dans cette partie que $Rep_{\Lambda}^{\phi}(G)$ est non vide si et seulement si $\phi$ est relevant, au sens suivant

\begin{Def}
Soit $\phi \in \Phi(I_{k}^{\Lambda},\mathbf{G})$ un paramètre inertiel. On dit que $\phi$ est relevant s'il existe $\varphi' \in \Phi(\mathbf{G})$ une extension de $\phi$ à $W_{k}'$ qui est relevant, c'est à dire que si l'image de $\varphi'$ est contenue dans un Levi de ${}^{L}\mathbf{G}$ alors ce dernier est relevant (au sens de \cite{borel} 3.4).
\end{Def}

Soit $\varphi \in \Phi(W_k,\mathbf{G})$ et posons $\phi:=\varphi_{|I_{k}^{\Lambda}}$. Pour $w \in W_k$, l'action par conjugaison de $\varphi(w)$ normalise $\phi(I_{k}^{\Lambda})$ donc normalise également $C_{\widehat{\mathbf{G}}}(\phi)^{\circ}$, le centralisateur connexe de l'image de $\phi$ dans $\widehat{\mathbf{G}}$. On définit alors 
\[\mathcal{M}_{\varphi}:=C_{{}^{L}\mathbf{G}}(Z(C_{\widehat{\mathbf{G}}}(\phi)^{\circ})^{\varphi(W_k),\circ})\]
qui est un Levi de ${}^{L}\mathbf{G}$ dont la partie connexe est $M_{\varphi}:=C_{\widehat{\mathbf{G}}}(Z(C_{\widehat{\mathbf{G}}}(\phi)^{\circ})^{\varphi(W_k),\circ})$.

\begin{Lem}
\label{lemFactoVarphi}
Soit $\varphi \in \Phi(W_k,\mathbf{G})$. Alors toute extension $\varphi' \in \Phi(\mathbf{G})$ de $\varphi$ à $W_{k}'$ se factorise par $\mathcal{M}_{\varphi}$. De plus il existe un $\varphi' \in \Phi(\mathbf{G})$ étendant $\phi:=\varphi_{|I_{k}}$ ne se factorisant par aucun sous Levi propre de $\mathcal{M}_{\varphi}$.
\end{Lem}

\begin{proof}
Ici, on écrira plutôt $W_k'$ sous la forme $W_k \times SL_2$. On prendra garde cependant à prendre la bonne "restriction" de $W_k \times SL_2$ à $W_k$ qui est donnée par le plongement $W_k \hookrightarrow W_k \times SL_2$, $w \mapsto (w,diag(|w|^{1/2},|w|^{-1/2}))$. Néanmoins, cela ne fait pas de différence lorsque l'on prend les restrictions à l'inertie.

Prenons $\varphi' \in \Phi(G)$ une extension de $\varphi$. Par définition $\mathcal{M}_{\varphi}$ contient $\varphi(W_k)$. De plus, $\varphi'(SL_{2})$ est contenue dans $C_{\widehat{\mathbf{G}}}(\phi)^{\circ}$ donc $\varphi'(SL_{2}) \subseteq \mathcal{M}_{\varphi}$ et par conséquent $\varphi'(W_{k}') \subseteq \mathcal{M}_{\varphi}$.

\medskip

Construisons maintenant un $\varphi'$ ne se factorisant par aucun sous Levi propre de $\mathcal{M}_{\varphi}$. Un Levi minimal de $\mathcal{M}_{\varphi}$ factorisant $\varphi'$ est obtenu en prenant le centralisateur dans $\mathcal{M}_{\varphi}$ d'un tore maximal de $C_{M_{\varphi}}(\varphi')^{\circ}$. Ainsi pour prouver la propriété demandée, il nous suffit de fabriquer un $\varphi'$ tel que $C_{M_{\varphi}}(\varphi')^{\circ} \subseteq Z(\mathcal{M}_{\varphi})$.

Soit $\phi \in \Phi(I_k,\mathbf{G})$. Prenons $\varphi$ étendant $\phi$ tel que l'automorphisme semi-simple $\theta$ de conjugaison par $\varphi(\text{Frob})$ préserve un épinglage $(C_{\widehat{\mathbf{G}}}(\phi)^{\circ},\mathbf{B},\mathbf{T},\{x_{\alpha}\}_{\alpha \in \Delta})$ de $C_{\widehat{\mathbf{G}}}(\phi)^{\circ}$. On définit $\varphi'_{|W_k}=\varphi$ (ici on considère la restriction naïve de $W_k \times SL_2$ à $W_k$) et $\varphi'_{|SL_2}:SL_{2} \rightarrow C_{\widehat{\mathbf{G}}}(\varphi)^{\circ}$ le morphisme principal de $SL_{2}$ à valeur dans $C_{\widehat{\mathbf{G}}}(\varphi)^{\circ}$ associé à l'épinglage choisi. Nous avons alors que $C_{\widehat{\mathbf{G}}}(\varphi')^{\circ}=C_{ C_{\widehat{\mathbf{G}}}(\varphi)^{\circ} }(\varphi'_{|SL_2})^{\circ} = Z(C_{\widehat{\mathbf{G}}}(\varphi)^{\circ})^{\circ}$. Pour achever la preuve il ne reste donc qu'à montrer que $Z(C_{\widehat{\mathbf{G}}}(\varphi)^{\circ})^{\circ} = Z(C_{\widehat{\mathbf{G}}}(\phi)^{\circ})^{\varphi(W_k),\circ}$. En effet, on aura alors le résultat voulu puisque $C_{M_{\varphi}}(\varphi')^{\circ} \subseteq C_{\widehat{\mathbf{G}}}(\varphi')^{\circ} = Z(C_{\widehat{\mathbf{G}}}(\phi)^{\circ})^{\varphi(W_k),\circ} \subseteq Z(\mathcal{M}_{\varphi})$.

Notons que $C_{\widehat{\mathbf{G}}}(\phi)=C_{\widehat{\mathbf{G}}}(\varphi)^{\theta}$. Pour simplifier les notations on pose $H=C_{\widehat{\mathbf{G}}}(\phi)$. Il nous reste donc à prouver que $Z(H^{\theta,\circ})^{\circ}=Z(H^{\circ})^{\theta,\circ}$. Calculons les centres ici présents. Nous avons que $Z(H^{\circ}) = \cap_{\alpha \in \Delta} \ker(\alpha)$ et par conséquent $Z(H^{\circ})^{\theta,\circ} = ((\cap_{\alpha \in \Delta} \ker(\alpha)) \cap T^{\theta,\circ})^{\circ}$. Comme $\theta$ préserve un épinglage, on a également grâce au théorème 1.8 (v) de \cite{DigneMichelgroupe}, $Z(H^{\theta,\circ}) = \cap_{\alpha \in \Delta/\theta} \ker(\alpha_{|T^{\theta,\circ}}) = \cap_{\alpha \in \Delta/\theta} \ker(\alpha) \cap T^{\theta,\circ}$, d'où le résultat voulu.
\end{proof}

On appelle tore maximal de ${}^{L}\mathbf{G}$ un sous-groupe $\mathcal{T}$ de ${}^{L}\mathbf{G}$ qui se surjecte sur $\langle \widehat{\vartheta} \rangle$ et dont l'intersection $\mathcal{T}^{\circ}$ avec $\widehat{\mathbf{G}}$ est un tore maximal de $\widehat{\mathbf{G}}$. Pour un tel tore, on notera $\widehat{\mathbf{T}}:=\mathcal{T}^{\circ}$ sa partie connexe. Nous avons la  suite exacte suivante : $\widehat{\mathbf{T}} \hookrightarrow \mathcal{T} \twoheadrightarrow \langle \widehat{\vartheta} \rangle$. Le tore $\mathcal{T}$ agit par conjugaison sur $\widehat{\mathbf{T}}$ et donc, on en déduit une action de $\langle \widehat{\vartheta} \rangle$ sur $\widehat{\mathbf{T}}$ qui nous permet de définir une $k$-structure sur $\mathbf{T}$ le dual de $\widehat{\mathbf{T}}$. On dira qu'un tore maximal $\mathcal{T}$ est relevant si le plongement $\widehat{\mathbf{T}} \hookrightarrow \widehat{\mathbf{G}}$ correspond dualement à un $k$-plongement $\mathbf{T} \hookrightarrow \mathbf{G}$. Enfin, on dira que $\mathcal{T}$ est elliptique dans ${}^{L}\mathbf{G}$ si $\mathcal{T}$ n'est contenu dans aucun Levi propre $\mathcal{M}$ de ${}^{L}\mathbf{G}$ ou de façon équivalente si $Z(\mathcal{T})^{\circ}=Z({}^{L}\mathbf{G})^{\circ}$.


\begin{Lem}
\label{lemLeviRelevant}
Soit $\mathcal{M}$ un Levi de ${}^{L}\mathbf{G}$.

\begin{enumerate}
\item Si $\mathcal{M}$ contient $\mathcal{T}$, un tore maximal relevant, alors $\mathcal{M}$ est relevant.
\item Si $\mathcal{M}$ est relevant et $\mathcal{T}$ est un tore maximal elliptique de $\mathcal{M}$ alors $\mathcal{T}$ est relevant.
\end{enumerate}

\end{Lem}

\begin{proof}
\begin{enumerate}
\item (Dat) Notons $(\mathcal{M}^{\circ})_{ab}$ l'abélianisé de $\mathcal{M}^{\circ}$ qui est un tore. Le groupe $\mathcal{M}$ agit par conjugaison sur $\mathcal{M}^{\circ}$ donc sur $(\mathcal{M}^{\circ})_{ab}$. Cette action est triviale sur $\mathcal{M}^{\circ}$ donc nous donne une action de $\langle \widehat{\vartheta} \rangle$ sur $(\mathcal{M}^{\circ})_{ab}$. Le plongement $\mathcal{T}^{\circ} \rightarrow \mathcal{M}^{\circ}$ induit alors un morphisme $\langle \widehat{\vartheta} \rangle$-équivariant $\mathcal{T}^{\circ} \rightarrow (\mathcal{M}^{\circ})_{ab}$. On obtient ainsi dualement un $k$-plongement $\mathbf{S} \hookrightarrow \mathbf{T}$. Si $\mathcal{T}$ est relevant, on peut choisir un plongement $\mathbf{T} \hookrightarrow \mathbf{G}$ rationnel, et alors $C_{\mathbf{G}}(\mathbf{S})$ est un Levi rationnel, dual de $\mathcal{M}$, qui est donc relevant.

\item 

Le tore $\mathcal{T}$ de $\mathcal{M}$ nous fournit dualement un plongement $\mathbf{T} \hookrightarrow \mathbf{M}_{qd}$, où $\mathbf{M}_{qd}$ est la forme quasi-déployée de $\mathbf{M}$, un $k$-sous-groupe de Levi de $\mathbf{G}$ dual de $\widehat{\mathbf{M}}$. Comme $\mathcal{T}$ est elliptique, le rang déployé de $\mathbf{T}$ est le même que celui du centre de $\mathbf{M}_{qd}$, et par conséquent $\mathbf{T}$ est elliptique. Ainsi, il se plonge dans toutes les formes intérieures de $\mathbf{M}_{qd}$ (voir par exemple \cite{kaletha} lemme 3.2.1) donc en particulier dans $\mathbf{M}$ et donc $\mathcal{T}$ est relevant.
\end{enumerate}
\end{proof}

\begin{Lem}
\label{lemToreElliptic}
Soit $\varphi \in \Phi(W_k,\mathbf{G})$. Notons $\phi=\varphi_{|I_{k}}$ et posons $\mathcal{C}:=C_{\widehat{\mathbf{G}}}(\phi)^{\circ}\varphi(W_k)$ qui est un sous-groupe de ${}^{L}\mathbf{G}$. Alors il existe $\mathcal{T} \subseteq \mathcal{C}$ un sous-tore maximal de $\mathcal{C}$ tel que $Z(\mathcal{T})^{\circ}=Z(\mathcal{C})^{\circ}$.
\end{Lem}

\begin{proof}
Notons $\theta$ la conjugaison par $\varphi(\text{Frob})$ qui est un automorphisme semi-simple de $C:=C_{\widehat{\mathbf{G}}}(\phi)^{\circ}$. Prenons alors $(\widehat{\mathbf{T}},\widehat{\mathbf{B}})$ une paire de Borel de  $C_{\widehat{\mathbf{G}}}(\phi)^{\circ}$ stable sous $\theta$ (existe par \cite{steinberg} théorème 7.5). Nous pouvons écrire $\mathcal{C}$ sous la forme $\mathcal{C}:=C \rtimes \langle \theta \rangle$. On forme alors $\mathcal{T}:=\widehat{\mathbf{T}} \rtimes \langle \theta \rangle$ qui est un tore maximal de $\mathcal{C}$. Nous allons modifier $\mathcal{T}$ pour le rendre elliptique. Toute section $\eta$ de la suite exacte $N_{C}(\widehat{\mathbf{T}}) \hookrightarrow N_{\mathcal{C}}(\widehat{\mathbf{T}}) \twoheadrightarrow \langle \theta \rangle$ (la notation $N_{C}(\widehat{\mathbf{T}})$ signifie le normalisateur de $\widehat{\mathbf{T}}$ dans $C$ et idem pour $N_{\mathcal{C}}(\widehat{\mathbf{T}})$ avec $\mathcal{C}$) nous permet de définir un tore maximal $\mathcal{T}_{\eta}$ par $\mathcal{T}_{\eta}:=\widehat{\mathbf{T}} \cdot \eta(\langle \theta \rangle)$. Une section $\eta$ est donnée par $\eta(\theta)=n_{\theta}\theta$ où $n_{\theta} \in N_{C}(\widehat{\mathbf{T}})$. On prend alors pour $n_{\theta}$ un élément de $\theta$-coxeter, c'est à dire un élément du groupe de Weyl formé en prenant un produit de réflexions simples, une pour chaque orbite sous $\theta$. Il découle alors du lemme 7.4 (i) de \cite{springerRegular} que le tore que l'on obtient est elliptique, ce qui achève la preuve.
\end{proof}

Rappelons nous que dans la section \ref{secRepIrr} nous avons fixé un tore maximal de référence $\mathbf{T}_{0}$, qui nous a permis d'associer à $\mathbf{T}$, un tore maximal non-ramifié de $\mathbf{G}$, un élément $w \in W_0$, un tore ${}^{L}\mathbf{T}=\widehat{\mathbf{T}}_{0} \rtimes \langle w \widehat{\vartheta} \rangle$ et un plongement $\iota : {}^{L}\mathbf{T} \hookrightarrow {}^{L}\mathbf{G}$ en choisissant un relèvement $\dot{w} \in N(\widehat{\mathbf{T}}_{0},\widehat{\mathbf{G}})$ de $w$.

\begin{Pro}
\label{proEqRelevance}
Soit $\phi \in \Phi_{m}(I_k,\mathbf{G})$. Alors les propositions suivantes sont équivalentes :
\begin{enumerate}
\item $\phi$ est relevant
\item Il existe $\mathbf{T}$ un tore maximal non-ramifié de $\mathbf{G}$ et $\phi_{T} \in \Phi_{m}(I_{k},\mathbf{T})$ tel que $\phi \sim \iota \circ \phi_{T}$ où $\iota : {}^{L}\mathbf{T} \hookrightarrow {}^{L}\mathbf{G}$
\item $Rep_{\overline{\mathbb{Q}}_{\ell}}^{\phi}(G)$ est non vide
\end{enumerate}
\end{Pro}

\begin{proof}
L'équivalence $(2) \Leftrightarrow (3)$ est donnée par le théorème \ref{theRepIrr}. Montrons $(1) \Leftrightarrow (2)$.

Supposons $\phi$ relevant. Par définition, il existe $\varphi':W_{k}' \rightarrow {}^{L}\mathbf{G}$ relevant qui étend $\phi$. Notons $\varphi = \varphi'_{|W_k}$. Alors le lemme \ref{lemFactoVarphi} nous dit que $\mathcal{M}_{\varphi}$ factorise $\varphi'$ et est donc un Levi relevant. Le lemme \ref{lemToreElliptic} nous fournit $\mathcal{T}$ un tore maximal de $\mathcal{C}:=C_{\widehat{\mathbf{G}}}(\phi)^{\circ}\varphi(W_k)$ tel que $Z(\mathcal{T})^{\circ}=Z(\mathcal{C})^{\circ}$. Comme $\mathcal{C} \subseteq \mathcal{M}_{\varphi}$ on a également que $\mathcal{T}$ est un tore maximal de $\mathcal{M}_{\varphi}$. Maintenant $Z(\mathcal{T})^{\circ}=Z(\mathcal{C})^{\circ} \subseteq Z(\mathcal{M}_{\varphi})^{\circ}$ et comme $\mathcal{T}$ est un tore maximal de $\mathcal{M}_{\varphi}$ on a aussi $Z(\mathcal{M}_{\varphi})^{\circ} \subseteq Z(\mathcal{T})^{\circ}$ et par conséquent $Z(\mathcal{M}_{\varphi})^{\circ} = Z(\mathcal{T})^{\circ}$, c'est à dire que $\mathcal{T}$ est un tore maximal elliptique de $\mathcal{M}_{\varphi}$. Le lemme \ref{lemLeviRelevant} $(2)$ nous dit alors que $\mathcal{T}$ est un tore relevant. De plus comme $\mathcal{T} \subseteq \mathcal{C}$, ce tore factorise $\phi$ et l'on a $(2)$.

Supposons maintenant qu'il existe $\mathbf{T} \subseteq \mathbf{G}$ un tore maximal non-ramifié de $\mathbf{G}$ et $\phi_{T} \in \Phi_{m}(I_{k},\mathbf{T})$ tel que $\phi \sim \iota \circ \phi_{T}$. On rappelle que ${}^{L}\mathbf{T}=\widehat{\mathbf{T}}_{0} \rtimes \langle w \widehat{\vartheta} \rangle$. Le paramètre $\phi_{T}$ se prolonge en $\varphi_{T} \in \Phi(W_k,\mathbf{T})$ et on pose $\varphi = \iota \circ \varphi_{T} \in \Phi(W_k,G)$ qui est une extension de $\phi$ à $W_k$. Ainsi $Im(\phi) \subseteq \widehat{T}_{0}$, donc $\widehat{T}_{0} \subseteq C_{\widehat{G}}(\phi)^{\circ}$ et donc $\widehat{T}_{0} \subseteq \mathcal{M}_{\phi}$. Nous avons également que $\dot{w}\vartheta = \varphi(\text{Frob}) \in \mathcal{M}_{\phi}$. Ainsi $\iota({}^{L}\mathbf{T}) \subseteq \mathcal{M}_{\phi}$ et le lemme \ref{lemLeviRelevant} $(1)$ nous dit alors que $\mathcal{M}_{\phi}$ est relevant. De plus, le lemme \ref{lemFactoVarphi} nous construit $\varphi' \in \Phi(G)$ un paramètre étendant $\phi$ et tel que $\mathcal{M}_{\varphi}$ soit un Levi minimal contenant son image. Comme $\mathcal{M}_{\varphi}$ est relevant, on en déduit que $\varphi'$ est relevant et donc que $\phi$ est relevant.
\end{proof}

\begin{The}
Soit $\phi \in \Phi_{m}(I_{k}^{\Lambda},\mathbf{G})$. Alors $Rep_{\Lambda}^{\phi}(G)$ est non vide si et seulement si $\phi$ est relevant.
\end{The}

\begin{proof}
La proposition \ref{proEqRelevance} nous donne le résultat lorsque $\Lambda=\overline{\mathbb{Q}}_{\ell}$. Pour le cas général, notons que par la proposition \ref{proLienLambda}, $Rep_{\Lambda}^{\phi}(G)$ est non vide si et seulement s'il existe $\phi' \in \Phi_{m}(I_{k},\mathbf{G})$ tel que $\phi'_{|I_{k}^{\Lambda}} \sim \phi$ et $Rep_{\Lambda}^{\phi}(G)$ est non vide, si et seulement s'il existe $\phi' \in \Phi_{m}(I_{k},\mathbf{G})$ relevant prolongeant $\phi$, si et seulement si $\phi$ est relevant.
\end{proof}

\subsection{Compatibilité à l'induction et à la restriction parabolique}

\label{secCompaInductionParabolique}

Cette sous-partie à pour but d'étudier le comportement des catégories $Rep_{\Lambda}^{\phi}(G)$ vis à vis de l'induction et de la restriction parabolique.

\bigskip

Jusqu'à présent, nous avons considéré l'immeuble de Bruhat-Tits semi-simple, puisque celui-ci est muni d'une structure de complexe polysimplicial. Cependant dans cette section, nous souhaitons comparer l'immeuble d'un Levi et celui de $G$. L'immeuble de  Bruhat-Tits "étendu" semble alors plus approprié. Cela ne fait pas une grosse différence. En effet nous traitions la structure polysimpliciale de façon purement combinatoire. De plus, les idempotents $e_{\phi,\sigma}$ auraient très bien pu être définis pour un point quelconque $x$ de l'immeuble, on aurait alors eu que $e_{\phi,x}=e_{\phi,\sigma}$, où $\sigma$ est le plus petit polysimplexe contenant $x$. Ainsi, dans cette partie seulement, on utilisera l'immeuble de  Bruhat-Tits "étendu", que l'on notera $BT^{e}(G)$.

\bigskip

Soit $\mathbf{P}$ un $k$-sous-groupe parabolique de $\mathbf{G}$ de quotient de Levi $\mathbf{M}$ défini sur $k$. Prenons $\mathbf{S}$ un tore déployé maximal de $\mathbf{G}$ contenu dans $\mathbf{P}$ et notons $\mathbf{T}$ son centralisateur dans $\mathbf{G}$. Il existe alors un unique relèvement de $\mathbf{M}$ en un sous-groupe de $\mathbf{G}$ contenant $\mathbf{T}$. Notons $\varphi$ le système de racines de $G$ relativement à $S$ et $\varphi_{M} \subseteq \varphi$ celui de $M$. L'appartement $\mathcal{A}_{M}^{e}(\mathbf{S},k)$ de $BT^{e}(M)$ relativement à $\mathbf{S}$ est égal à $\mathcal{A}^{e}:=\mathcal{A}^{e}(\mathbf{S},k)$ mais en ne gardant que les murs associés aux racines affines dont la partie vectorielle est dans $\varphi_{M}$. Soit $x \in \mathcal{A}$, alors $M_{x}^{\circ}=M \cap G_{x}^{\circ}$ et $M_{x}^{+}=M \cap G_{x}^{+}$ (voir \cite{MoyPrasad} section 4.3). Rappelons que l'on a déjà défini $\overline{\textsf{M}}_{x} \simeq M_{x}^{\circ}/M_{x}^{+}$ et posons $\overline{\textsf{P}}_{x}$ l'image de $P_{x}:=P \cap G_{x}^{\circ}$ dans $\overline{\textsf{G}}_{x}$. 

\begin{Lem}
\label{lemPraboliqueImmeuble}
$\overline{\textsf{P}}_{x}$ est un sous-groupe parabolique de $\overline{\mathsf{G}}_{x}$ de quotient de Levi $\overline{\textsf{M}}_{x}$.
\end{Lem}

\begin{proof}
Notons $\varphi_{P}$ le sous-ensemble de $\varphi$ des racines $\alpha$ telles que $P$ soit engendré par $T$ et les $U_{\alpha}$, $\alpha \in \varphi_{P}$. Notons maintenant $\varphi_{x}$ (resp. $\varphi_{P,x}$, resp. $\varphi_{M,x}$) les racines affines passant par $x$ et dont la partie vectorielle est dans $\varphi$ (resp. $\varphi_{P}$, resp. $\varphi_{M}$). Alors $\varphi_{x}$ (resp. $\varphi_{P,x}$, resp. $\varphi_{M,x}$) est le système de racine de $\overline{\textsf{G}}_{x}$ (resp. $\overline{\textsf{P}}_{x}$, resp. $\overline{\textsf{M}}_{x}$) relativement à $S_{x}$. Choisissons maintenant une forme linéaire $f : X^{*}(\mathbf{S}) \rightarrow \mathbb{R}$ telle que $\varphi_{P}=\{\alpha \in \varphi , f(\alpha) \geq 0\}$ et $\varphi_{M}=\{\alpha \in \varphi, f(\alpha)=0\}$. Les sous-ensembles $\varphi_{P,x}$ et $\varphi_{M,x}$ vérifient alors $\varphi_{P,x}=\{\alpha \in \varphi_{x} , f(\alpha) \geq 0\}$ et $\varphi_{M,x}=\{\alpha \in \varphi_{x}, f(\alpha)=0\}$, ce qui montre que $\overline{\mathsf{P}}_{x}$ est bien un parabolique de $\overline{\mathsf{G}}_{x}$ de quotient de Levi $\overline{\mathsf{M}}_{x}$.
\end{proof}

Considérons $\widehat{\mathbf{M}}$ un dual de $\mathbf{M}$ sur $\overline{\mathbb{Q}}_{\ell}$ muni d'un plongement $\iota:{}^{L}\mathbf{M} \hookrightarrow {}^{L}\mathbf{G}$ (voir \cite{borel} section 3.4), qui induit une application $\Phi_{m}(I_{k}^{\Lambda},\mathbf{M}) \rightarrow \Phi_{m}(I_{k}^{\Lambda},\mathbf{G})$.

Commençons par vérifier la compatibilité à la restriction parabolique.

\begin{The}
\label{theRestrictionParabo}
Soit $\phi \in  \Phi_{m}(I_{k}^{\Lambda},\mathbf{G})$. Alors pour tout sous-groupe parabolique $\mathbf{P}$ de $\mathbf{G}$ ayant pour facteur de Levi $\mathbf{M}$, on a
\[ r_{P}^{G}( Rep_{\Lambda}^{\phi}(G) ) \subseteq \prod_{\phi_{M}} Rep_{\Lambda}^{\phi_{M}}(G)\]
où le produit est pris sur les $\phi_{M} \in \Phi_{m}(I_{k}^{\Lambda},\mathbf{M})$ tels que $\iota \circ \phi_{M} \sim \phi$ et $r_{P}^{G}$ désigne la restriction parabolique.
\end{The}

\begin{proof}
Soit $V \in Rep_{\Lambda}^{\phi}(G)$. La restriction parabolique préserve le niveau donc $r_{P}^{G}(V) \in Rep_{\Lambda}^{0}(M)$. Il nous suffit donc de montrer que pour $x \in \mathcal{A}^{e}_{M}$ et $\phi' \in \Phi_{m}(I_{k}^{\Lambda},\mathbf{M})$ tel que $\iota \circ \phi' \neq \phi$, on a $e_{\phi',x}r_{P}^{G}(V)=0$.

Nous avons $r_{P}^{G}(V)^{M_{x}^{+}} \simeq  r_{\overline{\mathsf{P}}_{x}}^{\overline{\textsf{G}}_{x}} (V^{G_{x}^{+}})$ (voir \cite{datFinitude} propositions 3.1 et 6.2), donc $e_{\phi',x} r_{P}^{G}(V)^{M_{x}^{+}}  \simeq e_{\phi',x} r_{\overline{\mathsf{P}}_{x}}^{\overline{\textsf{G}}_{x}} (V^{G_{x}^{+}}) \simeq e_{\phi',x} r_{\overline{\mathsf{P}}_{x}}^{\overline{\textsf{G}}_{x}}  ( e_{\iota \circ \phi',x}(V^{G_{x}^{+}}))$ (la dernière égalité provient de \ref{proidemparabolique}). Or par hypothèse $\iota \circ \phi' \neq \phi$ donc $e_{\iota \circ \phi',x}(V^{G_{x}^{+}}) = 0$ d'où le résultat.
\end{proof}

Passons maintenant à l'induction parabolique.

\begin{The}
\label{theInductionParabolique}
Soit $\phi_{M} \in \Phi_{m}(I_{k}^{\Lambda},\mathbf{M})$ et notons $\phi \in  \Phi_{m}(I_{k}^{\Lambda},\mathbf{G})$ son image par $\Phi_{m}(I_{k}^{\Lambda},\mathbf{M}) \rightarrow \Phi_{m}(I_{k}^{\Lambda},\mathbf{G})$. Alors pour tout sous-groupe parabolique $\mathbf{P}$ de $\mathbf{G}$ ayant pour facteur de Levi $\mathbf{M}$, on a
\[ i_{P}^{G}( Rep_{\Lambda}^{\phi_{M}}(M) ) \subseteq Rep_{\Lambda}^{\phi}(G)\]
où $i_{P}^{G}$ désigne l'induction parabolique.
\end{The}

\begin{proof}
Cela découle du théorème \ref{theRestrictionParabo} et du fait que $r_{P}^{G}$ est adjoint à gauche de $i_{P}^{G}$.
\end{proof}

\begin{Pro}
Soit $\phi \in  \Phi_{m}(I_{k}^{\Lambda},\mathbf{G})$. Alors si $\phi$ est discret (c'est-à-dire ne se factorise pas par un Levi rationnel propre) toutes les représentations de $Rep_{\Lambda}^{\phi}(G)$ sont cuspidales et toutes les représentations irréductibles de $Rep_{\Lambda}^{\phi}(G)$ sont supercuspidales.

De plus, si $\mathbf{G}$ est quasi-déployé, on a la réciproque pour la cuspidalité, c'est à dire que $\phi$ est discret si et seulement si $Rep_{\Lambda}^{\phi}(G)$ ne contient que des cuspidales.
\end{Pro}

\begin{proof}
La cuspidalité découle immédiatement du théorème \ref{theRestrictionParabo}. Pour la supercuspidalité, remarquons que si $\phi$ est discret, alors le théorème \ref{theInductionParabolique} montre qu'une induite n'as pas de composante dans $Rep_{\Lambda}^{\phi}(G)$ et donc n'a pas de sous-quotient irréductible dans $Rep_{\Lambda}^{\phi}(G)$.

Maintenant si $\mathbf{G}$ est quasi-déployé et que $\phi$ n'est pas discret, alors $Rep_{\Lambda}^{\phi}(G)$ contient des induites d'après le théorème \ref{theInductionParabolique} (nous utilisons l'hypothèse quasi-déployé pour dire que ces facteurs sont non-nuls).
\end{proof}

Si $\mathbf{G}$ n'est pas quasi-déployé l'équivalence peut être fausse, comme le montre l'exemple de $G=D^{\times}$ où $D$ est une $k$-algèbre à division de dimension finie. Alors toutes les représentations sont cuspidales, en particulier $Rep_{\Lambda}^{1}(G)$ ne contient que des cuspidales.

Avoir $\phi$ discret n'est pas une condition nécessaire pour avoir des cuspidales (supercuspidales) dans $Rep_{\Lambda}^{\phi}(G)$. En effet, toute cuspidale unipotente se retrouvera dans $Rep_{\Lambda}^{1}(G)$.

\bigskip

Soit $\phi_{M} \in  \Phi_{m}(I_{k}^{\Lambda},\mathbf{M})$ et posons $\phi = \iota \circ \phi_{M} \in \Phi_{m}(I_{k}^{\Lambda},\mathbf{G})$. Nous venons de voir que $i_{P}^{G}$ réalise un foncteur
\[i_{P}^{G} : Rep_{\Lambda}^{\phi_{M}}(M) \longrightarrow Rep_{\Lambda}^{\phi}(G).\]
La catégorie $Rep_{\Lambda}^{\phi_{M}}(M)$ est un facteur direct de $Rep^{0}_{\Lambda}(M)$, on a donc un foncteur $e_{\phi_{M}} : Rep^{0}_{\Lambda}(M) \rightarrow Rep_{\Lambda}^{\phi_{M}}(M)$. Définissons $r_{P}^{G,\phi_{M}}:=e_{\phi_{M}} \circ r_{P}^{G}$, de sorte que $r_{P}^{G,\phi_{M}}$ soit un foncteur
\[r_{P}^{G,\phi_{M}} : Rep_{\Lambda}^{\phi}(G) \longrightarrow Rep_{\Lambda}^{\phi_{M}}(M).\]

\begin{Lem}
Le foncteur $r_{P}^{G,\phi_{M}}$ est adjoint à gauche de $i_{P}^{G}$.
\end{Lem}

\begin{proof}
Soient $V \in Rep_{\Lambda}^{\phi}(G)$ et $W \in Rep_{\Lambda}^{\phi_{M}}(M)$. Nous savons déjà que $r_{P}^{G}$ est adjoint à gauche de $i_{P}^{G}$, donc $Hom(r_{P}^{G}(V),W)=Hom(V,i_{P}^{G}(W))$. Maintenant $Rep_{\Lambda}^{\phi_{M}}(M)$ est un facteur direct de $Rep_{\Lambda}^{0}(M)$ de sorte que $Hom(r_{P}^{G}(V),W)=Hom(r_{P}^{G,\phi_{M}}(V),W)$ et on a le résultat.
\end{proof}

\begin{The}
Notons $C_{\widehat{\mathbf{G}}}(\phi)$ le centralisateur dans $\widehat{\mathbf{G}}$ de l'image de $\phi$. Alors si $C_{\widehat{\mathbf{G}}}(\phi) \subseteq \iota(\widehat{\mathbf{M}})$ la paire de foncteurs adjoints $(i_{P}^{G},r_{P}^{G,\phi_{M}})$ réalise une équivalence de catégories entre $Rep_{\Lambda}^{\phi_{M}}(M)$ et $Rep_{\Lambda}^{\phi}(G)$.
\end{The}

\begin{proof}
Soit $V \in Rep_{\Lambda}^{\phi_{M}}(M)$. Par adjonction, nous avons une application $r_{\mathbf{P}}^{\mathbf{G}}i_{\mathbf{P}}^{\mathbf{G}}(V)\rightarrow V$. Le lemme géométrique nous dit qu'elle est surjective et que son noyau $W$, admet une filtration dont les composantes du gradué associé sont isomorphes à $(i_{M \cap {}^{w}P}^{M} \circ w \circ r_{{}^{w^{-1}}P\cap M}^{M})(V)$, où $w$ parcourt un ensemble $\mathcal{W}^{P}$ de représentants particuliers dans $G$ des doubles classes $W_{M,0}\backslash W_{0}/W_{M,0}$ ne contenant pas la classe triviale ($W_{M,0}$ est le groupe de Weyl de $\mathbf{M}$). Nous souhaitons montrer que $e_{\phi_{M}}(W)=0$.

Prenons donc $w \in \mathcal{W}^{P}$ et montrons que 
\[e_{\phi_{M}}((i_{M \cap {}^{w}P}^{M} \circ w \circ r_{{}^{w^{-1}}P\cap M}^{M})(V))=0.\]
Identifions les paramètres inertiels avec des classes de conjugaison semi-simples $F$-stables. Ainsi $\phi_{M}$ correspond à $s \in (\textsf{\textbf{M}}^{*}_{ss})^{F}$ et on note $Rep_{\Lambda}^{s}(M)$ pour $Rep_{\Lambda}^{\phi_{M}}(M)$. Par le théorème \ref{theRestrictionParabo}
\[r_{{}^{w^{-1}}P\cap M}^{M}(V) \in \prod_{i=1}^{n} Rep_{\Lambda}^{s_{i}}({}^{w^{-1}}M\cap M)\]
où $\{s_{1},\cdots,s_{n}\}$ est l'image réciproque de $\{s\}$ par l'application $(({}^{w^{-1}}\textsf{\textbf{M}} \cap \textsf{\textbf{M}})^{*}_{ss})^{F} \rightarrow (\textsf{\textbf{M}}^{*}_{ss})^{F}$. Donc
\[w \circ r_{{}^{w^{-1}}P\cap M}^{M}(V) \in \prod_{i=1}^{n} Rep_{\Lambda}^{{}^{w}s_{i}}(M\cap {}^{w}M).\]
Enfin par le théorème \ref{theInductionParabolique}
\[(i_{M \cap {}^{w}P}^{M} \circ w \circ r_{{}^{w^{-1}}P\cap M}^{M})(V) \in \prod_{i=1}^{m} Rep_{\Lambda}^{t_{i}}(M)\]
où $\{t_{1},\cdots,t_{m}\}$ est l'image de $\{{}^{w}s_{1},\cdots,{}^{w}s_{n}\}$ par l'application $((\textsf{\textbf{M}} \cap {}^{w}\textsf{\textbf{M}})^{*}_{ss})^{F} \rightarrow (\textsf{\textbf{M}}^{*}_{ss})^{F}$.

On veut donc montrer qu'aucun des $t_{i}$ n'est égal à $s$. Supposons le contraire et que l'on ait un $i$ tel que $t_{i}=s$. Par construction, $t_{i}$ est dans l'une des classes de conjugaison sur $\textsf{\textbf{M}}^{*}$ des ${}^{w}(s_{j})$, donc il existe $g \in \textsf{\textbf{M}}^{*}$ et $j \in \{1,\cdots,n\}$ tels que $t_{i}={}^{gw}(s_{j})$. De même par construction des $s_{j}$, il existe $h \in \textsf{\textbf{M}}^{*}$ tel que $s_{j}={}^{h}s$. Donc $s=t_{i}={}^{gwh}(s)$ et $gwh \in C_{\textsf{\textbf{G}}^{*}}(s) \subseteq \textsf{\textbf{M}}^{*}$ (par hypothèse), ce qui est absurde car $w \notin \textsf{\textbf{M}}^{*}$. Ceci nous montre que $e_{\phi_{M}}(W)=0$ et donc que $r_{\mathbf{P}}^{\mathbf{G},\phi_{M}}i_{\mathbf{P}}^{\mathbf{G}}(V) \overset{\sim}\rightarrow V$ est un isomorphisme.

\medskip

Montrons maintenant que le foncteur $r_{\mathbf{P}}^{\mathbf{G},\phi_{M}}$ est conservatif. Ceci nous permettra de conclure grâce au lemme \ref{lemEqCategories} ci-dessous. Comme les catégories considérées sont abéliennes et que $r_{\mathbf{P}}^{\mathbf{G},\phi_{M}}$ est exact, il nous suffit de montrer que si $V\neq 0$ alors $r_{\mathbf{P}}^{\mathbf{G},\phi_{M}}(V) \neq 0$.

Prenons donc $V \in Rep_{\Lambda}^{\phi}(G)$ tel que $V\neq 0$. Il existe $x \in \mathcal{A}^{e}$ tel que $e_{\phi,x}V\neq 0$. Notons $s \in (\textsf{\textbf{M}}^{*}_{ss})^{F}$ la classe de conjugaison semi-simple correspondant à $\phi_{M}$. Comme $e_{\phi,x}V\neq 0$, il existe $s_{x} \in (\overline{\textsf{M}}^{*}_{x})_{ss}$ dont l'image par $(\overline{\textsf{\textbf{M}}}^{*}_{x})_{ss}^{F} \rightarrow (\textsf{\textbf{M}}^{*}_{ss})^{F}$ est $s$ et telle que $e_{s_{x},\Lambda}^{\overline{\textsf{G}}_{x}}V^{G_{x}^{+}}\neq 0$. L'hypothèse $C_{\textsf{\textbf{G}}^{*}}(s) \subseteq \textsf{\textbf{M}}^{*}$ peut se retraduire, si l'on voit $s$ comme un élément de $\textsf{\textbf{T}}^{*}$, de la manière suivante : si $w \in W_{0}$ est tel que $ws=s$ alors $w \in W_{M,0}$. Or l'application $(\overline{\textsf{\textbf{M}}}^{*}_{x})_{ss} \rightarrow (\textsf{\textbf{M}}^{*}_{ss})$ est définie par $\textsf{\textbf{T}}_{x}^{*}/W_{M,x} \rightarrow \textsf{\textbf{T}}^{*}/W_{M,0}$ et comme $W_{M,x}=W_{M,0} \cap W_{x}$ on en déduit que $s_{x}$ vérifie les mêmes hypothèses que $s$, c'est à dire $C_{\overline{\textsf{\textbf{G}}}_{x}^{*}}(s_{x}) \subseteq \overline{\textsf{\textbf{M}}}^{*}_{x}$. 

Maintenant, nous avons vu dans la preuve du théorème \ref{theRestrictionParabo} que 
\[e_{s_{x},\Lambda}^{\overline{\textsf{M}}_{x}}(r_{P}^{G}(V))^{M_{x}^{+}}\simeq e_{s_{x},\Lambda}^{\overline{\textsf{M}}_{x}} r_{\overline{\textsf{P}}_{x}}^{\overline{\textsf{G}}_{x}}(e_{s_{x},\Lambda}^{\overline{\textsf{G}}_{x}}(V^{G_{x}^{+}})).\]
En notant $r_{\overline{\textsf{P}}_{x}}^{\overline{\textsf{G}}_{x},s_{x}}$ le foncteur $e_{s_{x},\Lambda}^{\overline{\textsf{M}}_{x}} r_{\overline{\textsf{P}}_{x}}^{\overline{\textsf{G}}_{x}}$, on a 
\[e_{s_{x},\Lambda}^{\overline{\textsf{M}}_{x}}(r_{P}^{G}(V))^{M_{x}^{+}}\simeq r_{\overline{\textsf{P}}_{x}}^{\overline{\textsf{G}}_{x},s_{x}}(e_{s_{x},\Lambda}^{\overline{\textsf{G}}_{x}}(V^{G_{x}^{+}})).\]
Comme $C_{\overline{\textsf{\textbf{G}}}_{x}^{*}}(s_{x}) \subseteq \overline{\textsf{\textbf{M}}}^{*}_{x}$ le théorème B' de \cite{bonnafe_rouquier} nous dit que $r_{\overline{\textsf{P}}_{x}}^{\overline{\textsf{G}}_{x},s_{x}}$ réalise une équivalence de catégories. En particulier il est conservatif et comme $e_{s_{x},\Lambda}^{\overline{\textsf{G}}_{x}}(V^{G_{x}^{+}}) \neq 0$ on a 
\[e_{s_{x},\Lambda}^{\overline{\textsf{M}}_{x}}(r_{P}^{G}(V))^{M_{x}^{+}} \simeq r_{\overline{\textsf{P}}_{x}}^{\overline{\textsf{G}}_{x},s_{x}}(e_{s_{x},\Lambda}^{\overline{\textsf{G}}_{x}}(V^{G_{x}^{+}})) \neq 0\]
et donc $r_{\mathbf{P}}^{\mathbf{G},\phi_{M}}(V) \neq 0$ ce qui achève la preuve.
\end{proof}

\begin{Lem}
\label{lemEqCategories}
Soient $\mathcal{C}$, $\mathcal{D}$ deux catégories et $F : \mathcal{C} \rightarrow \mathcal{D}$ un foncteur. Si $F$ est conservatif et admet un adjoint à droite (ou à gauche) pleinement fidèle, alors $F$ réalise une équivalence de catégories.
\end{Lem}

\begin{proof}

Soit $G$ un adjoint à droite. Le cas d'un adjoint à gauche se traite de la même manière par dualité. Le foncteur $G$ étant pleinement fidèle, le morphisme $\alpha : FG \rightarrow id_{\mathcal{D}}$ est un isomorphisme naturel. Nous devons montrer que $\beta : id_{\mathcal{C}} \rightarrow GF$ est également un isomorphisme. Par les axiomes d'adjonctions, la composition
\[ F(x) \overset{F(\beta_{x})}{\longrightarrow}FGF(x) \overset{\alpha_{F(x)}}{\longrightarrow} F(x)\]
est $id_{F(x)}$ pour tout $x\in \mathcal{C}$. Comme $\alpha$ est déjà un isomorphisme, on en déduit que $F(\beta_{x})$ est un isomorphisme, et donc que $\beta_{x}$ est un isomorphisme puisque $F$ est conservatif.
\end{proof}

\subsection{Compatibilité à la correspondance de Langlands}

\label{secCompatibiliteLanglands}

Dans cette partie nous prendrons $\Lambda=\overline{\mathbb{Q}}_{\ell}$ et $k$ de caractéristique nulle. La correspondance de Langlands locale prédit une application à fibres finies $\text{Irr}_{\overline{\mathbb{Q}}_{\ell}}(G) \rightarrow \Phi(\mathbf{G})$, $\pi \mapsto \varphi_{\pi}$. Dans des cas où elle est connue, nous souhaitons vérifier que la décomposition du théorème \ref{thmDecompoInertiel} est bien compatible à cette dernière, c'est à dire que si $\pi \in \text{Irr}_{\overline{\mathbb{Q}}_{\ell}}(G)$ est une représentation de niveau 0 alors $\pi \in Rep_{\overline{\mathbb{Q}}_{\ell}}^{\phi}(G)$ avec $\varphi_{\pi | I_{k}} \sim \phi$.

La correspondance de Langlands est connue dans plusieurs cas dont : les tores (prouvé par Langlands lui-même), les représentations unipotentes des groupes $p$-adiques adjoints (\cite{lusztig}, \cite{lusztig2}) et les groupes classiques (\cite{HarrisTaylor} \cite{henniart} \cite{arthur} \cite{mok} \cite{KMSW}). La compatibilité à la correspondance de Langlands pour les tores est contenue dans le théorème \ref{theRepIrr}. Pour ce qui est des représentations unipotentes, par construction, elle appartiennent toutes à $Rep_{\overline{\mathbb{Q}}_{\ell}}^{1}(G)$. Par conséquent, dans cette section nous examinerons le cas des groupes classiques. Notons que $\mathbf{G}=GL_{n}$ a déjà été fait dans \cite{dat_equivalences_2014} section 3.2.6, on se concentrera ici sur les autres cas. Dans le but d'utiliser les résultats de \cite{StevensLust}, nous supposerons de plus que $k$ est de caractéristique résiduelle impaire.

\bigskip

Commençons par expliquer ce que l'on entend par un groupe classique. Soit $k'$ une extension de degré au plus 2 de $k$ et $\mathfrak{f}'$ son corps résiduel. Notons $\sigma$ le générateur du groupe de Galois de $k'/k$ et $N_{k'/k}:k'^{\times} \rightarrow k^{\times}$ l'application norme. Fixons un signe $\varepsilon=\pm 1$ et soit $(V,h)$ un espace $k'/k$-$\varepsilon$-hermitien non-dégénéré. Dans cette partie $G:=U(V)^{\circ}$ est la composante connexe du groupe des $k$-points du groupe réductif déterminé par $(V,h)$, c'est à dire $U(V):=\{g \in Aut_{k'}(V), h(gv,gw)=h(v,w) \text{ pour tout } v,w\in V \}$ et $U(V)^{\circ}:=\{g \in U(V), N_{k'/k} det_{k'}(g)=1\}$. 

\medskip

La correspondance de Langlands pour les groupes classiques (restreinte à $W_{k}$) est compatible à l'induction parabolique (voir \cite{moussaoui} théorème 4.9). Il en est de même pour $Rep_{\overline{\mathbb{Q}}_{\ell}}^{\phi}(G)$ d'après la section \ref{secCompaInductionParabolique}. Nous pouvons ainsi nous restreindre ici aux représentations irréductibles cuspidales. On pose $\mathcal{A}(G)$ l'ensemble des classes d'équivalence des représentations irréductibles cuspidales de $G$ et $\mathcal{A}_{[0]}(G)$ le sous-ensemble des représentations de niveau 0. Soit $\rho$ une représentation irréductible cuspidale de $GL_{n}(k')$. On définit la représentation $\rho^{\sigma}$ de $GL_{n}(k')$ par $\rho^{\sigma}(g)=\rho({}^{t}\sigma(g^{-1}))$, où ${}^{t}g$ désigne la transposée de $g$. On dit alors que $\rho$ est auto-duale si $\rho^{\sigma} \simeq \rho$ et on pose $\mathcal{A}_{n}^{\sigma}(k')$ l'ensemble des représentations irréductibles cuspidales auto-duales de $GL_{n}(k')$, $\mathcal{A}^{\sigma}(k'):=\cup_{n \geq 1} \mathcal{A}_{n}^{\sigma}(k')$ et $\mathcal{A}_{[0]}^{\sigma}(k')$ le sous-ensemble de $\mathcal{A}^{\sigma}(k')$ composé des représentations de niveau 0.

Notons $H=H^{-} \oplus H^{+}$ le plan hyperbolique, c'est à dire, $H^{\pm}$ est un $k'$-espace vectoriel de dimension 1 de base $e_{\pm}$ et $H$ est muni de la forme $h_{H}$ donnée par $h_{H}(\lambda_{-}e_{-}+\lambda_{+}e_{+},\mu_{-}e_{-}+\mu_{+}e_{+})=\lambda_{-}\sigma(\mu_{+})+\varepsilon\lambda_{+}\sigma(\mu_{-})$. Pour un entier $n \geq 0$, on pose $V_{n}:=V \oplus nH$ muni de la forme $h_{n}:=h \oplus h_{H} \oplus \cdots \oplus h_{H}$ et $G_{n}:=U(V_{n})^{\circ}$. Le stabilisateur dans $G_{n}$ de la décomposition $V_{n}=nH^{-} \oplus V \oplus nH^{+}$ est un sous-groupe de Levi $M_n$ de $G_n$ et l'on a un isomorphisme $M_n \simeq GL_n(k') \times G$. Le stabilisateur de $nH^{-}$ est quand à lui un sous-groupe parabolique $P_n$ de $G_n$ de facteur de Levi $M_n$. Ainsi si $\rho \in \mathcal{A}_{n}^{\sigma}(k')$ et $\pi \in \mathcal{A}(G)$, on peut former $\rho \otimes \pi$ que l'on considère comme une représentation de $M_n$. Définissons pour $s \in \mathbb{C}$, $I(\rho,\pi,s):=i_{P_{n}}^{G_{n}} \rho\mid det(\cdot)\mid_{k'}^{s} \otimes \pi$.

\begin{The}[\cite{silberger}, théorème 1.6] Si $I(\rho,\pi,s)$ est réductible pour un certain $s\in \mathbb{R}$ alors il existe un unique réel positif $s_{\pi}(\rho)$ tel que, $I(\rho,\pi,s)$ est réductible si et seulement si $s=\pm s_{\pi}(\rho)$.
\end{The}
Si $I(\rho,\pi,s)$ est irréductible pour tout $s\in \mathbb{R}$, on pose alors $s_{\pi}(\rho)=0$. Définissons alors l'ensemble de Jordans $Jord(\pi)$ par 
\[ Jord(\pi) := \{ (\rho,m) \in \mathcal{A}^{\sigma}(k') \times \mathbb{N}^{*}, 2s_{\pi}(\rho) -(m+1) \in 2 \mathbb{N} \} \]

Rappelons que comme $G$ est un groupe classique on a la correspondance de Langlands locale, $\text{Irr}_{\overline{\mathbb{Q}}_{\ell}}(G) \rightarrow \Phi(\mathbf{G})$, $\pi \mapsto \varphi_{\pi}$. Jusqu'à présent nos paramètres de Langlands $\varphi \in \Phi(\mathbf{G})$ étaient de la forme $\varphi : W'_{k} \rightarrow {}^{L}\textbf{G}(\overline{\mathbb{Q}}_{\ell})$ avec $W_{k}'=W_{k}\ltimes \overline{\mathbb{Q}}_{\ell}$. Dans cette section seulement, on considére une autre version du groupe de Weil-Deligne : $W_{k}' \simeq W_{k} \times SL_{2}(\overline{\mathbb{Q}}_{\ell})$. Ainsi les paramètres de Langlands sont de la forme $\varphi:W_{k} \times SL_{2}(\overline{\mathbb{Q}}_{\ell}) \rightarrow {}^{L}\mathbf{G}$. Notons $N_{\widehat{\mathbf{G}}}$ la dimension de l'espace vectoriel sur lequel agit naturellement $\widehat{\mathbf{G}}$. Alors si $\varphi:W_{k} \times SL_{2}(\overline{\mathbb{Q}}_{\ell}) \rightarrow {}^{L}\mathbf{G}$ est un paramètre de Langlands, on notera $\tilde{\varphi}:W_{k'} \times SL_{2}(\overline{\mathbb{Q}}_{\ell}) \rightarrow GL_{N_{\widehat{\mathbf{G}}}}(\overline{\mathbb{Q}}_{\ell})$ l'application obtenue en restreignant $\varphi$ à $W_{k'}$ et en la composant avec $\widehat{\mathbf{G}} \hookrightarrow GL_{N_{\widehat{\mathbf{G}}}}(\overline{\mathbb{Q}}_{\ell})$. Il est attendu (et connu au moins dans le cas où $\mathbf{G}$ est quasi-déployé, voir par exemple \cite{moeglin}) que le paramètre $\tilde{\varphi}_{\pi}$ soit décrit grâce à l'ensemble de Jordan $Jord(\pi)$ de la façon suivante

\begin{The}[\cite{moeglin}]
\label{thmLLC}
Si $\mathbf{G}$ est un groupe classique quasi-déployé et $\pi \in \mathcal{A}(G)$, on a
\[\tilde{\varphi}_{\pi}=\bigoplus_{(\rho,m) \in Jord(\pi)} \varphi_{\rho} \otimes st_{m}\]
où $\varphi_{\rho}$ est la représentation irréductible de $W_{k'}$ correspondant à $\rho$ via la correspondance de Langlands locale pour $GL_{n}$ et $st_{m}$ est la représentation irréductible $m$-dimensionnelle de $SL_{2}(\overline{\mathbb{Q}}_{\ell})$.
\end{The}

Le résultat précédent étant connu dans le cas où $\mathbf{G}$ est quasi-déployé et le théorème \ref{thmDecompoInertiel} nécessitant l'hypothèse $K$-déployé, nous nous limiterons ici au cas où $\mathbf{G}$ est un groupe classique non-ramifié, c'est à dire un groupe spécial orthogonal impair $SO_{2n+1}$, un groupe spécial orthogonal pair $SO_{2n}$ (groupe spécial orthogonal déployé) ou $SO_{2n}^{*}$ (groupe spécial orthogonal quasi-déployé associé à une extension quadratique non-ramifiée $k'/k$), un groupe symplectique $Sp_{2n}$ ou un groupe unitaire $U_{n}(k'/k)$ où $k'$ est une extension non-ramifiée de $k$.

\medskip

Pour comprendre $\varphi_{\pi}$ nous avons donc besoin de comprendre $Jord(\pi)$ et en particulier $s_{\pi}(\rho)$. Nous allons pour cela nous appuyer sur les résultats obtenus dans \cite{StevensLust}.

Soit $\pi \in \mathcal{A}_{[0]}(G)$. Il existe alors $x \in BT_{0}$ et $s \in (\overline{\textsf{G}}_{x}^{*})_{ss}$ tels que $e_{s,\overline{\mathbb{Q}}_{\ell}}^{\overline{\textsf{G}}_{x}} \pi^{G_{x}^{+}}$. Comme nous sommes dans le cas où $\mathbf{G}$ est un groupe classique, le groupe $\overline{\textsf{G}}_{x}$ se décompose en un produit de deux groupes $\overline{\textsf{G}}_{x} \simeq \overline{\textsf{G}}_{x,1} \times \overline{\textsf{G}}_{x,2}$, où les $\overline{\textsf{G}}_{x,i}$ sont des groupes classiques (voir par exemple \cite{StevensLust} section 2). Ainsi $s$ correspond via cet isomorphisme à $(s_1,s_2)$ où $s_i \in (\overline{\textsf{G}}_{x,i}^{*})_{ss}$. Soit $P \in \mathfrak{f}'[X]$ un polynôme unitaire. Désignons par $\tau$ le générateur du groupe de Galois $Gal(\mathfrak{f}'/\mathfrak{f})$. On définit $P^{\tau}(X):= \tau(P(0))^{-1}X^{deg(P)}\tau(P)(1/X)$ et on dit que $P$ est auto-dual si $P=P^{\tau}$. Le polynôme caractéristique $P_{s_i}$ de $s_i$ est alors un polynôme unitaire auto-dual et on l'écrit (comme dans \cite{StevensLust} section 7) $P_{s_i}(X)=\prod_{P} P(X)^{a_{P}^{(i)}}$, où le produit est pris sur l'ensemble des polynômes unitaires irréductibles auto-duaux sur $\mathfrak{f}'$ (une telle écriture est possible car la série de Deligne-Lusztig associée à $s_{i}$ contient une représentation cuspidale et que si $P_{s_{i}}$ contenait un facteur de la forme $P(X)P^{\tau}(X)$ avec $P$ irreductible et $P\neq P^{\tau}$ alors le centralisateur de $s_{i}$ dans $(\overline{\textsf{G}}_{x,i}^{*})_{ss}$ serait contenu dans un Levi rationnel propre).

Soit $\rho$ une représentation irréductible cuspidale auto-duale de niveau zero d'un certain $GL_n(k')$. Notons $G'=GL_n(k')$. Comme précédemment nous avons l'existence d'un $y \in BT_{0}(\mathbf{G}',k')$ et d'une classe de conjugaison semi-simple $s_{\rho} \in (\overline{\textsf{G}}'^{*}_{y})_{ss}$ telle que $e_{s_{\rho},\overline{\mathbb{Q}}_{\ell}}^{\overline{\textsf{G}}'_{y}} \rho^{G'^{+}_{y}}$. Notons qu'ici $\overline{\textsf{G}}'^{*}_{y} \simeq GL_n(\mathfrak{f}')$. Associons de même à $s_{\rho}$ son polynôme caractéristique $Q$ qui est un polynôme unitaire irréductible auto-dual de degré $n$.

\begin{The}[\cite{StevensLust} section 8]
\label{theCalculSpi}
Notons $\rho'\in \mathcal{A}_{n}^{\sigma}(k')$ l'unique (à équivalence près) twist non-ramifié (non-équivalent) de $\rho$ qui est auto-dual. Alors pour $\mathbf{G}$ un groupe classique non-ramifié, on a
\[\lfloor s_{\pi}(\rho)^{2} \rfloor + \lfloor s_{\pi}(\rho')^{2} \rfloor=a_{Q}^{(1)}+a_{Q}^{(2)}\]
sauf si $\mathbf{G}=Sp_{2n}$ et $Q(X)=X-1$ où dans ce cas on a
\[\lfloor s_{\pi}(\rho)^{2} \rfloor + \lfloor s_{\pi}(\rho')^{2} \rfloor=a_{(X-1)}^{(1)}+a_{(X-1)}^{(2)}-1.\]
\end{The}

Les résultats dans \cite{StevensLust} sont exprimés en termes de polynômes auto-duaux. Nous nous utilisons plutôt des classes de conjugaisons semi-simples. Nous souhaitons donc faire le lien entre les deux.

Soit $\textbf{\textsf{H}}$ un groupe du type $GL_{n}$, $Sp_{2n}$, $SO_{2n+1}$ ou $SO_{2n}$ sur $\mathfrak{F}$. On considère son plongement naturel $\textbf{\textsf{H}} \subseteq GL_{N}$ où $N$ est en entier naturel ($N=n$ pour $GL_{n}$, $N=2n+1$ pour $SO_{2n+1}$ et $N=2n$ pour $Sp_{2n}$ ou $SO_{2n}$). Soit $s \in \textbf{\textsf{H}}_{ss}$ une classe de conjugaison semi-simple. Celle-ci donne lieu à une classe de conjugaison semi-simple de $GL_{N}$ et l'on peut considérer $P$ son polynôme caractéristique.

\begin{Lem}
\label{lemPolyCaract}
Une classe de conjugaison géométrique semi-simple $s \in \textbf{\textsf{H}}_{ss}$ est caractérisée par son polynôme caractéristique $P$.
\end{Lem}

\begin{proof}
Traitons $\textsf{\textbf{H}}=Sp_{2n}$, les autres cas étant similaires. On a donc $N=2n$. Soit $\textsf{\textbf{T}}$ le tore déployé de $Sp_{2n}$, c'est à dire $\textsf{\textbf{T}}=Diag(a_{1},\cdots,a_{n},a_{n}^{-1},\cdots,a_{1}^{-1})$, $a_{i} \in \mathfrak{F}$. Nous pouvons supposer que $s \in \textsf{\textbf{T}}$ et donc écrire $s=(a_{1},\cdots,a_{2n})$ avec $a_{n+i} =a_{n-i+1}^{-1}$ pour $1\leq i \leq n$. Prenons $s'=(b_{1},\cdots,b_{2n}) \in \textsf{\textbf{T}}$ une autre classe de conjugaison semi-simple géométrique de $\textbf{\textsf{H}}$ ayant pour polynôme caractéristique $P$. Nous souhaitons donc montrer que $s$ et $s'$ sont conjuguées dans $\textbf{\textsf{H}}$. Comme ce sont deux éléments de $\textsf{\textbf{T}}$ cela est équivalent au fait qu'il existe $w \in W_{\textsf{\textbf{H}}}^{0}$, le groupe de Weyl de  $\textsf{\textbf{H}}$, tel que $w\cdot s'=s$. Rappelons que $W_{\textsf{\textbf{H}}}^{0} \simeq \mathcal{S}_{n} \ltimes (\mathbb{Z}/2\mathbb{Z})^{n}$. Comme $s$ et $s'$ ont même polynôme caractéristique, il existe une permutation $\sigma \in \mathcal{S}_{2n}$ telle que $b_{i} = a_{\sigma(i)}$. Comme $s \in \textsf{\textbf{T}}$, $\{a_{\sigma(1)},\cdots,a_{\sigma(2n)}\}=\{a_{1}^{\pm},\cdots,a_{n}^{\pm}\}$ (comptés avec multiplicités), donc il existe $i\in \{1,\cdots,n\}$ tel que $a_{\sigma(1)}=a_{i}^{\pm}$. Posons $\tau(1)=i$. Comme $s' \in \textsf{\textbf{T}}$, $a_{\sigma(2n)}=b_{2n}=b_{1}^{-1}=a_{\sigma(1)}^{-1}$, donc $\{a_{\sigma(2)},\cdots,a_{\sigma(2n-1)}\}=\{a_{i}^{\pm},1 \leq i \leq n, i\neq \tau(1)\}$. On construit donc par récurrence une permutation $\tau \in \mathcal{S}_{n}$ telle que $a_{\sigma(i)}=a_{\tau(i)}^{\pm}$, pour $1\leq i \leq n$. Dit autrement, on vient de fabriquer un élément $w \in W_{\textsf{\textbf{H}}}^{0} \simeq \mathcal{S}_{n} \ltimes (\mathbb{Z}/2\mathbb{Z})^{n}$ tel que $b_{i}=a_{\sigma(i)}=w \cdot a_{i} $, donc tel que $s' = w \cdot s$ ce qui achève la preuve.
\end{proof}

Soit $x \in BT_{0}$ et $s \in (\overline{\textsf{G}}_{x}^{*})_{ss}$. La classe de conjugaison $s$ correspond comme précédemment à $(s_{1},s_{2})$, $s_{i} \in (\overline{\textsf{G}}_{x,i}^{*})_{ss}$ qui ont pour polynômes caractéristiques $P_{s_1}$ et $P_{s_2}$. Dans la section \ref{secConjParahoriques} on a construit une application $\tilde{\psi}_{x} : ((\overline{\textsf{\textbf{G}}}^{*}_{x})_{ss})^{F} \rightarrow (\textsf{\textbf{G}}^{*}(\mathfrak{F})_{ss})^{F}$. Nommons $\tilde{s}:=\tilde{\psi}_{x}(s)$ l'image de $s$ par cette application.

\begin{Lem}
\label{lemPstilde}
Notons $P_{\tilde{s}}$ le polynôme caractéristique de $\tilde{s}$ alors on a :
\begin{enumerate}
\item Si $\mathbf{G}\neq Sp_{2n}$ : $P_{\tilde{s}}(X)=P_{s_1}(X)P_{s_2}(X)$
\item Si $\mathbf{G} = Sp_{2n}$ : $P_{\tilde{s}}(X)=P_{s_1}(X)P_{s_2}(X)/(X-1)$
\end{enumerate}
\end{Lem}

\begin{proof}
Traitons par exemple le cas où $\mathbf{G} = Sp_{2n}$ est un groupe symplectique. Dans ce cas $\overline{\textsf{G}}_{x_{i}}=Sp_{2n_{i}}(\mathfrak{f})$ (avec $n_{1}+n_{2}=n$). On a $\textbf{\textsf{G}}^{*}=SO_{2n+1}$ et $\overline{\textbf{\textsf{G}}}_{x_i}^{*}=SO_{2n_i+1}$. Notons $\textsf{\textbf{T}}^{*}$ le tore déployé de $\textbf{\textsf{G}}^{*}$ et $\textsf{\textbf{T}}_{i}^{*}$ celui de $\overline{\textbf{\textsf{G}}}_{x_i}^{*}$. On peut alors considérer que $\tilde{s} \in \textbf{\textsf{G}}^{*}$, $s_{i} \in \textsf{\textbf{T}}_{i}^{*}$ et donc écrire
\[s_i=(1,a_1^{(i)},\cdots,a_{n_i}^{(i)},(a_{n_i}^{(i)})^{-1}, \cdots, (a_1^{(i)})^{-1}).\]
On obtient alors que 
\[\tilde{s}=(1,a_1^{(1)},\cdots,a_{n_1}^{(1)},a_1^{(2)},\cdots,a_{n_2}^{(2)},(a_{n_2}^{(2)})^{-1}, \cdots, (a_1^{(2)})^{-1},(a_{n_1}^{(1)})^{-1}, \cdots, (a_1^{(1)})^{-1})\]
d'où le résultat. On fait de même avec les autres cas.
\end{proof}

\begin{The}
\label{theCompatibiliteLanglands}
Supposons que $\mathbf{G}$ est un groupe classique non-ramifié, $k$ de caractéristique nulle et $p \neq 2$. Soient $\pi \in \text{Irr}_{\overline{\mathbb{Q}}_{\ell}}(G)$ une représentation de niveau 0 et $\phi \in \Phi_{m}(I_k,\mathbf{G})$ tel que $\pi \in Rep_{\overline{\mathbb{Q}}_{\ell}}^{\phi}(G)$. Notons $\varphi_{\pi}$ le paramètre de Langlands associé à $\pi$ via la correspondance de Langlands locale pour les groupes classiques. Alors $\varphi_{\pi | I_{k}} \sim \phi$.
\end{The}

\begin{proof}
La section \ref{secCompaInductionParabolique} permet de nous ramener au cas où $\pi$ est cuspidale. Le théorème \ref{thmLLC} nous dit alors $\tilde{\varphi}_{\pi}=\bigoplus_{(\rho,m) \in Jord(\pi)} \varphi_{\rho} \otimes st_{m}$ donc $\tilde{\varphi}_{\pi | I_{k}} =\bigoplus_{(\rho,m) \in Jord(\pi)} m \varphi_{\rho | I_{k}} $. Comme $\sum_{m,(\rho,m) \in Jord(\pi)} m = \lfloor s_{\pi}(\rho)^{2} \rfloor$ (où $\lfloor \cdot \rfloor$ désigne la partie entière inférieure) on obtient $\tilde{\varphi}_{\pi | I_{k}} =\bigoplus_{\rho \in \mathcal{A}^{\sigma}_{[0]}(k)} \lfloor s_{\pi}(\rho)^{2} \rfloor \varphi_{\rho | I_{k}}$. Notons $[\rho]$ la classe d'inertie de $\rho$ de sorte que $[\rho] \cap \mathcal{A}^{\sigma}(k') = \{ \rho,\rho'\}$. Les deux paramètres de Langlands $\varphi_{\rho}$ et $\varphi_{\rho'}$ ont la même restriction à l'inertie donc $\tilde{\varphi}_{\pi | I_{k}} =\bigoplus_{[\rho]} (\lfloor s_{\pi}(\rho)^{2} \rfloor + \lfloor s_{\pi}(\rho')^{2} \rfloor)  \varphi_{\rho | I_{k}}$. Notons $t \in (\textsf{\textbf{G}}^{*}_{ss})^{F}$ la classe de conjugaison semi-simple associée à $\varphi_{\pi | I_{k}}$.

Reprenons les notations précédentes, $\pi$ nous fournit un $x \in BT_{0}$ et un $s \in (\overline{\textsf{G}}_{x}^{*})_{ss}$ correspondant à $(s_{1},s_{2})$, $s_{i} \in (\overline{\textsf{G}}_{x,i}^{*})_{ss}$. On définit également $\tilde{s} \in (\textbf{\textsf{G}}^{*}_{ss})^{F}$ par $\tilde{s}:=\tilde{\psi}_{x}(s)$. Par définition $\tilde{s}$ est la classe de conjugaison semi-simple $F$-stable associée à $\phi$. Il nous faut donc montrer que $t=\tilde{s}$. Le lemme \ref{lemPolyCaract} nous dit qu'il suffit de montrer que $P_{t}=P_{\tilde{s}}$ où $P_{t}$ et $P_{\tilde{s}}$ sont les polynômes caractéristiques de $t$ et $\tilde{s}$.

Si $\rho \in \mathcal{A}_{[0]}^{\sigma}(k)$, on a vu qu'on pouvait lui associer une classe de conjugaison semi-simple $s_{\rho}$ et l'on note $Q$ son polynôme caractéristique. La compatibilité à Langlands dans le cas $\mathbf{G}=GL_{n}$, démontré dans \cite{dat_equivalences_2014} section 3.2.6, montre que $Q$ est bien le polynôme caractéristique de la classe de conjugaison semi-simple $F$-stable associée à $\varphi_{\rho \mid I_{k}}$.

Le théorème \ref{theCalculSpi} et le lemme \ref{lemPstilde} permettent alors de conclure. Traitons par exemple la cas $\mathbf{G}=Sp_{2n}$. Nous avons $\tilde{\varphi}_{\pi | I_{k}} =\bigoplus_{[\rho]} (\lfloor s_{\pi}(\rho)^{2} \rfloor + \lfloor s_{\pi}(\rho')^{2} \rfloor)  \varphi_{\rho | I_{k}}$. Le théorème \ref{theCalculSpi} nous donne alors $\lfloor s_{\pi}(\rho)^{2} \rfloor + \lfloor s_{\pi}(\rho')^{2} \rfloor=a_{Q}^{(1)}+a_{Q}^{(2)}$ si $Q(X) \neq X-1$ et $\lfloor s_{\pi}(\rho)^{2} \rfloor + \lfloor s_{\pi}(\rho')^{2} \rfloor=a_{(X-1)}^{(1)}+a_{(X-1)}^{(2)}-1$ si $Q(X) = X-1$. Donc 
\begin{align*}
P_{t}(X)&=(X-1)^{(a_{(X-1)}^{(1)}+a_{(X-1)}^{(2)}-1)}\prod_{Q(X) \neq X-1} Q(X)^{(a_{Q}^{(1)}+a_{Q}^{(2)})}\\
&=1/(X-1)\left(\prod_{Q} Q^{a_{Q}^{(1)}}\right)\left(\prod_{Q} Q^{a_{Q}^{(2)}}\right)\\
&=1/(X-1)P_{s_1}(X)P_{s_2}(X)
\end{align*}
Le lemme \ref{lemPstilde} montre alors que $P_{t}=P_{\tilde{s}}$ qui est le résultat recherché.
\end{proof}

Le théorème de décomposition de Bernstein fournit une partition des irréductibles $Irr_{\overline{\mathbb{Q}}_{\ell}}(G)=\sqcup_{\mathfrak{s} \in \mathcal{B}(G)} Irr_{\mathfrak{s}}(G)$, où $\mathcal{B}(G)$ désigne l'ensemble des classes d'inertie de données cuspidales. En supposant vraie la correspondance de Langlands locale, on obtient une autre partition $Irr_{\overline{\mathbb{Q}}_{\ell}}(G)=\sqcup_{\varphi \in \Phi(\mathbf{G})} \Pi_{\varphi}$, où $\Pi_{\varphi}$ est le $L$-paquet associé au paramètre $\varphi$. Haines introduit dans \cite{haines} la notion de "centre de Bernstein stable" qui permet de comparer ces deux décompositions. 

\bigskip

Soit $\lambda \in \Phi(W_k,\mathbf{G})$ et $\mathfrak{i}$ sa classe d'inertie (on rappelle la définition de l'équivalence inertielle de Haines en \ref{defEqInert}). On définit un paquet inertiel par
\[\Pi_{\mathfrak{i}}^{+} := \bigsqcup_{\underset{\varphi_{|W_{k}} \in \mathfrak{i}}{\varphi \in \Phi(\mathbf{G})}} \Pi_{\varphi}(G).\]

Supposons que l'on ait la correspondance de Langlands locale pour $G$ ainsi que ses sous-groupes de Levi, la compatibilité à l'induction parabolique et à certains isomorphismes (voir \cite{haines} définition 5.2.1 pour plus de détails). Alors on peut définir une application $\mathcal{L}_{i}$ qui à une classe d'inertie de données cuspidales $\mathfrak{s}=[L,\sigma] \in \mathcal{B}(G)$ associe $\mathcal{L}_{i}(\mathfrak{s})$ la classe d'inertie du paramètre de Weil $\varphi_{\sigma|W_{k}}$. Les paquets inertiels permettent de comparer les décompositions $Irr_{\overline{\mathbb{Q}}_{\ell}}(G)=\sqcup_{\mathfrak{s} \in \mathcal{B}(G)} Irr_{\mathfrak{s}}(G)=\sqcup_{\varphi \in \Phi(\mathbf{G})} \Pi_{\varphi}$ :

\begin{The}[\cite{moussaouiABPS} théorème 2.9] Soient $G$ un groupe classique déployé et $\mathfrak{i}$ une classe d'inertie de paramètres de Weil. Alors
\[\Pi_{\mathfrak{i}}^{+}(G)=\bigsqcup_{\underset{\mathcal{L}_{i}(\mathfrak{s})=\mathfrak{i}}{\mathfrak{s}\in\mathcal{B}(G)}} Irr_{\mathfrak{s}}(G)\]

\end{The}

Revenons à l'étude de $Rep_{\overline{\mathbb{Q}}_{\ell}}^{\phi}(G)$.

\begin{The}
\label{theBlocStable}
Soient $\mathbf{G}$ un groupe classique non-ramifié, $k$ de caractéristique nulle, $p \neq 2$ et $\phi \in \Phi_{m}(I_k,\mathbf{G})$ un paramètre inertiel modéré. Alors si $C_{\widehat{\mathbf{G}}}(\phi)$, le centralisateur de $\phi(I_{k})$ dans $\widehat{\mathbf{G}}$, est connexe
\[ Irr_{\overline{\mathbb{Q}}_{\ell}}(G) \cap Rep_{\overline{\mathbb{Q}}_{\ell}}^{\phi}(G) = \Pi_{\mathfrak{i}}^{+}(G)\]
où $\mathfrak{i}$ est la classe d'inertie formée des paramètres de Weil qui étendent $\phi$.

Dit autrement, $Rep_{\overline{\mathbb{Q}}_{\ell}}^{\phi}(G)$ est un "bloc stable" (c'est-à-dire, correspond à un idempotent primitif du centre de Bernstein stable au sens de Haines \cite{haines}).
\end{The}

\begin{proof}
La proposition \ref{proClasseInertie} montre que si $C_{\widehat{\mathbf{G}}}(\phi)$ est connexe, alors l'ensemble des $\lambda \in \Phi(W_k,\mathbf{G})$ tels que $\lambda_{|I_{k}} \sim \phi$ forme une classe d'équivalence inertielle. De plus, le théorème \ref{theCompatibiliteLanglands} montre que la construction de $Rep_{\overline{\mathbb{Q}}_{\ell}}^{\phi}$ est compatible avec la correspondance de Langlands, d'où le résultat.
\end{proof}

Lorsque $C_{\widehat{\mathbf{G}}}(\phi)$ n'est pas connexe, $Rep_{\overline{\mathbb{Q}}_{\ell}}^{\phi}(G)$ est une somme de "blocs stables". Nous expliquerons dans un prochain article comment décomposer naturellement la catégorie $Rep_{\overline{\mathbb{Q}}_{\ell}}^{\phi}(G)$ en "blocs stables".

\appendix

\section{Rappels sur le groupe dual}

\label{sectionRappelsGD}

Soient $\Omega$ un corps algébriquement clos et $\mathbf{G}$, $\mathbf{G}'$ deux groupes réductifs définis sur $\Omega$. Prenons $\varphi : \mathbf{G} \rightarrow \mathbf{G}'$ un isomorphisme. Choisissons $\textbf{T}$ un tore maximal de $\mathbf{G}$ et notons $\mathbf{T}' = \varphi(\mathbf{T})$ qui est un tore maximal de $\mathbf{G}'$. On associe à $\mathbf{G}$ et $\mathbf{T}$ une donnée radicielle $\phi(\mathbf{G},\mathbf{T})=(X,X^{\vee},\phi,\phi^{\vee})$ où $X=X^{*}(\mathbf{T})$ est le groupe des caractère de $\mathbf{T}$, $X^{\vee}=X_{*}(\mathbf{T})$ est le groupe des co-caractères, $\phi$ est l'ensemble des racines et $\phi^{\vee}$ l'ensemble des co-racines. À partir de $\phi(\mathbf{G},\mathbf{T})$ on forme la donnée radicielle duale $\widehat{\phi(\mathbf{G},\mathbf{T})}$ définie par $\widehat{\phi(\mathbf{G},\mathbf{T})}=(X^{\vee},X,\phi^{\vee},\phi)$. D'après \cite{springer} théorème 2.9 nous savons qu'il existe un groupe réductif $\widehat{\mathbf{G}}$ et un tore $\widehat{\mathbf{T}}$ tel que $\phi(\widehat{\mathbf{G}},\widehat{\mathbf{T}})=\widehat{\phi(\mathbf{G},\mathbf{T})}$. De façon analogue il existe $\widehat{\mathbf{G}}'$ et $\widehat{\mathbf{T}}'$ tel que $\phi(\widehat{\mathbf{G}}',\widehat{\mathbf{T}}')=\widehat{\phi(\mathbf{G}',\mathbf{T}')}$.

L'isomorphisme $\varphi : \mathbf{G} \rightarrow \mathbf{G}'$ induit un isomorphisme $f(\varphi):\phi(\mathbf{G}',\mathbf{T}') \rightarrow \phi(\mathbf{G},\mathbf{T})$. Nous avons alors aussi l'isogénie ${}^{t}f(\varphi):\widehat{\phi(\mathbf{G}',\mathbf{T}')} \rightarrow \widehat{\phi(\mathbf{G},\mathbf{T})}$, c'est à dire ${}^{t}f(\varphi):\phi(\widehat{\mathbf{G}}',\widehat{\mathbf{T}}') \rightarrow \phi(\widehat{\mathbf{G}},\widehat{\mathbf{T}})$. Le théorème 2.9 de \cite{springer} nous dit également qu'il existe un isomorphisme $\widehat{\varphi} :\widehat{\mathbf{G}}' \rightarrow \widehat{\mathbf{G}} $ qui envoie $\widehat{\mathbf{T}}'$ sur $\widehat{\mathbf{T}}$ et tel que $f(\widehat{\varphi})={}^{t}f(\varphi)$. Ce $\widehat{\varphi}$ n'est pas unique et deux tels $\widehat{\varphi}$ diffèrent par un automorphisme $Int(\hat{t}')$ où $\hat{t}' \in \widehat{\mathbf{T}}'$.

\bigskip

Nous savons que $\widehat{\mathbf{T}} \simeq X_{*}(\widehat{\mathbf{T}})\otimes_{\mathbb{Z}} \Omega^{\times} = X \otimes_{\mathbb{Z}} \Omega^{\times}$ et de même $\widehat{\mathbf{T}}' \simeq  X' \otimes_{\mathbb{Z}} \Omega^{\times}$. Par définition de $\widehat{\varphi}$, le morphisme $\widehat{\varphi}:\widehat{\mathbf{T}}' \rightarrow \widehat{\mathbf{T}}$ est donné par $f(\varphi) \otimes id : X' \otimes_{\mathbb{Z}} \Omega^{\times} \rightarrow X \otimes_{\mathbb{Z}} \Omega^{\times}$. Nous avons donc le diagramme commutatif suivant :

\[ \xymatrix{
X' \otimes_{\mathbb{Z}} \Omega^{\times} \ar@{->}[r] \ar@{->}[d]^{f(\varphi) \otimes id} & \widehat{\mathbf{G}}' \ar@{->}[d]^{\widehat{\varphi}}\\
X \otimes_{\mathbb{Z}} \Omega^{\times} \ar@{->}[r] & \widehat{\mathbf{G}}}
\]

$\widehat{\varphi}$ est défini à conjugaison intérieure près, ainsi si l'on passe aux classes de conjugaison celui ci est bien défini. De plus on sait que $\widehat{\mathbf{G}}_{ss} \simeq (\widehat{\mathbf{T}}/W)$ où $W$ est le groupe de Weyl de $\widehat{\mathbf{G}}$ relativement à $\widehat{\mathbf{T}}$. Ainsi on obtient

\begin{Lem}
\label{lemactiongroupedual}
Avec les notations précédentes on a un diagramme commutatif

\[ \xymatrix{
(X' \otimes_{\mathbb{Z}} \Omega^{\times})/W' \ar@{->}[r]^-{\sim} \ar@{->}[d]^{f(\varphi) \otimes id} & \widehat{\mathbf{G}}'_{ss} \ar@{->}[d]^{\widehat{\varphi}_{ss}}\\
(X \otimes_{\mathbb{Z}} \Omega^{\times})/W \ar@{->}[r]^-{\sim} & \widehat{\mathbf{G}}_{ss}}
\]
\end{Lem}

\bigskip

Si maintenant on prend $\mathbf{G}'=\mathbf{G}$ et $\varphi \in Aut(\mathbf{G})$. Nous pouvons identifier de manière canonique $\widehat{\mathbf{G}}_{ss}$ et $\widehat{\mathbf{G}}'_{ss}$ de la façon suivante. Il existe un $g \in \mathbf{G}$ tel que $\mathbf{T}'=Ad(g)(\mathbf{T})$. L'automorphisme $Ad(g)$ induit donc un isomorphisme $f$ de $X'$ dans $X$ et un isomorphisme canonique
\[ f \otimes id : (X' \otimes_{\mathbb{Z}} \Omega^{\times})/W' \longrightarrow (X \otimes_{\mathbb{Z}} \Omega^{\times})/W\]
(comme on a quotienté par le groupe de Weyl, $f \otimes id$ ne dépend pas du choix de $g$).

Le lemme \ref{lemactiongroupedual} nous fournit un isomorphisme canonique $\widehat{f}_{ss}$ entre $\widehat{\mathbf{G}}_{ss}$ et $\widehat{\mathbf{G}}'_{ss}$. Via cette identification, l'automorphisme $\varphi$ donne lieu à un isomorphisme $\widehat{\varphi}_{ss} \in Aut(\widehat{\mathbf{G}}_{ss})$. Ce dernier est défini par le diagramme commutatif suivant

\[ \xymatrix{
(X \otimes_{\mathbb{Z}} \Omega^{\times})/W \ar@{->}[rr]^-{\sim} \ar@{->}[d]^{f(\varphi) \otimes id} & &\widehat{\mathbf{G}}_{ss} \ar@{->}[d]^{\widehat{\varphi}_{ss}}\\
(X' \otimes_{\mathbb{Z}} \Omega^{\times})/W' \ar@{->}[r]^{f \otimes id} &  (X \otimes_{\mathbb{Z}} \Omega^{\times})/W \ar@{->}[r]^-{\sim}& \widehat{\mathbf{G}}_{ss}}
\]

Il correspond donc à l'automorphisme de $X$ : $f \circ f(\varphi)$. Une autre façon de voir les choses est la suivante : 

Fixons un épinglage $(\mathbf{G},\mathbf{B},\mathbf{T},\{u_{\alpha}\}_{\alpha \in \Delta})$. On a alors une suite exacte scindée
\[ \{1\} \longrightarrow Int(\mathbf{G}) \longrightarrow Aut(\mathbf{G}) \longrightarrow Out(\mathbf{G}) \longrightarrow \{1\}\]
L'automorphisme $f \circ f(\varphi)$ correspond à l'image de $\varphi$ par l'application $Aut(\mathbf{G}) \longrightarrow Out(\mathbf{G})$. Son image par $Out(\mathbf{G}) \rightarrow Out(\widehat{\mathbf{G}})$ est ${}^{t}(f \circ f(\varphi))$. On peut donc prendre pour $\widehat{\varphi}$ l'image de ${}^{t}(f \circ f(\varphi))$ par $Out(\widehat{\mathbf{G}}) \rightarrow Aut(\widehat{\mathbf{G}})$. L'automorphisme $\widehat{\varphi}_{ss}$ recherché est alors l'application induite par $\widehat{\varphi}$ sur les classes de conjugaison semi-simples.

\begin{Lem}
\label{lemconjugaisondual}
On a un diagramme commutatif 
\[ \xymatrix{
(X \otimes_{\mathbb{Z}} \Omega^{\times})/W \ar@{->}[r]^-{\sim} \ar@{->}[d]^{f(\varphi) \otimes id} & \widehat{\mathbf{G}}_{ss} \ar@{->}[d]^{\widehat{\varphi}_{ss}}\\
(X' \otimes_{\mathbb{Z}} \Omega^{\times})/W \ar@{->}[r]^-{\sim} & \widehat{\mathbf{G}}_{ss}}
\]
où $\widehat{\varphi}$ est l'image de $\varphi$ par l'application $Aut(\mathbf{G}) \rightarrow Out(\mathbf{G}) \rightarrow Out(\widehat{\mathbf{G}}) \rightarrow Aut(\widehat{\mathbf{G}})$. En particulier, si $\varphi \in Int(\mathbf{G})$, $\widehat{\varphi}_{ss}=id$.
\end{Lem}

\section{Rappels sur le Frobenius}

\label{secFrobenius}

Soit $\textsf{\textbf{G}}$ un groupe réductif connexe défini sur $\mathfrak{F}$. Une structure sur $\mathfrak{f}$ donne lieu à un endomorphisme de Frobenius noté $F:\textsf{\textbf{G}} \rightarrow \textsf{\textbf{G}}$, tel que les $\mathfrak{f}$-points de $\textsf{\textbf{G}}$ soient $\textsf{G}:=\textsf{\textbf{G}}^{F}=\textsf{\textbf{G}}(\mathfrak{f})$. Celui-ci est défini de la manière suivante :

$\textsf{\textbf{G}}$ possède une $\mathfrak{f}$-structure si et seulement si son algèbre des fonctions, que l'on note $A$, vérifie $A=A_{0} \otimes_{\mathfrak{f}} \mathfrak{F}$ où $A_{0}$ est une $\mathfrak{f}$-algèbre. L'endomorphisme de Frobenius $F:\textsf{\textbf{G}} \rightarrow \textsf{\textbf{G}}$ est alors défini par son co-morphisme $F^{*} \in End(A_{0} \otimes_{\mathfrak{f}} \mathfrak{F})$ qui a $x \otimes \lambda$ associe $x^{q} \otimes \lambda$.

Le groupe de Galois $Gal(\mathfrak{F} / \mathfrak{f})$ agit également sur $A_{0} \otimes_{\mathfrak{f}} \mathfrak{F}$ par $x\otimes \lambda \mapsto x \otimes \lambda^{q}$ et donc induit une action que l'on note $\tau$ sur $\textsf{\textbf{G}}$.

\bigskip

Nous avons besoin de comprendre un peu plus en détail comment ces deux actions agissent lorsque $\textsf{\textbf{G}}=\textsf{\textbf{T}}$ est un tore défini sur $\mathfrak{f}$. Notons $X=X^{*}(\textsf{\textbf{T}})=Hom(\textsf{\textbf{T}},\mathbb{G}_{m})$ le groupe des caractères de $\textsf{\textbf{T}}$. On a $X\simeq Hom_{alg-Hopf}(\mathfrak{F}[t,t^{-1}],A)$ et donc $\tau$ et $F$ se prolongent en des actions sur $X$. Notons $\tau_{X}$ l'action de $\tau$ induit sur $X$.

Soit $\alpha \in X$ et calculons $\tau_{X}\cdot F \cdot \alpha$.
$$\begin{array}{ccccccccc}
\tau_{X}\cdot F \cdot \alpha & : & \mathfrak{F}[t,t^{-1}] & \overset{\alpha}{\longrightarrow} & A_{0} \otimes_{\mathfrak{f}} \mathfrak{F}& \overset{F}{\longrightarrow} & A_{0} \otimes_{\mathfrak{f}} \mathfrak{F}& \overset{\tau}{\longrightarrow} & A_{0} \otimes_{\mathfrak{f}} \mathfrak{F}\\
 & &  &  & x\otimes \lambda & \mapsto & x^{q} \otimes \lambda  & \mapsto & x^{q} \otimes \lambda^{q}\\
\end{array}$$
On voit donc que $\tau_{X} \cdot F \cdot \alpha = \alpha^{q}$. Notons $\psi$ l'élévation à la puissance $q$. Comme $\textsf{\textbf{T}}\simeq Hom(X,\mathfrak{F}^{\times})$, ces actions se transfèrent en des actions sur $\textsf{\textbf{T}}$. A quoi correspond $\psi$? Prenons $u:X \to \mathfrak{F}^{\times}$ alors 
$$\begin{array}{ccccccc}
\psi \cdot u & : & X & \overset{\psi}{\longrightarrow} & X & \overset{u}{\longrightarrow} & \mathfrak{F}^{\times}\\
 & & \alpha & \mapsto & \alpha^{q} & \mapsto & u(\alpha^{q})=u(\alpha)^{q}\\
\end{array}$$

Donc $\psi$ agit comme l'élévation à la puissance $q$ sur $\textsf{\textbf{T}}$. Finalement on trouve que l'action de $F$ sur $\textsf{\textbf{T}}$ est donnée par $F=\tau_{X}^{-1} \circ \psi$.

\section{Équivalence inertielle pour les paramètres de Langlands}

\label{secEqInert}

Nous rappelons dans cette section la définition de classe inertielle introduite par Haines dans \cite{haines}. Nous montrons également une proposition, utile pour le théorème \ref{theBlocStable}, qui relie une classe inertielle aux paramètres de l'inertie.

Soit $\mathbf{G}$ un groupe réductif connexe défini sur $k$ un corps $p$-adique. On considère ici des représentations à coefficients complexes.

\bigskip

Commençons par rappeler la notion de parabolique standard et de Levi standard de ${}^{L}\mathbf{G}$, définis dans \cite{borel} paragraphes 3.3 et 3.4. Fixons des données $\widehat{\mathbf{T}}_{0}\subseteq \widehat{\mathbf{B}}_{0} \subseteq \widehat{\mathbf{G}}$, composées d'un tore maximal et d'un Borel, stables sous l'action du groupe de Galois. On dit alors qu'un sous-groupe parabolique $\mathcal{P}$ de ${}^{L}\mathbf{G}$ est standard si $\mathcal{P} \supseteq {}^{L}\mathbf{B}_{0}$. Sa composante neutre $\mathcal{P}^{\circ}:=\mathcal{P} \cap \widehat{\mathbf{G}}$ est alors un sous-groupe parabolique standard de $\widehat{\mathbf{G}}$ contenant $\widehat{\mathbf{B}}_{0}$ et on a $\mathcal{P}=\mathcal{P}^{\circ} \rtimes W_k$. Soit $\mathcal{M}^{\circ}$ l'unique Levi de $\mathcal{P}^{\circ}$ contenant $\widehat{\mathbf{T}}_{0}$. Alors $\mathcal{M}:=N_{\mathcal{P}}(\mathcal{M}^{\circ})$ est un sous-groupe de Levi de $\mathcal{P}$ et $\mathcal{M}=\mathcal{M}^{\circ} \rtimes W_k$. Les sous-groupes de Levi de ${}^{L}\mathbf{G}$ construits de cette manière sont appelés standards. Tout sous-groupe de Levi de ${}^{L}\mathbf{G}$ est $\widehat{\mathbf{G}}$-conjugué à un Levi standard, et pour $\mathcal{M}$ un Levi standard de ${}^{L}\mathbf{G}$, on note $\{\mathcal{M}\}$ l'ensemble des sous-groupes de Levi standards qui sont $\widehat{\mathbf{G}}$-conjugués à $\mathcal{M}$.

Soit $\lambda : W_{k} \rightarrow {}^{L}\mathbf{G}$ un morphisme admissible. L'image de $\lambda$ est alors contenue dans un Levi minimal de ${}^{L}\mathbf{G}$, bien défini à conjugaison par un élément de $C_{\widehat{\mathbf{G}}}(\lambda)^{\circ}$, où $C_{\widehat{\mathbf{G}}}(\lambda)$ désigne le centralisateur de $\lambda(W_{k})$ dans $\widehat{\mathbf{G}}$ (voir \cite{borel} proposition 3.6). Notons $(\lambda)_{\widehat{\mathbf{G}}}$ la classe de $\widehat{\mathbf{G}}$-conjugaison de $\lambda$. Alors $\lambda$ donne lieu à une unique classe de Levi standards $\{\mathcal{M}_{\lambda}\}$ telle qu'il existe $\lambda^{+} \in (\lambda)_{\widehat{\mathbf{G}}}$ dont l'image est contenue minimalement dans $\mathcal{M}_{\lambda}$, pour un $\mathcal{M}_{\lambda}$ dans cette classe.

\begin{Def}[\cite{haines} définition 5.3.3]
\label{defEqInert}
Soient $\lambda_{1},\lambda_{2} : W_{k} \rightarrow {}^{L}\mathbf{G}$ deux paramètres admissibles. On dit que $\lambda_{1}$ et $\lambda_{2}$ sont \textit{inertiellement équivalents} si
\begin{enumerate}
\item $\{\mathcal{M}_{\lambda_{1}}\}=\{\mathcal{M}_{\lambda_{2}}\}$
\item il existe $\mathcal{M} \in \{\mathcal{M}_{\lambda_{1}}\}$, $\lambda_{1}^{+} \in (\lambda_{1})_{\widehat{\mathbf{G}}}$ et $\lambda_{2}^{+} \in (\lambda_{2})_{\widehat{\mathbf{G}}}$ dont les images sont minimalement contenues dans $\mathcal{M}$, et $z \in H^{1}(\langle \widehat{\vartheta} \rangle, (Z(\mathcal{M}^{\circ})^{I_k})^{\circ})$ vérifiant
\[ (z\lambda_{1}^{+})_{\mathcal{M}^{\circ}}=(\lambda_{2}^{+})_{\mathcal{M}^{\circ}}\]
\end{enumerate}
\end{Def}

\begin{Pro}
\label{proClasseInertie}
Soit $\phi \in \Phi(I_k,\mathbf{G})$ tel que $C_{\widehat{\mathbf{G}}}(\phi)$, le centralisateur de $\phi(I_{k})$ dans $\widehat{\mathbf{G}}$, est connexe. Alors l'ensemble des $\lambda \in \Phi(W_k,\mathbf{G})$ tels que $\lambda_{|I_{k}} \sim \phi$ forme une classe d'équivalence inertielle.
\end{Pro}

\begin{proof}
Remarquons tout d'abord que deux paramètres admissibles de $W_{k}$ inertiellement équivalents ont des restrictions à l'inertie conjuguées. Prenons donc deux paramètres admissibles $\lambda_{1},\lambda_{2} : W_{k} \rightarrow {}^{L}\mathbf{G}$ tels que $\lambda_{1|I_{k}} \sim \lambda_{2|I_{k}} \sim \phi$ et montrons qu'ils sont inertiellement équivalents. Quitte à conjuguer par $\widehat{\mathbf{G}}$, on peut supposer que $\lambda_{1|I_{k}} = \lambda_{2|I_{k}} = \phi$. Fixons ($\widehat{\mathbf{T}},\widehat{\mathbf{B}})$ une paire constituée d'un tore maximal et d'un Borel de $C_{\widehat{\mathbf{G}}}(\phi)$. 

Comme $\lambda_{1|I_{k}} = \phi$, pour $w \in W_{k}$, $\lambda_{1}(w)$ normalise $\phi(I_{k})$ donc normalise également $C_{\widehat{\mathbf{G}}}(\phi)$. En faisant agir $w$ par conjugaison par $\lambda_{1}(w)$ on obtient une action $Ad_{\lambda_{1}}:W_{k}/I_{k} \rightarrow Aut(C_{\widehat{\mathbf{G}}}(\phi))$. La conjugaison par $\lambda_{1}(\text{Frob})$, $Ad_{\lambda_{1}}(\text{Frob})$ est donc un automorphisme semi-simple de $C_{\widehat{\mathbf{G}}}(\phi)$, que l'on notera $\theta_{\lambda_1}$. Par le théorème 7.5 de \cite{steinberg} il existe une paire $(\mathbf{T}_{1},\mathbf{B}_{1})$ constituée d'un tore maximal et d'un Borel de $C_{\widehat{\mathbf{G}}}(\phi)$ tous les deux $\theta_{\lambda_1}$-stables. Quitte à conjuguer $\lambda_{1}$ par un élément de $C_{\widehat{\mathbf{G}}}(\phi)$, on peut supposer que $(\mathbf{T}_{1},\mathbf{B}_{1})=(\widehat{\mathbf{T}},\widehat{\mathbf{B}})$, c'est à dire que $\widehat{\mathbf{T}}$ et $\widehat{\mathbf{B}}$ sont $\theta_{\lambda_1}$-stables. On fait de même pour $\lambda_{2}$ et donc $\widehat{\mathbf{T}}$ et $\widehat{\mathbf{B}}$ sont également $\theta_{\lambda_2}$-stables.

Posons $\mathcal{M}_{\lambda_1}:=C_{{}^{L}\mathbf{G}}([\widehat{\mathbf{T}}^{\theta_{\lambda_1}}]^{\circ})$. Comme $(\widehat{\mathbf{T}}^{\theta_{\lambda_1}})^{\circ}$ est un tore maximal de $C_{\widehat{\mathbf{G}}}(\lambda_{1})^{\circ}$ (\cite{DigneMichelgroupe} théorème 1.8 (iii)), $\mathcal{M}_{\lambda_{1}}$ est un Levi minimal contenant l'image de $\lambda_{1}$ (\cite{borel} proposition 3.6). Adoptons les même notations pour $\lambda_2$.

Écrivons pour $w \in W_{k}$, $\lambda_{2}(w)=\eta(w)\lambda_{1}(w)$ avec $\eta(w) \in \widehat{\mathbf{G}}$. Alors $\eta$ est un cocycle à valeurs dans $C_{\widehat{\mathbf{G}}}(\phi)$ pour l'action $Ad_{\lambda_{1}}$, c'est-à-dire $\eta \in Z^{1}_{Ad_{\lambda_{1}}}(W_{k}/I_{k}, C_{\widehat{\mathbf{G}}}(\phi))$. Comme la paire $(\widehat{\mathbf{T}},\widehat{\mathbf{B}})$ est à la fois $\theta_{\lambda_1}$-stable et $\theta_{\lambda_2}$-stable, $\eta(\text{Frob})$ est dans le normalisateur dans $C_{\widehat{\mathbf{G}}}(\phi)$ de la paire $(\widehat{\mathbf{T}},\widehat{\mathbf{B}})$ qui est égal à $\widehat{\mathbf{T}}$. Puisque $\eta(\text{Frob}) \in \widehat{\mathbf{T}}$, on a en particulier que $\widehat{\mathbf{T}}^{\theta_{\lambda_1}} = \widehat{\mathbf{T}}^{\theta_{\lambda_2}}$ et donc $\mathcal{M}_{\lambda_{1}}=\mathcal{M}_{\lambda_{2}}$. Notons maintenant $\mathcal{M}:=\mathcal{M}_{\lambda_{1}}=\mathcal{M}_{\lambda_{2}}$.

La suite exacte $1 \rightarrow \mathcal{M}^{\circ} \rightarrow \mathcal{M} \rightarrow \langle \widehat{\vartheta} \rangle \rightarrow 1$ et le choix d'une section permettent de définir une action de $\langle \widehat{\vartheta} \rangle$ sur $\mathcal{M}^{\circ}$. Par conjugaison, elle induit une action de $\langle \widehat{\vartheta} \rangle$ sur $Z(\mathcal{M}^{\circ})$ qui est alors indépendante du choix de la section. On a donc une action canonique de $\langle \widehat{\vartheta} \rangle$ sur $Z(\mathcal{M}^{\circ})$.

On aurait pu conjuguer $\phi$ dès le départ pour avoir $\mathcal{M}$ standard. Il nous suffit donc pour montrer que $\lambda_1$ et $\lambda_2$ sont inertiellement équivalents de montrer l'existence d'un $z \in H^{1}(\langle \widehat{\vartheta} \rangle, (Z(\mathcal{M}^{\circ})^{I_k})^{\circ})$ tel que $(z\lambda_{1})_{\mathcal{M}^{\circ}}=(\lambda_{2})_{\mathcal{M}^{\circ}}$. Soit $t \in \widehat{\mathbf{T}}$, alors ${}^{t}\lambda_{2}=\eta'\lambda_{1}$ avec $\eta'=t^{-1}\eta Ad_{\lambda_1}(t)$. Ainsi, pour finir la preuve de la proposition, il nous suffit de montrer qu'il existe $t \in \widehat{\mathbf{T}}$ tel que $t^{-1}\eta(\text{Frob}) \theta_{\lambda_1}(t) \in (Z(\mathcal{M}^{\circ})^{I_k})^{\circ}$. Or comme $(\widehat{\mathbf{T}}^{\theta_{\lambda_1}})^{\circ} \subseteq (Z(\mathcal{M}^{\circ})^{I_k})^{\circ}$ cela découle du lemme \ref{lemDecomposTores} ci-dessous. Pour pouvoir appliquer ce dernier, il nous reste à vérifier que $\theta_{\lambda_1}$ induit un automorphisme d'ordre fini sur $X_{*}(\widehat{\mathbf{T}})$. Par dualité, il suffit de vérifier que $\theta_{\lambda_1}$ a une action d'ordre fini sur le groupe des caractère $X^{*}(\widehat{\mathbf{T}})$. Pour cela nous allons montrer la finitude de l'action de $\theta_{\lambda_1}$ sur $Q$, le réseau engendré par les racines, puis sur le centre $Z(C_{\widehat{\mathbf{G}}}(\phi))$. Comme le groupe des caractères de $Z(C_{\widehat{\mathbf{G}}}(\phi))$ est $X^{*}(\widehat{\mathbf{T}})/Q$ (\cite{springer} 2.15), on en déduit la finitude de l'action sur $X^{*}(\widehat{\mathbf{T}})$.

L'automorphisme $\theta_{\lambda_1}$ de $C_{\widehat{\mathbf{G}}}(\phi)$ stabilise $\widehat{\mathbf{T}}$, donc agit sur l'ensemble des racines de $C_{\widehat{\mathbf{G}}}(\phi)$ par rapport à $\widehat{\mathbf{T}}$. Cet ensemble étant fini, cette action est également d'ordre fini. Il nous reste à étudier l'action de $\theta_{\lambda_1}$ sur $Z(C_{\widehat{\mathbf{G}}}(\phi))$ le centre de $C_{\widehat{\mathbf{G}}}(\phi)$. D'après \cite{datFunctoriality} lemme 2.1.1, il existe $n \in \mathbb{N}^{*}$ et $\lambda : W_{k} \rightarrow {}^{L}\mathbf{G}$ une extension de $\phi$ telle que $\lambda(n\text{Frob})=(1,n\widehat{\vartheta})$. En particulier $\theta_{\lambda}$ a une action d'ordre fini sur $Z(C_{\widehat{\mathbf{G}}}(\phi))$. Or nous avons vu que deux paramètres de $W_{k}$ qui étendent $\phi$ diffèrent par un cocycle à valeur dans $C_{\widehat{\mathbf{G}}}(\phi)$, donc l'action sur le centre $Z(C_{\widehat{\mathbf{G}}}(\phi))$ est indépendante du choix de l'extension de $\phi$. Par conséquent $\theta_{\lambda_1}$ a une action d'ordre fini sur le centre et on a le résultat.

\end{proof}

\begin{Lem}
\label{lemDecomposTores}
Soient $\mathbf{T}$ un tore sur $\mathbb{C}$ et $\theta \in Aut(\mathbf{T})$. On note $X_{*}:=X_{*}(\mathbf{T})$ l'ensemble des co-caractères et on suppose que $\theta$ induit un automorphisme d'ordre fini sur $X_{*}$. Posons $L_{\theta}$ l'application $L_{\theta}:\mathbf{T} \rightarrow \mathbf{T}$, $t\mapsto t^{-1}\theta(t)$. Alors $\mathbf{T}=L_{\theta}(\mathbf{T}) \cdot (\mathbf{T}^{\theta})^{\circ}$.
\end{Lem}

\begin{proof}
Le groupe des co-caractères de $(\mathbf{T}^{\theta})^{\circ}$ vaut $X_{*}((\mathbf{T}^{\theta})^{\circ})=X_{*}^{\theta}$.
On définit $L_{\theta}^{X}:X_{*} \rightarrow X_{*}$ par $L_{\theta}^{X}(\lambda)=\lambda-\theta(\lambda)$. Alors $X_{*}(L_{\theta}(\mathbf{T})) \supseteq Im(L_{\theta}^{X})$ .

Posons $X_{*,\mathbb{Q}}:=X_{*} \otimes \mathbb{Q}$ et de même $L_{\theta,\mathbb{Q}}^{X}:X_{*,\mathbb{Q}} \rightarrow X_{*,\mathbb{Q}}$. Comme $\theta$ est d'ordre fini, $\theta$ est un endormorphisme semi-simple du $\mathbb{Q}$-espace vectoriel $X_{*,\mathbb{Q}}$. Ainsi, on a la décomposition $X_{*,\mathbb{Q}}=X_{*,\mathbb{Q}}^{\theta} \oplus Im(L_{\theta,\mathbb{Q}}^{X})$. On en déduit que $X_{*}^{\theta}+Im(L_{\theta}^{X})$ est d'indice fini dans $X_{*}$ et donc, comme $\mathbb{C}^{*}$ est divisible, que l'application $(X_{*}^{\theta} \otimes \mathbb{C}^{*}) \times (Im(L_{\theta}^{X})\otimes \mathbb{C}^{*})\rightarrow X_{*} \otimes \mathbb{C}^{*}$ est surjective. En particulier l'application $(X_{*}((\mathbf{T}^{\theta})^{\circ}) \otimes \mathbb{C}^{*}) \times (X_{*}(L_{\theta}(\mathbf{T}))\otimes \mathbb{C}^{*})\rightarrow X_{*} \otimes \mathbb{C}^{*}$ est surjective, d'où le résultat.
\end{proof}

\bibliographystyle{hep}
\bibliography{biblio}
\end{document}